\documentclass[english,10pt,a4paper]{article}
\usepackage{amsmath,amstext, amsthm, amssymb}
\usepackage[hypertex]{hyperref}

\usepackage[totalheight=23 true cm, totalwidth=15 true cm]{geometry}
\usepackage[english]{babel}
\usepackage[all]{xy}
\usepackage{color}
\usepackage{tocvsec2}

\newtheorem{thm}{Theorem}[section]
\newtheorem{cor}[thm]{Corollary}
\newtheorem{lemma}[thm]{Lemma}
\newtheorem{prop}[thm]{Proposition}

\newtheorem{defn}[thm]{Definition}
\newtheorem{defi}[thm]{Definition} 

\theoremstyle{remark}
\theoremstyle{definition}

\newtheorem{rem}[thm]{Remark} 

\newtheorem{exa}[thm]{Example}
\newtheorem{ex}[thm]{Example} 

\numberwithin{equation}{section}

\def\beq{\begin{equation}}
\def\eeq{\end{equation}}

\def\crash#1{}

\def\N{{\mathbb N}}
\def\Z{{\mathbb Z}}

\def\C{{\mathbb C}}

\def\I{{\mathbb I}}

\def\F{{\mathbb F}}

\def\l{\left}
\def\r{\right}

\def\ie{i.e.~}

\def\cE{{\mathcal E}}

\def\cM{{\mathcal M}}

\def\veps{\varepsilon}
\def\a{\alpha}

\def\de{\delta}

\def\s{\sigma}
\def\V{\mathbb{V}}
\def\A{\mathbb{A}}
\def\Hom{\operatorname{Hom}}
\def\Gl{\operatorname{Gl}}
\def\Sl{\operatorname{Sl}}
\def\G{\mathcal{G}}
\def\Gal{\operatorname{Gal}}
\def\Aut{\operatorname{Aut}}
\def\ida{\mathfrak{a}}
\def\idb{\mathfrak{b}}
\def\Alg{\operatorname{Alg}}
\def\spec{{\operatorname{Spec}}}
\def\q{\mathfrak{q}}
\def\m{\mathfrak{m}}
\def\quot{\operatorname{Quot}}
\def\sdim{\sigma\text{-}\dim}
\def\strdeg{\sigma\text{-}\operatorname{trdeg}}

\def\Gm{\bold{G}_m}

\def\pa{\delta}

\def\<{\langle}
\def\>{\rangle}

\def\ds{{\delta\sigma}}

\def\sgal{\sigma\text{-}\operatorname{Gal}}

\author{Lucia Di Vizio, Charlotte Hardouin and Michael Wibmer}

\title{Difference Galois theory of linear differential equations
\footnotetext{Lucia Di Vizio,
Laboratoire de Math\'ematiques UMR8100, UVSQ,
45 avenue des \'Etats-Unis
78035 Versailles cedex, France.
{e-mail: {\tt divizio@math.cnrs.fr}.}}
\footnotetext{Charlotte Hardouin, {Institut de Math\'{e}matiques de Toulouse,}
{118 route de Narbonne,
31062 Toulouse Cedex 9, France.}
{e-mail: {\tt hardouin@math.ups-tlse.fr}}.}
\footnotetext{Michael Wibmer, {RWTH Aachen}, {Templergraben 64,
52062 Aachen, Germany.} {e-mail: {\tt michael.wibmer@matha.rwth-aachen.de}.}}
\footnotetext{Work supported by project ANR-2010-JCJC-0105 01 qDIFF.}}

\begin{document}
\maketitle
\bibliographystyle{alpha}


\begin{abstract}
We develop a Galois theory for linear differential equations equipped with the action of an endomorphism. This theory is aimed at studying the difference algebraic relations among the solutions of a linear differential equation. The Galois groups here are linear difference algebraic groups, i.e., matrix groups defined by algebraic difference equations.
\end{abstract}

\renewcommand{\labelenumi}{{\rm (\roman{enumi})}}
\renewcommand{\labelenumii}{{\rm (\alph{enumii})}}

\maxtocdepth{section}


\section*{Introduction}
Linear differential equations with coefficients in a differential field $(K,\de)$ and their
behavior under the action of an endomorphism $\s$ of $K$ are a frequent object of study.
Let us start with some classical examples.
For instance, one can consider the field $K=\C(\a,x)$ of rational functions in the variables $\a,x$ and equip $K$ with the derivation $\de=\frac{d}{dx}$ and
the endomorphism $\s:f(\a,x)\mapsto f(\a+1,x)$.
The Bessel function $J_\alpha(x)$, which solves Bessel's differential equation
$$x^2\de^2(y)+x\de(y)+(x^2-\alpha^2)y=0$$
satisfies the linear difference equation
$$xJ_{\alpha+2}(x)-2(\alpha+1)J_{\alpha+1}(x)+xJ_{\a}(x)=0.$$
Contiguity relations for hypergeometric series provide a large class of examples in a similar spirit. (See for instance \cite[Chapter XIV]{WW}.)

Another occasion, where a linear differential equation comes naturally equipped with the action of an endomorphism arises in the $p$-adic analysis of linear differential equations, when considering Frobenius lifts. For example,
let $p$ be a prime number and let us consider the field $\C_p$ with
its norm $|~|$, such that $|p|=p^{-1}$, and an element $\pi\in\C_p$ verifying $\pi^{p-1}=-p$.
Following \cite[Chapter II, \S 6]{DGS} the series
$\theta(x)\in\C_p[[x]]$, defined by $\theta(x)=\exp(\pi(x^p-x))$,
has a radius of convergence bigger than $1$. 
Therefore it belongs to the field $\cE^\dag_{\C_p}$, consisting of all series $\sum_{n\in\Z}a_nx^n$ with $a_n\in\C_p$ such that
\begin{itemize}
 \item $\exists\ \veps>0$ such that $\forall\ \rho\in]1,1+\veps[$ we have
$\lim_{n\to\pm\infty}|a_n|\rho^n=0$ and
\item $\sup_n|a_n|$ is bounded.
\end{itemize}
One can endow $\cE^\dag_{\C_p}$ with an endomorphism $\s\colon\sum_{n\in\Z}a_nx^n\mapsto \sum_{n\in\Z}a_nx^{pn}$. (For the sake of simplicity we assume here that $\s$ is $\C_p$-linear.)
The solution $\exp(\pi x)$ of the equation $\de(y)=\pi y$, where $\de=\frac{d}{dx}$, does not
belong to $\cE^\dag_{\C_p}$, since it has radius of convergence $1$. Moreover, $\exp(\pi x)$ is a solution of an order one linear difference
equation with coefficients in $\cE^\dag_{\C_p}$, namely:
$$
\s(y)=\theta(x) y.
$$
So, here is another very classical situation in which
one considers solutions of a linear differential equation and finds difference algebraic relations among them. (Coincidentally, in the above two examples the difference algebraic relations are linear.)

Understanding the relations among solutions of an equation is a question which is very much in the spirit of Galois theory.
In this article we introduce a Galois theory which is able to handle linear differential equations in situations like the ones
described above. More precisely, we develop a Galois theory which deals with the difference algebraic relations among solutions of linear differential equations. The Galois groups here are linear difference algebraic groups, i.e., matrix groups defined by algebraic difference equations.

\medskip

Galois theories for various types of equations have become available over the years. The classical Galois theory of linear differential (or difference) equations, also known as Picard-Vessiot theory, deals with the algebraic relations among the solutions of linear differential (or difference) equation. See \cite{vdPutSingerDifferential} (respectively \cite{vdPutSingerDifference}) for an introduction. The Galois groups in these theories are linear algebraic groups.

Differential algebraic relations among solutions of linear differential equations are addressed by the Galois theory in \cite{cassisinger}, while \cite{HardouinSinger} also addresses differential algebraic relations among solutions of linear difference equations. In these theories the Galois groups are linear differential algebraic groups (\cite{Cassidy:differentialalgebraicgroups}). The setting in \cite{Landesman:GeneralizedDifferentialGaloisTheory} encompasses the setting in \cite{cassisinger}, in a similar vein as Kolchin's strongly normal Galois theory encompasses Picard-Vessiot theory. In \cite{Landesman:GeneralizedDifferentialGaloisTheory} the Galois groups are differential algebraic groups, but they need not be linear.

A first step towards a Galois theory that can handle difference algebraic relations among solutions of linear difference equations has been made in \cite{AntieauOvchinnikovTrushin:PeriodicParameters}. The endomorphisms in \cite{AntieauOvchinnikovTrushin:PeriodicParameters} however, are required to be of finite order, so that classical examples, like the shift $\alpha\mapsto \alpha+1$ or the Frobenius operator considered in the $p$-adic example above, remain out of reach. An approach to difference algebraic relations among solutions of linear difference equations, in the spirit of this article can be found in \cite{OvchinnikovWibmer:SGaloisTheoryOfLinearDifferenceEquations}.


The above described theories have been applied in various areas, for example in questions of integrability and isomonodromy (\cite{MitschiSinger:MonodromyGroupsOfParameterizedLinearDifferentialEquationsWithRegularSingularities}, \cite{MitchSinger:ProjectiveIsomonodromyAndGaloisGroups}, \cite{Dreyfus:density}) or in combinatorial problems (\cite{ChenSinger:ResiduesAndTelescopersForBivariateRationalFunctions}).
A prototypical application is to show that certain special functions are independent in a strong sense. For example, \cite{HardouinSinger} provides a Galoisian proof of H\"{o}lders theorem, stating that the Gamma function, which satisfies the linear difference equation $\Gamma(x+1)=x\Gamma(x)$, does not satisfy an algebraic differential equation over $\C(x)$. We refer to \cite{DiVizio:ApprocheGaloisienneDeLaTranscendanceDifferentielle} for an overview of the Galoisian approach to differential transcendence of special functions.

The Galois theories in \cite{cassisinger}, \cite{HardouinSinger} and \cite{Landesman:GeneralizedDifferentialGaloisTheory} study differential algebraic relations with respect to a finite number of commuting derivations, i.e., partial differential equations are considered. In this article we only study difference algebraic relations with respect to one endomorphism, i.e., only ordinary difference algebraic equations are considered. Certain aspects of our theory, e.g. the Galois correspondence, will generalize to the partial case in a straight forward manner. However, other aspects, e.g. existence and uniqueness of Picard-Vessiot extensions appear to be more challenging.

Given the well-known analogy between difference and differential algebra, it might at first sight seem a rather straight forward matter to pass from differential algebraic relations to difference algebraic relations. However, a closer look reveals quite the contrary.
Indeed, some of our results actually differ from the statements anticipated by naive analogy. For example, a $\s$-Picard-Vessiot extension (a suitable minimal field containing all solutions and their transforms under $\s$) for a given linear differential equation is not unique, it is only unique up to powers of $\s$.

There is a certain class of results in differential algebra whose difference analogs simply fail. For example, a difference ideal which is maximal in the set of all proper difference ideals of a given difference ring, need not be prime. This jeopardizes the classical construction of a Picard-Vessiot extension by taking the quotient by some suitable ideal in the universal solution ring. So genuinely new ideas are needed. Moreover, certain results, well-known in differential algebra, have not been available in the literature for the difference case. For example, Kolchin's theory of constrained extensions (\cite{Kolchin:constrainedExtensions}) plays a crucial role in the developments in \cite{cassisinger} and \cite{HardouinSinger}. A difference analog has been made available only recently by the third author in \cite{Wibmer:Chevalley}.

\medskip

In the classical Picard-Vessiot theory of linear differential equations, it is usually assumed that the field $k=\{a\in K|\ \de(a)=0\}$ of all constant elements in the base differential field $(K,\de)$ is algebraically closed. The reason for this is twofold:
\begin{enumerate}
\item It is needed to establish the existence and uniqueness of a Picard-Vessiot extensions $L|K$ for a given differential
equation $\de(y)=Ay$, $A\in K^{n\times n}$.
\item Identifying a linear algebraic group $G$ over $k$ with its $k$-rational points $G(k)$, we can make sense of the statement that
the group $\Aut^\de(L|K)$ of differential automorphism of $L|K$ is a linear algebraic group over $k$, by identifying $\Aut^\de(L|K)$ with a subgroup of $\Gl_n(k)$ via the choice of a fundamental solution matrix for $\de(y)=Ay$ in $L$.
\end{enumerate}

If $k$ is not algebraically closed the existence and uniqueness can not be guaranteed in all generality but there are relative versions: After a finite algebraic extension of the constants, there always exists a Picard-Vessiot extension. Similarly, two Picard-Vessiot extensions become isomorphic after a finite base extension of the constants.

Concerning (ii), $\Aut^\de(L|K)$ can still be identified with the $k$-rational points $G(k)$ of a linear algebraic group $G$, but as $k$ is not algebraically closed, $G$ may contain significantly more information than $G(k)$. The Galois correspondence no longer holds in the naive sense. For example, not every element of $L$, which is fixed by $\Aut^\de(L|K)$, must lie in $K$. It is well known that this defect can be avoided by employing an appropriate functorial formalism (\cite[Appendix]{Takeuchi:hopfalgebraicapproach}, \cite{dyckdesc}, \cite{Deligne:categoriestannakien}). All the expected results can be restored by considering not only the action of $G(k)$ on the solutions of $\de(y)=Ay$, but also the action of $G(S)$, where $S$ is an arbitrary $k$-algebra.

If we are interested not only in algebraic relations among the solutions of $\de(y)=Ay$, but in differential algebraic relations or difference algebraic relations, then, instead of assuming that the $\de$-constants $k$ are algebraically closed, it is natural to assume that $k$ is a differentially closed or difference closed field. While the algebraically closed field of complex numbers comes up in the theory of linear differential equations rather naturally, differentially closed fields or difference closed fields are not the kind of objects a mathematician will encounter on a daily basis. We rather see them as a tool to make certain things work: It is sometimes convenient to work inside these large fields so that you do not need to pass to extensions when performing certain constructions.

However, for the applications of the theory, the assumption that the $\de$-constants are differentially or difference closed is somewhat of a hindrance: One has to first extend the field of constants, then apply the theory, and finally find some usually rather ad-hoc descent arguments to get back to the situation originally of interest.

In this article we completely avoid the assumption that the $\de$-constants are difference closed. We provide relative versions of the existence and uniqueness theorem and we employ a functorial formalism for difference algebraic groups, allowing us to establish the Galois correspondence and related results in maximal generality. We think that, based on the approach of this article, it will be a straight forward matter to also remove the assumption of differentially closed constants from \cite{cassisinger} and \cite{HardouinSinger}.

It is well-known that a functorial-schematic approach to algebraic groups has many benefits and introduces genuinely new phenomena, which are not visible over the algebraic closure. For example, over a field $k$ of characteristic $p>0$, the center of $\Sl_{p}$ is $\mu_p$, the group scheme of $p$-th roots of unity, whereas the center of $\Sl_p(\overline{k})$ is trivial. A similar phenomenon occurs for difference algebraic groups, even in characteristic zero, since ``difference nilpotence'' is not restricted to positive characteristic. If $G$ is a difference algebraic group over a difference field $k$, then, even if $k$ is difference closed, $G$ might contain significantly more information than $G(k)$. For example,
let $G_1,G_2$ denote the difference algebraic subgroups of $\Gl_{1,k}$ given by
$$G_1(S)=\{g\in S^\times|\ g^2=1\}\leq\Gl_1(S)$$ and
$$G_2(S)=\{g\in S^\times|\ g^2=1,\ \s(g)=g\}\leq\Gl_1(S).$$
Here $S$ is any difference algebra over $k$, i.e. a $k$-algebra, equipped with an endomorphism $\s\colon S\to S$ which extends $\s\colon k\to k$.
Then $G_1(k)=\{1,-1\}=G_2(k)$, but of course $G_1\neq G_2$. Moreover, if the $\s$-Galois group of $\de(y)=ay$ is $G_2$, then
$$g\left(\frac{\s(y)}{y}\right)=\frac{\s(gy)}{gy}=\frac{\s(y)}{y}$$ for every $g\in G_2(S)$ and every $k$-$\s$-algebra $S$. So, by the Galois correspondence, $\frac{\s(y)}{y}$ lies in the base field. In other words, $y$ satisfies a difference equation of the form $\s(y)=by$. This simple and important difference algebraic relation, expressing the $\s$-integrability of $\de(y)=ay$, is not detected by $G_2(k)$. As illustrated in \cite{articletwo}, an equation $\de(y)=ay$ with $\s$-Galois group $G_1$ is not $\s$-integrable.
So we might loose a lot of valuable information about the differential equation if we replace its Galois group $G$ by $G(k)$. This is another reason why, in our opinion, a functorial approach is indispensable. However, in Section \ref{subsec:perfsep}, we explain in all detail what the outcome will be if one really wants to insist to work with $G(k)$ instead of $G$.

The theory of difference algebraic groups is still in its infancy. In the context of groups definable in ACFA, certain groups defined by algebraic difference equations have played a quite crucial role in some of the recent applications of model theory to number theory. (See the appendix for references.) However, there is no coverage of foundational results concerning difference algebraic groups in the literature which fits our needs.
We have therefore collected certain basic aspects of the theory of difference algebraic groups in an appendix. This appendix will also be used in \cite{OvchinnikovWibmer:SGaloisTheoryOfLinearDifferenceEquations}, but in view of potential applications of difference algebraic groups beyond $\s$-Galois theory, we hope that this appendix will serve other purposes as well.

\medskip

Our main motivation to initialize this Galois theory was the creation of a versatile tool for the systematic study of the difference algebraic relations among the solutions of a linear differential equation. The application of our theory to this problem follows the usual paradigm: Since the difference algebraic relations among the solutions of a linear differential equation are governed by a difference algebraic group, they must follow a rather restricted pattern. Using structure theorems for difference algebraic groups one is often able to elucidate this pattern. In Section \ref{sec: The sGalois group of a linear differential equation}, we show that the Zariski closure of the Galois group (in the sense of our theory) of a linear differential system $\de(y)=Ay$ agrees with the Galois group $\G$ of $\de(y)=Ay$ in the sense of classical Picard-Vessiot theory. Thus, if we can classify the Zariski dense difference algebraic subgroups of $\G$, we obtain, via our Galois theory, a classification of the possible difference algebraic relations among the solutions of $\de(y)=Ay$. Such classifications are available for tori $\bold{G}_m^n$ (Lemma \ref{lemma:classgm}), vector groups $\bold{G}_a^n$ and the semidirect product $\bold{G}_a\rtimes\bold{G}_m$. Thanks to \cite{ChatHrusPet}, such a classification is also available for almost simple algebraic groups. The applications of these structure theorems for difference algebraic groups to the study of difference algebraic relations among the solutions of a linear differential equation will be presented in \cite{articletwo}.
Exemplarily, let us state here a result which corresponds to the group $\bold{G}_m$ and applies to the field $\C(x)$ of rational functions.

\begin{cor}[{\cite[Cor. 3.11]{articletwo}}]
Let $L$ be a field extension of $\C(x)$ eqipped with a derivation $\de$ and an endomorphism $\s$ such that $\de\s=\s\de$ and $\{a\in L|\ \de(a)=0\}=\C$. Assume that the restriction of $\de$ to $\C(x)$ equals $\frac{d}{dx}$ and that $\s(f(x))=f(x+1)$ for $f\in\C(x)$. If $z\in L$ satsfies $\de(z)=a z$ with $a\in\C(x)^\times$, then $z$ is transformally dependent over $\C(x)$, i.e., $z,\s(z),\s^2(z),\ldots$ are algebraically dependent over $\C(x)$, if and only if
there exist $P\in \C[x]$, $f \in \C(x)^\times$ and $N\in \Z^\times$ such that $a=P + \frac{1}{N}\frac{\de(f)}{f}$.
\end{cor}

In Theorem 5.11 of \cite{articletwo} we also show how the discrete integrability of a linear differential equation can be characterized through our new Galois group. Combining this with the structure of the Zariksi dense difference algebraic subgroups of almost simple algebraic groups yields a method to show that certain special function are independent in the sense of difference algebra. For example:

\begin{cor}[{\cite[Cor. 6.10]{articletwo}}]
Let $A(x)$ and $B(x)$ be two $\C$-linearly independent solutions of Airy's equation $\de^2(y)-xy=0$. Then $A(x),B(x)$ and $\de(B)(x)$ are transformally independent, i.e., the functions $A(x),B(x), \de(B)(x),A(x+1), B(x+1), \de(B)(x+1), A(x+2),\ldots$ are algebraically independent over $\C(x)$.
\end{cor}


\medskip

We now describe the content of the article in more detail. In Section \ref{sec: sPicard-Vessiot extensions and sPicard-Vessiot rings} we introduce $\s$-Picard-Vessiot extensions and $\s$-Picard-Vessiot rings. These are the places where the solutions to our linear differential equations live. The basic questions of existence and uniqueness of a $\s$-Picard-Vessiot extension for a given linear differential equation have already been addressed by the third author in \cite{Wibmer:Chevalley}. So, concerning these questions, we largely only recall the results from \cite{Wibmer:Chevalley}.

In Section \ref{sec: The sGalois group of a linear differential equation}, we introduce and study the $\s$-Galois group $G$ of a $\s$-Picard-Vessiot extension $L|K$. It is a difference algebraic group over the difference field $k$ of $\de$-constants of the base $\ds$-field $K$. Roughly speaking, the $\s$-Galois group consists of all automorphisms of the solutions which respect $\de$ and $\s$.
The choice of a fundamental solution matrix in $L$ determines an embedding $G\hookrightarrow \Gl_{n,k}$ of difference algebraic groups.
We also compute the $\s$-Galois group in some simple and classical examples. Concerning applications, the most important result here is that the difference transcendence degree of the $\s$-Picard-Vessiot extension agrees with the difference dimension of the $\s$-Galois group. As it is also illustrated in \cite{articletwo}, this allows one to reduce questions of difference transcendence to questions about difference algebraic groups.

In Section \ref{sec: The SGalois correspondence}, we establish the analogs of the first and second fundamental theorem of Galois theory. Here we employ some ideas of the Hopf-algebraic approach to Picard-Vessiot theory (\cite{AmanoMasuokaTakeuchi:HopfPVtheory}).

In Section \ref{sec:sseparability}, we present some refinements of our $\s$-Galois correspondence. We show how certain properties of the $\s$-Galois group are reflected by properties of the $\s$-Picard-Vessiot extension. Most notably, this concerns the property of the $\s$-Galois group to be perfectly $\s$-reduced. Perfectly $\s$-reduced $\s$-schemes correspond to what goes under the name ``difference variety'' in the classical literature \cite{Cohn:difference} and \cite{Levin}.

In the first few sections of the appendix we have collected some basic results pertaining to the geometry of difference equations. Then we present some foundational aspects of the theory of difference algebraic groups in a way that is suitable for the main text.

\medskip

We are grateful to Phyllis Cassidy, Zo\'{e} Chatzidakis, Shaoshi Chen, Moshe Kamensky, Akira Masuoka, Alexey Ovchinnikov and Michael Singer for helpful comments.
We would also like to acknowledge the support of CIRM, where part of this work was conducted.

\section{$\s$-Picard-Vessiot extensions and $\s$-Picard-Vessiot rings} \label{sec: sPicard-Vessiot extensions and sPicard-Vessiot rings}

In this section, we introduce the notions of $\s$-Picard-Vessiot extension and $\s$-Picard-Vessiot ring for a linear differential equation. We establish some first properties of these minimal solution fields, respectively rings and show that the existence of a $\s$-Picard-Vessiot extension can be guaranteed under very mild restrictions. We also show that a $\s$-Picard-Vessiot extension for a given differential equation is essentially unique.

Before really getting started, let us agree on some conventions: All rings are commutative with identity and contain the field of rational numbers. In particular all fields are of characteristic zero. A differential ring (or $\de$-ring for short) is a ring $R$ together with a derivation $\de:R\rightarrow R$. A difference ring (or $\s$-ring for short) is a ring $R$ together with a ring endomorphism $\s:R\rightarrow R$. We do not assume that $\s$ is an automorphism or injective. A $\s$-ring with $\s$ injective is called $\s$-reduced. If $\s$ is an automorphism the $\s$-ring is called inversive.
The $\de$-constants are $R^\de=\{r\in R| \ \de(r)=0\}$ and the $\s$-constants are $R^\s=\{r\in R|\ \s(r)=r\}$.

The basic algebraic concept that facilitates the study of difference algebraic relations among solutions of differential equations is the notion of a $\ds$-ring. A $\ds$-ring is a ring $R$, that is simultaneously a $\de$ and a $\s$-ring such that for some unit $\hslash\in R^\de$
\beq \label{eq:com} \de(\s(r)) = \hslash\s(\de(r)) \eeq
for all $r\in R$. If $\hslash=1$, then $\s$ and $\de$ commute. The element $\hslash$ is understood to be part of the data of a $\ds$-ring. So a morphism $\psi:R\rightarrow R'$ of $\ds$-rings is a morphism of rings such that $\psi\s=\s'\psi$, $\psi\de=\de'\psi$ and $\psi(\hslash)=\hslash'$. The reason for not simply assuming $\hslash=1$ is that this factor appears in some examples of interest. See Example \ref{exa:dsfields} below. Note that condition (\ref{eq:com}) implies that $R^\de$ is a $\s$-ring.

We refer the reader to Section \ref{sec:Someterminoplogyfromdifferencealgebra} in the appendix for an exposition of some basic notions in difference algebra. We largely use standard notations of difference and differential algebra as can be found in \cite{Cohn:difference}, \cite{Levin} and \cite{Kolchin:differentialalgebraandalgebraicgroups}.
For the convenience  of the reader we recall the basic conventions: Algebraic attributes always refer to the underlying ring. For example a $\ds$-field is a $\ds$-ring whose underlying ring is a field. By a $K$-$\ds$-algebra $R$ over a $\ds$-field $K$ one means a $K$-algebra $R$ that has the structure of a $\ds$-ring such that $K\rightarrow R$ is a morphism of $\ds$-rings. Similarly for $\de$ or $\s$ instead of $\ds$.
An extension of $\ds$-fields is an extension of fields such that the inclusion map is a morphism of $\ds$-rings.
If $R$ is a $K$-$\s$-algebra over a $\s$-field $K$ and $B$ a subset of $R$ then $K\{B\}_\s$ denotes the smallest $K$-$\s$-subalgebra of $R$ that contains $B$. If $R=K\{B\}_\s$ for some finite subset $B$ of $R$ we say that $R$ is finitely $\s$-generated over $K$. If $L|K$ is an extension of $\s$-fields and $B\subset L$ then $K\<B\>_\s$ denotes the smallest $\s$-field extension of $K$ inside $L$ that contains $B$.

\begin{exa} \label{exa:dsfields}Some basic examples of $\ds$-fields of interest for us are the following:
\begin{itemize}
\item The field $K=\C(x)$ of rational functions in one variable $x$ over the field of complex numbers becomes a $\ds$-field by setting $\de := \frac{d}{d x} $ and $\s(f(x)):=f(x+1)$ for $f\in \C(x)$. We have $\hslash=1$ and $K^\de=\C$.

One can also take $\de := x\frac{d}{d x}$ and $\s(f(x)):=f(qx)$ for some $q\in\C\smallsetminus\{0\}$. Again we have $\hslash=1$ and $K^\de=\C$.

If we set $\de := x\frac{d}{d x}$ and $\s(f(x)):=f(x^d)$ for some integer $d\geq 2$. Then $K$ becomes a $\ds$-field with $\hslash=d$ and $K^\de=\C$.

\item Let $K=\C(\alpha,x)$ be the field of rational functions in two variables $\alpha$ and $x$. We set $\de:=\frac{d}{d x}$ and $\s(f(\alpha,x)):=f(\alpha+1,x)$. Then we have $\hslash=1$ and $K^\de=\C(\alpha)$ with $\s(\alpha)=\alpha+1$. Alternatively, one could also take $\s$ defined by $\s(f(\alpha,x))=f(q\alpha,x)$. The field $\C(\alpha)((x))$ of formal Laurent series over $\C(\alpha)$ naturally is a $\ds$-extension of $K$.

\item Let $k$ be an ultrametric field of characteristic zero,
complete with respect to a discrete valuation. Assume that the residue field of $k$ is $\F_q$, a field of characteristic $p>0$ with $q$ elements. 
We denote by $|~|$ the $p$-adic norm of $k$, normalized so that $|p|=p^{-1}$.
The ring $\cE_k^\dag$ of all $f=\sum_{n\in\Z}a_nx^n$, with $a_n\in k$, such that
\begin{itemize}
\item
there exists $\veps>0$, depending on $f$, such that
for any $1<\varrho<1+\veps$
we have $\lim_{n\to \pm\infty}|a_n|\varrho^n=0$;
\item
$\sup_n|a_n|$ is bounded;
\end{itemize}
is actually a field with residue field $\F_q((x))$. (Cf. \cite[Lemma 15.1.3, p. 263]{Kedlaya:padicdifferentialequations}.)
\par
We consider the field $K= \cE_k^\dag$ as a $\ds$-field with derivation $\de:=x \frac{d}{d x}$ and endomorphism $\s\colon K\to K$, a lifting of the Frobenius endomorphism of $\F_q$.
Namely, we consider an endomorphism $\s$ of $k$ such that $|\s(a)-a^p|<1$, for any $a\in k$, $|a|\leq 1$ and we extend 
the action of $\s$ to $K$ by setting $\s(x)=x^p$, so that $\s(\sum_{n\in\Z}a_nx^n)=\sum_{n\in\Z}\s(a_n)x^{pn}$. We have
$\hslash=p$ and $K^\de=k$.

\end{itemize}
\end{exa}

%

\begin{defn}\label{defi:PVextPVring}
Let $K$ be a $\ds$-field and $A\in K^{n\times n}$. A $\ds$-field extension $L$ of $K$ is called a \emph{$\s$-Picard-Vessiot extension for $\de(y)=Ay$ (or $A$)} if
\begin{enumerate}
\item there exists $Y \in\Gl_{n}(L)$ such that $\de(Y)=AY$ and $L =  K \< Y_{ij}|\ 1\leq i,j \leq n  \>_{\s}$ and
\item $L^\de = K^\de$.
\end{enumerate}
A $K$-$\ds$-algebra $R$ is called a \emph{$\s$-Picard-Vessiot ring for $\de(y)=Ay$} if
\begin{enumerate}
\item there exists $Y \in\Gl_{n}(R)$ such that $\de(Y)=AY$ and $R=K\big\{Y_{ij},\frac{1}{\det(Y)}\big\}_\s$ and
\item $R$ is $\de$-simple, i.e., $R$ has no non-trivial $\de$-ideals.
\end{enumerate}
A $\ds$-field extension $L$ of $K$ is called a \emph{$\s$-Picard-Vessiot extension} if it is a $\s$-Picard-Vessiot extension for some differential equation $\de(y)=Ay$ with $A\in K^{n\times n}$; similarly for $\s$-Picard-Vessiot rings.
\end{defn}

To simplify the notation we write $K\<Y\>_\s$ for $K \< Y_{ij}|\ 1\leq i,j \leq n  \>_{\s}$ and $K\big\{Y,\frac{1}{\det(Y)}\big\}_\s$ for $K\big\{Y_{ij},\frac{1}{\det(Y)}\big\}_\s$.
If $R$ is a $K$-$\ds$-algebra, then a matrix $Y\in\Gl_n(R)$ such that $\de(Y)=AY$ is called a \emph{fundamental solution matrix for $\de(y)=Ay$}. Thus, a $\s$-Picard-Vessiot extension is a $\ds$-field extension of $K$ without new $\de$-constants, $\s$-generated by the entries of a fundamental solution matrix.

If $Y,Y'\in\Gl_n(R)$ are two fundamental solution matrices for $\de(y)=Ay$ in some $K$-$\ds$-algebra then there exists $C\in\Gl_n(R^\de)$ such that $Y'=YC$. This is simply because
\[\de(Y^{-1}Y')=\de(Y^{-1})Y'+Y^{-1}\de(Y')=-Y^{-1}\de(Y)Y^{-1}Y'+Y^{-1}AY'=-Y^{-1}AYY^{-1}Y'+Y^{-1}AY'=0.\]

Note that we obtain the usual definitions of Picard-Vessiot extension and Picard-Vessiot ring of a linear differential equation $\de(y)=Ay$ if we require that $\s$ is the identity (on $K,L$ and $R$) in Definition \ref{defi:PVextPVring}.

The keen reader might have noticed a slight deviation between our definition of a $\s$-Picard-Vessiot ring and the corresponding notion in \cite{cassisinger} and \cite{HardouinSinger}: We require a $\s$-Picard-Vessiot ring to be $\de$-simple and not only to be $\ds$-simple. See \cite[p. 167]{wibdesc} for some comments on this issue.

\subsection{First properties of $\s$-Picard-Vessiot rings and extensions}

Our first concern is to show that the notions of $\s$-Picard-Vessiot ring and $\s$-Picard-Vessiot extension are essentially equivalent.

$\s$-Picard-Vessiot extensions can be seen as $\s$-analogs of classical Picard-Vessiot extensions. There is, however, another relation between the classical Picard-Vessiot theory and our $\s$-Picard-Vessiot theory: As we will now explain, every $\s$-Picard-Vessiot extension is a limit of Picard-Vessiot extensions.


Let $K$ be a $\ds$-field. From $\de(\s(a))=\hslash(\s(\de(a)))$ for $a\in K$ it follows that $\de(\s^d(a))=\hslash_d\s^d(\de(a))$ for $d\geq 0$ where
\[ \hslash_d:=\hslash\s(\hslash)\cdots\s^{d-1}(\hslash).\]

Given a differential equation $\de(y)=A y$ with $A\in K^{n \times n}$ we can consider for every $d\geq 0$ the differential equation of the  $\s$-jets of order $d$
of $\de(y)=A y$, namely, the linear system $\de(y)=A_dy$, where
\begin{equation} \label{eqn:Ad}
A_d:=\left(\begin{array}{ccccc}
 A & 0 & 0 & \cdots & 0 \\
 0 & \hslash\s(A) & 0 &\cdots & 0\\
\vdots & & \ddots & & \vdots  \\
0 & \cdots & 0& \hslash_{d-1}\s^{d-1}(A)& 0\\
0 & \cdots & 0 & 0& \hslash_d\s^d(A)
\end{array}\right)\in\Gl_{n(d+1)}(K).
\end{equation}
\begin{lemma} \label{lemma: PVjets}
Let $K$ be a $\ds$-field and $A\in K^{n\times n}$. If $L|K$ is a $\s$-Picard-Vessiot extension for $\de(y)=Ay$ with fundamental solution matrix $Y\in\Gl_n(L)$, then
$R_d:=K\left[Y,\frac{1}{\det(Y)},\s(Y),\frac{1}{\det(\s(Y))},\ldots,\s^d(Y),\frac{1}{\det(\s^d(Y))} \right]$ is a (classical) Picard-Vessiot ring for $\de(y)=A_dy$ for every $d\geq 0$. \end{lemma}
\begin{proof}

For $i=0,\ldots,d$ we have  $\de(\s^i(Y))=\hslash_i\s^i(\de(Y))=\hslash_i\s^i(A)\s^i(Y)$. It follows that the fraction field $L_d$ of $R_d$
 is a $\de$-subfield of $L$ and that
\[Y_d=\left(\begin{array}{ccccc}
 Y & 0 & 0 & \cdots & 0 \\
 0 & \s(Y) & 0 &\cdots & 0\\
\vdots & & \ddots & & \vdots  \\
0 & \cdots & 0& \s^{d-1}(Y)& 0\\
0 & \cdots & 0 & 0& \s^d(Y)
\end{array}\right)\in\Gl_{n(d+1)}(L_d).
\]
is a fundamental solution matrix for $\de(y)=A_dy$. Because $L_d^\de\subset L^\de=K^\de$, we conclude that $L_d$ is a Picard-Vessiot extension for $\de(y)=A_dy$. Therefore $R_d$ is a Picard-Vessiot ring for $\de(y)=A_dy$. (See \cite[Prop. 1.22, p. 16]{vdPutSingerDifferential} for the case that $K^\de$ is algebraically closed or \cite[Cor. 2.6, p. 6]{dyckdesc} for the general case.)
\end{proof}
In order to prove the next proposition, we need another simple lemma:

\begin{lemma}\label{lemma:sdomain} Let $R$ be a $\de$-simple $\ds$-ring. Then $R$ is a
$\s$-domain, \ie, $R$ is an integral domain and $\s$ is injective on $R$. In particular, $\de$ and $\s$ naturally extend to the field of fractions $L$ of $R$. Moreover, $L^\de=R^\de$. \end{lemma}
\begin{proof}
It is well known that a $\de$-simple differential ring is an integral domain. (See e.g. \cite[Lemma 1.17, p. 13]{vdPutSingerDifferential}.)
It follows from the commutativity relation (\ref{eq:com}) for $\de$ and $\s$ that the kernel of $\s$ on $R$ is a $\de$-ideal. So by $\de$-simplicity, $\s$ must be injective on $R$.

Let $a \in L^\de$. Then, $\ida=\{ r \in R |\ ar \in R \}$ is a non-zero $\de$-ideal of $R$. Thus $a \in R^\de$.\end{proof}

As in the classical theory, $\s$-Picard-Vessiot rings and $\s$-Picard-Vessiot extensions are closely related:

\begin{prop}\label{prop:pvringpvext}
Let $K$ be a $\ds$-field and $A\in K^{n\times n}$. If $L|K$ is a $\s$-Picard-Vessiot extension for $\de(y)=Ay$ with fundamental solution matrix $Y\in\Gl_n(L)$, then
$R:=K\{Y,\frac{1}{\det(Y)} \}_\s$
 is a $\s$-Picard-Vessiot ring for $\de(y)=Ay$.
Conversely, if $R$ is a $\s$-Picard-Vessiot ring for $\de(y)=Ay$ with $R^\de=K^\de$, then the field of fractions of $R$ is a $\s$-Picard-Vessiot extension for $\de(y)=Ay$.
\end{prop}
\begin{proof}
To prove the first claim we only have to show that $R:=K\{Y,\frac{1}{\det(Y)} \}_\s$ is $\de$-simple. Suppose that $\ida$ is a non-trivial $\de$-ideal of $R$. Then, for a suitable $d\geq 0$, the ideal $\ida\cap R_d$ is a non-trivial $\de$-ideal of $$R_d:=K\left[Y,\frac{1}{\det(Y)},\s(Y),\frac{1}{\det(\s(Y))},\ldots,\s^d(Y),\frac{1}{\det(\s^d(Y))} \right].$$
This contradicts the fact that $R_d$ is $\de$-simple (Lemma \ref{lemma: PVjets}).

The second claim is clear from Lemma \ref{lemma:sdomain}.
\end{proof}

The next proposition states that the condition $R^\de=K^\de$ in Proposition \ref{prop:pvringpvext} is always satisfied if $K^\de$ is $\s$-closed.
But let us first recall what it means for a $\s$-field to be $\s$-closed:

\begin{defn} A $\s$-field $k$ is called \emph{$\s$-closed} if for every finitely $\s$-generated $k$-$\s$-algebra $R$ which is a $\s$-domain (i.e., $R$ is an integral domain and $\s\colon R\to R$ is injective) there exists a morphism $R\to k$ of $k$-$\s$-algebras.\end{defn}

Model theorists usually call $\s$-closed $\s$-fields ``existentially closed'' or a ``model of ACFA''. (See e.g. \cite{Macintyre:GenricAutomorphismsOfFields} or \cite{Hrushovskietal:ModelTheoryofDifferencefields}.)
The assumption that $R$ is a $\s$-domain in the above definition is quite crucial. If $k$ is a $\s$-closed $\s$-field and $R$ a finitely $\s$-generated $k$-$\s$-algebra, there need not exist a $k$-$\s$-morphism $R\to k$. Indeed,
there exists a $k$-$\s$-morphism $R\to k$ if and only if there exists a $\s$-prime $\s$-ideal $\q$ in $R$, i.e., a prime ideal $\q$ of $R$ with $\s^{-1}(\q)=\q$. See Lemma \ref{lemma:explainsclosed}.

As indicated in the introduction, in the classical Picard-Vessiot theory the assumption that the constants are algebraically closed is widely-used to avoid certain technicalities and it seems that many authors consider this assumption as natural. So, by way of analogy, it would be natural for us to assume that $K^\de$ is $\s$-closed. However, none of the examples of $\ds$-fields relevant for us (Example \ref{exa:dsfields}) does us the favour to have $\s$-closed $\de$-constants. So we have been careful to avoid this assumption.

\begin{prop} \label{prop:nonewconstants}
Let $K$ be a $\ds$-field such that $K^\de$ is a $\s$-closed $\s$-field. Then $R^\de=K^\de$ for every $\s$-Picard-Vessiot ring $R$ over $K$.
\end{prop}
\begin{proof}
This is \cite[Corollary. 2.18, p. 1393]{Wibmer:Chevalley}.
\end{proof}

The following simple lemma is a fundamental tool for the development of our $\s$-Galois theory.
\begin{lemma} \label{lemma:fonda}
Let $K$ be a $\ds$-field and $A\in K^{n\times n}$. If $R_1$ and $R_2$ are $\s$-Picard-Vessiot rings for $\de(y)=Ay$ with $R_1^\de=R_2^\de=k:=K^\de$, then the canonical map
$$R_1\otimes_k(R_1\otimes_KR_2)^\de\longrightarrow R_1\otimes_K R_2$$
is an isomorphism of $R_1$-$\ds$-algebras. Moreover, $(R_1\otimes_KR_2)^\de$ is finitely $\s$-generated over $k$. Indeed if $Y_1\in\Gl_n(R_1)$ and $Y_2\in\Gl_n(R_2)$ are fundamental solution matrices for $\de(y)=Ay$, then $(R_1\otimes_K R_2)^\de=k\{Z,\frac{1}{\det(Z)}\}_\s$ where
$Z:=(Y_1\otimes 1)^{-1}(1\otimes Y_2)\in\Gl_n(R_1\otimes_K R_2)$.
\end{lemma}
\begin{proof}
Because $Y_1\otimes 1, 1\otimes Y_2\in \Gl_n(R_1\otimes_K R_2)$ are fundamental solution matrices for $\de(y)=Ay$ we have $\de(Z)=0$. It follows
 from the commuting relation (\ref{eq:com}) that $k\{Z,\frac{1}{\det{Z}}\}_\s\subset (R_1\otimes_K R_2)^\de$.
Since $1\otimes Y_2=(Y_1\otimes 1)Z$ and
 $R_2$ is $\s$-generated by $Y_2$ and $\frac{1}{\det(Y_2)}$ we see that $R_1\otimes_K R_2=R_1\cdot k\{Z,\frac{1}{\det{Z}}\}_\s$. It holds in general that the $\de$-constants
 of a $\de$-algebra over a $\de$-simple $\de$-ring $R_1$ are linearly disjoint from $R_1$ over $R_1^\de$. (See e.g. \cite[Corollary 3.2, p. 753]{AmanoMasuoka:artiniansimple}.)
 Therefore $R_1\otimes_K R_2=R_1\otimes_k k\{Z,\frac{1}{\det{Z}}\}_\s$. This also shows that $(R_1\otimes_K R_2)^\de=k\{Z,\frac{1}{\det{Z}}\}_\s$.
\end{proof}

\subsubsection{Comparison with the Hopf algebraic approach to Picard-Vessiot theory}

In \cite{Takeuchi:hopfalgebraicapproach} (see also \cite{AmanoMasuokaTakeuchi:HopfPVtheory}) M. Takeuchi gave a very general definition of Picard-Vessiot extensions which does not require any finiteness assumptions. It simply postulates the validity of an algebraic reformulation of the torsor theorem.
\begin{defn}[Definition 1.8 in \cite{AmanoMasuokaTakeuchi:HopfPVtheory}]
An extension $L|K$ of $\pa$-fields is called \emph{Picard-Vessiot in the sense of Takeuchi} if the following conditions are satisfied.
\begin{enumerate}
\item $L^\pa =K^\pa.$
\item There exists a $K$-$\de$-subalgebra $R$ of $L$ such that $L$ is the quotient field of $R$ and the canonical map
$$ R \otimes_{K^\de} ( R \otimes_{K} R)^\pa\rightarrow  R \otimes_{K} R $$
is surjective. (It is then automatically an isomorphism.) We call the $\de$-ring $R$ a Picard-Vessiot ring in the sense of Takeuchi.
\end{enumerate}
\end{defn}

As in Section \ref{subsec:ksigmaSchemeAssociatedwithkScheme}, we denote with $(-)^\sharp$ the forgetful functor that forgets $\s$. So if $R$ is a $\ds$-ring then $R^\sharp$ is a $\de$-ring.

\begin{rem} \label{rem:compareTakeuchi}
Let $L|K$ be a $\s$-Picard-Vessiot extension. Then $L^\sharp|K^\sharp$ is a Picard-Vessiot extension in the sense of Takeuchi.
\end{rem}
\begin{proof}
This is clear from Lemma \ref{lemma:fonda} (with $R_1=R_2$).
\end{proof}

Once we have defined the $\s$-Galois group $G$ of a $\s$-Picard-Vessiot extension $L|K$ we will show that also the Galois group of $L^\sharp|K^\sharp$ in the sense of Takeuchi (an affine group scheme, in general not of finite type over $k=K^\de$) can be obtained from $G$ by forgetting $\s$. See Remark \ref{rem:compareTakeuchigroup}.

Since the Picard-Vessiot theory in \cite{Takeuchi:hopfalgebraicapproach} (or \cite{AmanoMasuokaTakeuchi:HopfPVtheory}) does not at all take into account $\s$, it is clearly not an appropriate theory to discuss the questions of this article, e.g., the $\s$-algebraic relations among the solutions of a linear differential equation. Nevertheless, it is sometimes very convenient to know that every $\s$-Picard-Vessiot extension can be seen as a Picard-Vessiot extension in the sense of Takeuchi; it allows for some shortcuts in the proofs. This applies most notably to our proof of the $\s$-Galois correspondence, which can be interpreted as the restriction of Takeuchi's correspondence to the $\s$-stable objects on both sides.


The fact that every $\s$-Picard-Vessiot extension can be seen as a Picard-Vessiot extension in the sense of Takeuchi is also used in the proof of the following lemma.

\begin{lemma} \label{lemma:sigmaPVringunique}
Let $L|K$ be a $\s$-Picard-Vessiot extension. Let $A\in K^{n\times n}$ and $Y\in\Gl_n(L)$ such that $L|K$ is a $\s$-Picard-Vessiot extension for $\de(y)=Ay$ with fundamental solution matrix $Y$. Let $A'\in K^{n'\times n'}$ and $Y'\in\Gl_{n'}(L)$ be another pair of matrices such that $L|K$ is a $\s$-Picard-Vessiot extension for $\de(y)=A'y$ with fundamental solution matrix $Y'$. Then the corresponding $\s$-Picard-Vessiot rings $R=K\{Y,\frac{1}{\det(Y)}\}_\s$ and $R'=K\{Y',\frac{1}{\det(Y')}\}_\s$ are equal.
%
\end{lemma}
\begin{proof}
It follows from Lemma \ref{lemma:fonda} that $R$ and $R'$ are Picard-Vessiot rings in the sense of Takeuchi. Because a Picard-Vessiot ring in the sense of Takeuchi is unique (inside L) by \cite[Lemma 1.11, p. 133]{AmanoMasuokaTakeuchi:HopfPVtheory}, it follows that $R=R'$.
\end{proof}

By the above remark it makes sense to speak of the $\s$-Picard-Vessiot ring $R$ of the $\s$-Galois extension $L|K$ without reference to a specific equation $\de(y)=Ay$ and we shall henceforth adhere to this practice.

\subsubsection{Existence and uniqueness of $\s$-Picard-Vessiot extensions}

Before proceeding to develop the Galois theory of $\s$-Picard-Vessiot extensions, we shall be concerned with the fundamental questions of existence and uniqueness of $\s$-Picard-Vessiot extensions. These questions have already been addressed in \cite{Wibmer:Chevalley} to illustrate the usefulness of constrained extensions of $\s$-pseudo fields.
So we largely only recall the results from \cite{Wibmer:Chevalley}\footnote{In \cite{Wibmer:Chevalley} it is assumed that $\de$ and $\s$ commute, i.e., $\hslash=1$. However, the proofs in \cite{Wibmer:Chevalley} generalize to the slightly more general setting of this article without difficulty.}

\begin{prop}[Existence of $\s$-Picard-Vessiot rings] \label{prop:existence}
Let $K$ be a $\ds$-field and $A\in K^{n\times n}$. Then there exists a $\s$-Picard-Vessiot $R$ ring for $\de(y)=Ay$ such that $R^\de$ is an algebraic field extension of $K^\de$. 
\end{prop}
\begin{proof}
This is \cite[Lemma 2.16, p. 1392]{Wibmer:Chevalley}. Because of the importance of the result we recall the contruction:
For each of the systems $\de(y)=A_dy$ from Lemma \ref{lemma: PVjets} we are able to construct a (classical) Picard-Vessiot ring individually by taking the quotient of
$$S_{d}:= K\left[X,\tfrac{1}{\det(X)},\s(X),\tfrac{1}{\s(\det(X))},\ldots,\s^d(X), \tfrac{1}{\s^d(\det(X))}\right]$$ by some
$\de$-maximal $\de$-ideal $\m_d$ of $S_d$. Here $X$ is an $n\times n$-matrix of $\s$-indeterminates and the action of $\de$ on $S_d$ is determined by $\de(X)=AX$ and the commutativity relation (\ref{eq:com}). The difficulty is to make this construction compatible with $\s$: We need $\m_{d-1}\subset\m_d$ and $\s(\m_{d-1})\subset \m_d$. This difficulty can be resolved by a recourse to the prolongation lemma for difference kernels (\cite[Lemma 1, Chapter 6, p. 149]{Cohn:difference}).

We set $\m:=\bigcup_{d\geq 0}\m_d$ and $R:=k\{X,\frac{1}{\det(X)}\}_\s/\m$. So $R$ is the union of the $\de$-simple rings $R_d:=S_d/\m_d$. The $\de$-constants of a $\de$-simple $\de$-ring which is finitely generated as an algebra over a $\de$-field $K$ are algebraic over $K$. (See \cite[Lemma 1.17, p. 13]{vdPutSingerDifferential} for the case $K^\de$ algebraically closed or \cite[Theorem 4.4, p. 505]{Takeuchi:hopfalgebraicapproach} for the general case).
Thus $R_d^\de$ is algebraic over $K^\de$ and it follows that also $R^\de$ is algebraic over $K^\de$.
\end{proof}

%

From Propositions \ref{prop:pvringpvext} and \ref{prop:existence} we immediately obtain the existence of $\s$-Picard-Vessiot extensions over $\ds$-fields with algebraically closed $\de$-constants:

\begin{cor}[Existence of $\sigma$-Picard-Vessiot extensions] \label{cor:existenceofPVextension}
Let $K$ be a $\ds$-field and $A\in K^{n\times n}$. Assume that $K^\de$ is an algebraically closed field. Then there exists a $\s$-Picard-Vessiot extension for $\de(y)=Ay$ over $K$. \qed
\end{cor}

%
%

The standard assumption to guarantee the existence of (classical) Picard-Vessiot extensions is ``algebraically closed constants''. Since we can get by with the same assumption the above results are more or less optimal. In all generality the existence of $\s$-Picard-Vessiot extensions can not be guaranteed. Indeed, if $L|K$ is a $\s$-Picard-Vessiot extension for $\de(y)=Ay$ with fundamental solution matrix $Y\in\Gl_n(L)$ then $K(Y)\subset L$ is a (classical) Picard-Vessiot extension of $K$ for $\de(y)=Ay$. Thus, if there is no Picard-Vessiot extension for $\de(y)=Ay$ there can be no $\s$-Picard-Vessiot extension for $\de(y)=Ay$. A concrete example of a $\de$-field $K$ and an equation $\de(y)=Ay$ such that there exists no Picard-Vessiot extension for $\de(y)=Ay$ over $K$ has been provided by Seidenberg in \cite{seiden}. To obtain a concrete example of a $\ds$-field $K$ and an equation $\de(y)=Ay$ such that there exists no $\s$-Picard-Vessiot extension for $\de(y)=Ay$ over $K$ one simply has to add $\s$ as the identity on $K$.

Nevertheless, since often one can find solutions in some suitable field of functions, there are many natural situations where there exists a $\s$-Picard-Vessiot extension for a given linear differential equation, even if the $\de$-constants are not algebraically closed. The next proposition gives an example of such a ``natural situation''. See also Examples \ref{exa:exponential} to \ref{exa:padic}.


\begin{prop}\label{prop:taylorexp}
Let $k$ be a $\s$-field and let $K=k(x)$ denote the field of rational functions in one variable $x$ over $k$. Extend $\s$ to $K$ by setting $\s(x)=x$ and consider the derivation $\de=\frac{d}{d x}$. Thus $K$ is a $\ds$-field with $\hslash =1$ and $K^\de=k$.
Then for every $A\in K^{n\times n}$, there exists a $\s$-Picard-Vessiot extension $L|K$ for $\de(y)=Ay$.
\end{prop}
\begin{proof}
Since we are in characteristic zero, there exists an $a\in k^\s$ which is a regular point for $\de(y)=Ay$. That is, no denominator appearing in the entries of $A$ vanishes at $a$. We consider the field $k((x-a))$ of formal Laurent series in $x-a$ as a $\ds$-field by setting
$\de(\sum b_i(x-a)^i)=\sum ib_i(x-a)^{i-1}$ and $\s(\sum b_i(x-a)^i)=\sum\s(b_i)(x-a)^i$. Then $k((x-a))$ is naturally a $\ds$-field extension of $K$. By choice of $a$, there exists a fundamental solution matrix $Y\in \Gl_n(k((x-a))$ for $\de(y)=Ay$.

Since $k((x-a))^\de=k$ it is clear that $L:=K\langle Y\rangle_\s\subset k((x-a))$ is a $\s$-Picard-Vessiot extension for $\de(y)=Ay$.
%
%
\end{proof}

\medskip

Next we shall be concerned with the uniqueness of $\s$-Picard-Vessiot extensions and $\s$-Picard-Vessiot rings for a given linear differential equation $\de(y)=Ay$. To motivate our results we first recall the classical situation: Let $K$ be a $\de$-field and $A\in K^{n\times n}$. If $R_1$ and $R_2$ are two Picard-Vessiot rings for $\de(y)=Ay$, then there exists a finite algebraic extension $k'$ of $k:=K^\de$ containing $k_1:=R_1^\de$ and $k_2:=R_2^\de$ and an isomorphism of $K\otimes_k k'$-$\de$-algebras
\[R_1\otimes_{k_1}k'\simeq R_2\otimes_{k_2}k'.\]
In particular, if $k=K^\de$ is an algebraically closed field, a Picard-Vessiot ring (and thus also a Picard-extension) for a given equation is unique up to $K$-$\de$-isomorphisms.

To obtain a similar uniqueness result for $\s$-Picard-Vessiot extensions, one needs to understand the $\s$-analog of finite algebraic extensions.
This has been worked out in \cite[Section 2.1]{Wibmer:Chevalley}, where it is shown that constrained extensions of $\s$-pseudo fields (\cite[Definition 2.3, p. 1388]{Wibmer:Chevalley}) satisfy properties similar to algebraic extensions of fields. These constrained extensions can also be seen as $\s$-analogs of the constrained extensions of differential fields studied by E. Kolchin in \cite{Kolchin:constrainedExtensions}.

To state the main uniqueness result we need the following definition which will also be relevant later on in Section \ref{sec:sseparability}.

\begin{defi}
Let $k$ be a $\s$-field and $R$ a $k$-$\s$-algebra. We say that $R$ is \emph{$\s$-separable over $k$} if $\s$ is injective
on $R\otimes_k k'$ for every $\s$-field extension $k'$ of $k$.
\end{defi}

Note that in characteristic $p>0$ every field (or ring) can be considered as a difference field (ring) equipped with the Frobenius endomorphism $\s(a)=a^p$. In this situation $\s$-separability is the same thing as separability. The well-known characterizations of separability generalize in a straight forward manner. (Cf. \cite{Hrushovski:elementarytheoryoffrobenius} or \cite{Wibmer:thesis}.)
For example, the well-known fact a reduced $k$-algebra over a perfect field is separable generalizes to ``Every $\s$-reduced $k$-$\s$-algebra over an inversive $\s$-field $k$ is $\s$-separable.'' (Corollary \ref{cor:sigmaseparablesufficesextension} (i).) Here a $\s$-ring $R$ is called $\s$-reduced if $\s\colon R\to R$ is injective.

Now we can state the general uniqueness theorem for $\s$-Picard-Vessiot rings (\cite[Theorem 2.19, p. 1393]{Wibmer:Chevalley}).

\begin{thm}[Uniqueness of $\s$-Picard-Vessiot rings]\label{thm:unicity}
Let $K$ be a $\ds$-field such that $K$ is $\s$-separable over $k:=K^\de$. Assume that $R_1$ and $R_2$ are two $\s$-Picard-Vessiot rings over $K$ for the same equation $\de(y)=Ay$, $A\in K^{n\times n}$. Then there exists a finitely $\s$-generated constrained $\s$-pseudo field extension $k'$ of $k$ containing $k_1:=R_1^\de$ and $k_2:=R_2^\de$ such that
$R_1\otimes_{k_1}k'$ and $R_2\otimes_{k_2}k'$ are isomorphic as $K\otimes_k k'$-$\ds$-algebras.

\end{thm}

Note that the assumption ``$K$ is $\s$-separable over $k:=K^\de$'' is automatically satisfied if $k$ is $\s$-closed because a $\s$-closed $\s$-field is inversive. (If $a\in k$ and $b\in k^*$ is an element in the inversive closure $k^*$ (\cite[Def. 2.1.6, p. 109]{Levin}) of $k$ such that $\s(b)=a$ then $k\{b\}_\s$ is a $\s$-domain.)

Since the $k'$ in Theorem \ref{thm:unicity} is a pseudo-field rather than a field, we obtain uniqueness over $\s$-closed $\de$-constants only up to powers of $\s$.

\begin{cor} \label{cor:uniqueness}
Let $K$ be a $\ds$-field such that $K^\de$ is a $\s$-closed $\s$-field. Let $R_1$ and $R_2$ be two $\s$-Picard-Vessiot rings for $\de(y)=Ay$ with $A\in K^{n\times n}$. Then there exists an integer $l\geq 1$ such that $R_1$ and $R_2$ are isomorphic as $K$-$\de\s^l$-algebras.
\end{cor}
\begin{proof}
This is \cite[Corollary 2.21, p. 1394]{Wibmer:Chevalley}. To be precise, the corollary states that there exists an integer $l\geq 1$ and an isomorphism $\psi\colon R_1\to R_2$ of $K$-$\de$-algebras which commutes with $\s^l$, but maybe not with $\s$.
\end{proof}

\begin{cor} \label{cor:uniqueness2}
Let $K$ be a $\ds$-field such that $K^\de$ is a $\s$-closed $\s$-field. Let $L_1$ and $L_2$ be two $\s$-Picard-Vessiot extensions for $\de(y)=Ay$ with $A\in K^{n\times n}$. Then there exists an integer $l\geq 1$ such that $L_1|K$ and $L_2|K$ are isomorphic as $\de\s^l$-field extensions of $K$.
\end{cor}
\begin{proof}
We know from Proposition \ref{prop:pvringpvext} that $L_1$ and $L_2$ are the quotient fields of some $\s$-Picard-Vessiot rings $R_1\subset L_1$ and $R_2\subset L_2$ for $\de(y)=Ay$. The $K$-$\de\s^l$-isomorphism $R_1\to R_2$ which exists by Corollary \ref{cor:uniqueness} extends to an isomorphism of $\ds^l$-field extensions of $K$.
\end{proof}

An example, illustrating that in general it is not possible to choose $l=1$ in the above corollaries can be found in \cite[Example 2.22, p. 1394]{Wibmer:Chevalley}. In the remaining part of this subsection we provide some information on when it is possible to choose $l=1$. To formulate our results we need to recall the notion of compatibility of difference field extensions (\cite[Def. 5.1.1, p. 311]{Levin}):
Two extensions $L_1|K$ and $L_2|K$ of $\s$-fields are called \emph{compatible} if there exists a $\s$-field extension $M|K$ and $K$-$\s$-morphisms $L_1\to M$ and $L_2\to M$. 

\begin{prop} \label{prop:l=1}
Let $K$ be a $\ds$-field such that $K^\de$ is a $\s$-closed $\s$-field. Let $L_1$ and $L_2$ be two $\s$-Picard-Vessiot extensions of $K$ for the same equation $\de(y)=Ay$, $A\in K^{n\times n}$. Then $L_1$ and $L_2$ are isomorphic (as $\ds$-field extensions of $K$) if and only if $L_1$ and $L_2$ are compatible $\s$-field extensions of $K$.
\end{prop}
\begin{proof}
Of course the extensions $L_1|K$ and $L_2|K$ are compatible if they are isomorphic.

Assume that $L_1|K$ and $L_2|K$ are compatible $\s$-field extensions. We have to show that $L_1|K$ and $L_2|K$ are isomorphic.
Let $R_1\subset L_1$ and $R_2\subset L_2$ denote the corresponding $\s$-Picard-Vessiot rings for $\de(y)=Ay$ and set $k:=K^\de$.
From Lemma \ref{lemma:fonda} we know that $R_1\otimes_K R_2=R_1\otimes_k(R_1\otimes_K R_2)^\de$ and that $U:=(R_1\otimes_K R_2)^\de$ is a finitely $\s$-generated $k$-$\s$-algebra.

Because $L_1|K$ and $L_2|K$ are compatible, there exists a $\s$-prime ideal in $L_1\otimes_K L_2$. (If $M$ is a $\s$-field extension of $K$ containing copies of $L_1$ and $L_2$, then the kernel of $L_1\otimes_K L_2\to M$ is a $\s$-prime ideal.) Via the inclusion
\[U\hookrightarrow R_1\otimes_k U=R_1\otimes_K R_2\hookrightarrow L_1\otimes_K L_2\] this $\s$-prime ideal of $L_1\otimes_K L_2$ contracts to a $\s$-prime ideal of $U$. Because
$U$ is finitely $\s$-generated over the $\s$-closed $\s$-field $k$, the existence of a $\s$-prime ideal in $U$ is sufficient to guarantee the existence of a morphism
$\psi\colon U\to k$ of $k$-$\s$-algebras. This yields a morphism
\[\varphi\colon R_2\to R_1\otimes_K R_2=R_1\otimes_k U\xrightarrow{\operatorname{id}\cdot \psi} R_1\] of $K$-$\ds$-algebras. Because $R_2$ is $\de$-simple $\varphi$ is injective, and because
$R_1$ and $R_2$ are $\s$-generated over $K$ by a fundamental solution matrix for the same equation $\de(y)=Ay$ we see that $\varphi$ is surjective. So $\varphi\colon R_2\to R_1$ is an isomorphism. Of course $\varphi$ extends to an isomorphism $L_2\simeq L_1$.
\end{proof}

\begin{cor} \label{cor:l=1}
Let $K$ be a $\ds$-field such that $K^\de$ is a $\s$-closed $\s$-field. Let $L_1$ and $L_2$ be two $\s$-Picard-Vessiot extensions of $K$ for the same equation. Assume that $K$ is relatively algebraically closed in $L_1$. Then $L_1$ and $L_2$ are isomorphic (as $\ds$-field extensions of $K$).
\end{cor}
\begin{proof}
In view of Proposition \ref{prop:l=1}, it suffices to acknowledge that a $\s$-field extension $L|K$ such that $K$ is relatively algebraically closed in $L$ is compatible with any other $\s$-field extension of $K$. This follows for example from \cite[Theorem 5.1.6, p. 313]{Levin}.
\end{proof}

\begin{rem}
The condition ``$K$ is relatively algebraically closed in $L_1$'' in Corollary \ref{cor:l=1} can be weakened to ``The core of $L_1|K$ is equal to $K$''. This means that every \emph{finite} $\s$-field extension of $K$ inside $L_1$ is equal to $K$. (In general, the relative algebraic closure of $K$ in $L_1$ is of infinite degree over $K$ and may or may not contain finite $\s$-field extensions of $K$.)
\end{rem}
\begin{proof}
This follows from the classical compatibility theorem \cite[Theorem 5.4.22, p. 342]{Levin}.
\end{proof}

\section{The $\s$-Galois group of a linear differential equation} \label{sec: The sGalois group of a linear differential equation}

In this section, we introduce the $\s$-Galois group of a linear differential equation $\de(y)=Ay$ over a $\ds$-field $K$. More precisely, we will define the $\s$-Galois group of a fixed $\s$-Picard-Vessiot extension for $\de(y)=Ay$. It is a $\s$-algebraic group over $K^\de$. We show that the Zariski closure of the $\s$-Galois group is the classical Galois group of $\de(y)=Ay$. We also explain the significance of the higher order Zariski closures of the $\s$-Galois group and show that the $\s$-transcendence degree of a $\s$-Picard-Vessiot extension equals the $\s$-dimension of its $\s$-Galois group.

For a brief introduction to $\s$-algebraic groups, we refer the reader
to the appendix. Here we only recall the definition.

\begin{defn} Let $k$ be a $\s$-field. A \emph{$\s$-algebraic group over $k$} is a (covariant) functor $G$ from the category of $k$-$\s$-algebras to the category of groups which is representable by a finitely $\s$-generated $k$-$\s$-algebra.
I.e., there exists a finitely $\s$-generated $k$-$\s$-algebra $k\{G\}$ such that
\[G\simeq\Alg_k^\s(k\{G\},-).\]
\end{defn}
Here $\Alg_k^\s$ stands for morphisms of $k$-$\s$-algebras. By the Yoneda lemma $k\{G\}$ is unique up to isomorphisms.
If $R\subset S$ is an inclusion of $\ds$-rings, we denote by $\Aut^\ds(S|R)$ the automorphisms of $S$ over $R$ in the category of $\ds$-rings, i.e., the automorphisms are required to be the identity on $R$ and to commute with $\de$ and $\s$.

%
%
%
%

\begin{defn}
Let $L|K$ be a $\s$-Picard-Vessiot extension with $\s$-Picard-Vessiot ring $R\subset L$. Set $k=K^\de$. We define $\sgal(L|K)$ to be the functor from the category of $k$-$\s$-algebras to the category of groups
given by
\[\sgal(L|K)(S):=\Aut^\ds(R \otimes_{k} S | K \otimes_{k} S)\]
for every $k$-$\s$-algebra $S$. The action of $\de$ on $S$ is trivial, i.e., $\de(r\otimes s)=\de(r)\otimes s$ for $r\in R$ and $s\in S$. We call $\sgal(L|K)$ the \emph{$\s$-Galois group of $L|K$}.
\end{defn}
On morphisms $\sgal(L|K)$ is given by base extension: If $\psi\colon S\to S'$ is a morphism of $k$-$\s$-algebras, then 
$(\sgal(L|K))(\psi)\colon\sgal(L|K)(S)\to \sgal(L|K)(S')$ is the morphism of groups which associates to a $K\otimes_k S$-$\ds$-automorphims $\tau\colon R\otimes_kS\to R\otimes_k S$ the $K\otimes_k S'$-$\ds$-automorphims
$R\otimes_k S'=(R\otimes_k S)\otimes_S S'\xrightarrow{\tau\otimes\operatorname{id}} (R\otimes_k S)\otimes_S S'=R\otimes_kS'$.

Note that $\sgal(L|K)(k)=\Aut^\ds(R|K)=\Aut^\ds(L|K)$. To show that $\sgal(L|K)$ is a $\s$-algebraic group we shall need two simple lemmas.
\begin{lemma} \label{lemma:simple}
Let $R$ be a $\de$-simple $\de$-ring, $k:=R^\de$ and $S$ a $k$-algebra, considered as a constant $\de$-algebra. Then $(R\otimes_k S)^\de=S$ and the assignments
$\ida\mapsto R\otimes_k \ida$ and $\idb\mapsto S\cap\idb$ define mutually inverse bijections between the set of ideals of $S$ and the set of $\de$-ideals of $R\otimes_k S$.
In particular, every $\de$-ideal $\idb$ of $R\otimes_k S$ is generated by $\idb\cap S$ as an ideal.
\end{lemma}
\begin{proof}
The first claim follows immediately when choosing appropriate bases. The second claim follows as in \cite[Prop. 5.6, p. 4484]{Kovacic:differentialgaloistheoryofstronglynormal}. See also \cite[Lemma 10.7, p. 5443]{MauGal}.
\end{proof}

\begin{lemma} \label{lemma:endo=auto}
Let $L|K$ be a $\s$-Picard-Vessiot extension with $\s$-Picard-Vessiot ring $R\subset L$ and $\de$-constants $k$. If $S$ is a $k$-$\s$-algebra, then every
$K\otimes_k S$-$\ds$-endomorphism of $R\otimes_k S$ is an automorphism.
\end{lemma}

\begin{proof}
Fix matrices $A\in K^{n\times n}$ and $Y\in\Gl_n(L)$ such that $L|K$ is a $\s$-Picard-Vessiot extension for $\de(y)=Ay$ with fundamental solution matrix $Y$. If $\tau\colon R\otimes_k S\rightarrow R\otimes_k S$ is a $K\otimes_k S$-$\ds$-morphism then $\tau(Y)\in\Gl_n(R\otimes_k S)$ is a fundamental solution matrix for $\de(y)=Ay$. Since also $Y(=Y\otimes 1)\in\Gl_n(R\otimes_k S)$ is a fundamental solution matrix, there exists a (unique) matrix $[\tau]_Y\in\Gl_n((R\otimes_k S)^\de)=\Gl_n(S)$ such that $\tau(Y)=Y[\tau]_Y$. Because $R$ is $\s$-generated by $Y$, it follows that $\tau$ is surjective. The kernel of $\tau$ is a $\de$-ideal of $R\otimes_k S$. It follows from Lemma \ref{lemma:simple} that $\tau$ is injective.
\end{proof}

\begin{prop}\label{prop:defgal} Let $L|K$ be a $\s$-Picard-Vessiot extension with $\s$-Picard-Vessiot ring $R\subset L$. Then $\sgal(L|K)$ is a $\s$-algebraic group over $k=K^\de$. More precisely, $\sgal(L|K)$ is represented by the finitely $\s$-generated $k$-$\s$-algebra $(R \otimes_{K} R)^\de$. The choice of matrices $A\in K^{n\times n}$ and $Y\in\Gl_n(L)$ such that $L|K$ is a $\s$-Picard-Vessiot extension for $\de(y)=Ay$ with fundamental solution matrix $Y$ defines a $\s$-closed embedding
\[\sgal(L|K)\hookrightarrow\Gl_{n,k} \] of $\s$-algebraic groups.
\end{prop}
\begin{proof}
Let $S$ be a $k$-$\s$-algebra. From Lemma \ref{lemma:endo=auto} and Lemma \ref{lemma:fonda}, we obtain the following chain of identifications:

\begin{equation} \label{eq:identifications}
\begin{split}
\sgal(L|K)(S) & =\Alg_{K\otimes_k S}^\ds(R\otimes_k S,R\otimes_k S)=\Alg_K^\ds(R,R\otimes_k S)= \\
 & =\Alg_R^\ds(R\otimes_K R, R\otimes_k S)=\Alg_R^\ds(R\otimes_k(R\otimes_KR)^\de,R\otimes_kS)= \\
& =\Alg_k^\s((R\otimes_K R)^\de,S).
\end{split}
\end{equation}


The last identity holds because every $R$-$\ds$-morphism $R\otimes_k(R\otimes_KR)^\de\to R\otimes_kS$ restricts to a $k$-$\s$-morphism $(R\otimes_KR)^\de\to S$ by taking $\de$-constants.

We have already seen in Lemma \ref{lemma:fonda} that $(R\otimes_K R)^\de=k\{Z,\frac{1}{\det(Z)}\}_\s$ where $Z=(Y\otimes 1)^{-1}(1\otimes Y)\in\Gl_n(R\otimes_K R)$. So
$(R\otimes_KR)^\de$ is finitely $\s$-generated over $k$.

As in the proof of Lemma \ref{lemma:endo=auto}, every automorphism $\tau\in\Gal(L|K)(S)$ is given by a matrix $[\tau]_Y\in\Gl_n(S)=\Gl_n(1\otimes S)$ such that $\tau(Y)=Y[\tau]_Y$.
Then $\tau$ is given in $\Alg_R^\ds(R\otimes_K R, R\otimes_k S)$ by $\tau(1\otimes Y)=(Y\otimes 1)[\tau]_Y$. So, as an element of $\Alg_R^\ds(R\otimes_k(R\otimes_KR)^\de,R\otimes_kS)$ the morphism $\tau$ is given by $\tau(Z)=(Y\otimes 1)^{-1}\tau(1\otimes Y)=[\tau]_Y$.
In summary, we see that, under the identification $\sgal(L|K)(S)=\Alg_k^\s((R\otimes_K R)^\de,S)$, an automorphism $\tau\in\sgal(L|K)(S)$ corresponds to the $k$-$\s$-morphism
$(R\otimes_K R)^\de\to S$ determined by $Z\mapsto [\tau]_Y$. It follows that the identification $\sgal(L|K)(S)=\Alg_k^\s((R\otimes_K R)^\de,S)$ is functorial in $S$. So $\sgal(L|K)$ is represented by $(R\otimes_K R)^\de$.

Moreover, the assignment $\tau\mapsto [\tau]_Y$ defines an injection $\sgal(L|K)(S)\hookrightarrow \Gl_{n,k}(S)$. This yields an embedding of functors
$\sgal(L|K)\hookrightarrow \Gl_{n,k}$. If we set $S:=(R\otimes_K R)^\de=k\{Z\frac{1}{\det(Z)}\}_\s$ the automorphism $\tau_{\operatorname{univ}}\in\sgal(L|K)(S)$ corresponding to $\operatorname{id}\in\Alg_k^\s((R\otimes_K R)^\de, S)$ is given by $\tau_{\operatorname{univ}}(Y\otimes 1)=(Y\otimes 1)Z$, i.e., $[\tau_{\operatorname{univ}}]_Y=Z$.
Thus, the dual morphism of $\sgal(L|K)\hookrightarrow \Gl_{n,k}$ is given by
\[k\{\Gl_{n,k}\}=k\{X,\tfrac{1}{\det(X)}\}_\s\longrightarrow (R\otimes_K R)^\de=k\{Z,\tfrac{1}{\det(Z)}\}_\s, \ \ X\mapsto Z.\] This is clearly surjective.
So $\sgal(L|K)\hookrightarrow \Gl_{n,k}$ is a $\s$-closed embedding. (See Definition \ref{defi:sigmaclosed}.)
\end{proof}

\begin{defi}
Let $L|K$ be a $\s$-Picard-Vessiot extension with $\s$-Picard-Vessiot ring $R$, $\s$-field of $\de$-constants $k$, and $\s$-Galois group $G:=\sgal(L|K)$. We set
\[k\{G\}=(R\otimes_K R)^\de.\]
By the above proposition $k\{G\}$ is a finitely $\s$-generated $k$-$\s$-algebra representing $G$.
\end{defi}

From Lemma \ref{lemma:fonda} we immediately obtain the algebraic recast of the torsor theorem:
\begin{lemma}
Let $L|K$ be a $\s$-Picard-Vessiot extension with $\s$-Picard-Vessiot ring $R$, $\s$-field of $\de$-constants $k$, and $\s$-Galois group $G$.
Then
\[R\otimes_K R=R\otimes_kk\{G\}.\]\qed
\end{lemma}

We will next explain the relation between our $\s$-Galois group and the Galois group in the sense of Takeuchi.
Let $k$ be a $\s$-field and $X=\Alg_k^\s(k\{X\},-)$ a $k$-$\s$-scheme. (See Definition \ref{defi:ksigmaScheme}.) The $k$-scheme obtained from $X$ by forgetting $\s$ is
$X^\sharp=\Alg_k(k\{X\}^\sharp,-)$. (See Section \ref{subsec:ksigmaSchemeAssociatedwithkScheme}.)

Let $L|K$ be a $\s$-Picard-Vessiot extension with $\s$-Picard-Vessiot ring $R$, $\s$-field of $\de$-constants $k$, and $\s$-Galois group $G$. Let $T$ be a $k$-algebra.
Replacing $S$ by $T$ and forgetting $\s$ in the identifications of equation (\ref{eq:identifications}) above, we see that
\[G^\sharp(T)=\Aut^\de(R\otimes_kT|K\otimes_k T).\]

\begin{rem} \label{rem:compareTakeuchigroup}
Let $L|K$ be a $\s$-Picard-Vessiot extension with $\s$-Galois group $G$. We already noted in Remark \ref{rem:compareTakeuchi} that $L^\sharp|K^\sharp$ is a Picard-Vessiot extension in the sense of Takeuchi. The Galois group of $L^\sharp|K^\sharp$ (in the sense of \cite{Takeuchi:hopfalgebraicapproach}) agrees with $G^\sharp$, the group scheme obtained from $G$ by forgetting $\sigma$.
\end{rem}
\begin{proof}
Let $R\subset L$ denote the $\s$-Picard-Vessiot ring of $L|K$. Initially the Galois group of $L^\sharp|K^\sharp$, or more precisely its representing Hopf algebra $D$, is defined by some abstract algebraic manipulations with corings. See \cite[Section 1]{AmanoMasuokaTakeuchi:HopfPVtheory}. As a $k$-algebra $D=(R\otimes_kR)^\de=k\{G\}^\sharp$. Moreover in \cite[Appendix, Theorem A.2, p. 508]{Takeuchi:hopfalgebraicapproach}, it is shown that $D$ represents the automorphism functor
\[T\mapsto\Aut^\de(R\otimes_k T|K\otimes_k T).\]
\end{proof}


To get a better feeling for what is really going on, let us compute the $\s$-Galois group in some simple examples, including the ones given in the introduction.

\begin{exa} \label{exa:exponential}
Let $K=\C(x)$ be the $\ds$-field of rational functions in the variable $x$ over $\C$, where $\de=\frac{d}{d x}$ and $\s(f(x))=f(x+1)$. So $k=K^\de=\C$, with $\s$ the identity map. Consider the equation $$\de(y)=2xy$$ over $K$. The field $\mathcal{M}$ of meromorphic functions on $\C$ is naturally a $\ds$-field extensions of $K$ (with $\de=\frac{d}{d x}$ and $\s\colon f(x)\mapsto f(x+1)$). Since $\mathcal{M}^\de=k$, it is clear that
$L=K\langle e^{x^2}\rangle_\s\subset\mathcal{M}$ is a $\s$-Picard-Vessiot extension for $\de(y)=2xy$. The $\s$-Picard-Vessiot ring is $R=K\{e^{x^2},e^{-x^2}\}_\s\subset L$.
Since $x^2-2(x+1)^2+(x+2)^2=2$, we have
$$f\s(f)^{-2}\s^2(f)=e^2\in K,$$
where we have set $f=e^{x^2}$ to simplify the notation.
So $R=K\big[e^{x^2},e^{-x^2},e^{(x+1)^2},e^{-(x+1)^2}\big]$. Let $G=\sgal(L|K)$ denote the $\s$-Galois group. We consider $G$ as a $\s$-closed subgroup of $\Gl_{1,k}$ via the fundamental solution matrix $Y=e^{x^2}\in\Gl_1(L)$. Let $S$ be a $k$-$\s$-algebra. For $g\in G(S)\subset S^\times$ we have
$$e^2=g(e^2)=g(f)\s(g(f))^{-2}\s^2(g(f))=g\s(g)^{-2}\s^2(g)e^2.$$
(Note that the above computation takes place in $R\otimes_k S$.) Therefore $g\s(g)^{-2}\s^2(g)=1$.
On the other side, the functions $e^{x^2}$ and $e^{(x+1)^2}$ are algebraically independent over $K$. (This follows for example from the Kolchin-Ostrowski theorem.) So for any $g\in S^\times$, satisfying $g\s(g)^{-2}\s^2(g)=1$, we have a well-defined $K\otimes_k S$-$\ds$-automorphism of $R\otimes_k S$ determined by $e^{x^2}\mapsto ge^{x^2}$.

In summary, $G\leq\Gl_{1,k}$ is given by
$$G(S)=\{g\in S^\times|\ g\s(g)^{-2}\s^2(g)=1\}\leq\Gl_1(S)$$
for any $k$-$\s$-algebra $S$.
\end{exa}

\begin{exa} \label{exa:benign}
As in the above example, let $K=\C(x)$ be the $\ds$-field of rational functions in the variable $x$ over $\C$, where $\de=\frac{d}{d x}$ and $\s(f(x))=f(x+1)$. Consider the equation $$\de(y)=\frac{1}{2x}y$$ over $K$.

Fix an algebraic closure $\overline{K}$ of $K$. Then $\de$ extends uniquely to $\overline{K}$. We can also extend $\s$ to $\overline{K}$. The extension of $\s$ to $\overline{K}$ is of course not unique but one can show that any two extensions are isomorphic. In particular, the $\s$-Galois group will be independent of this choice. The derivation $\de$ and the endomorphisms $\s$ also commute on $\overline{K}$. I.e., $\overline{K}$ is a $\ds$-extension of $K$. Obviously $\sqrt{x}\in \overline{K}$ is a fundamental solution matrix for $\de(y)=\frac{1}{2x}y$. Since $\overline{K}^\de=K^\de=\C=:k$ it is clear that
$$L:=K\langle \sqrt{x}\rangle_\s=K(\sqrt{x},\sqrt{x+1},\ldots)\subset \overline{K}$$ is a $\s$-Picard-Vessiot extension for $\de(y)=\frac{1}{2x}y$. The $\s$-Picard-Vessiot ring $R$ is equal to $L$. Let $G=\sgal(L|K)$ denote the $\s$-Galois group.
Since a $\ds$-automorphism of $R|K$ is determined by its action on $\sqrt{x}$ we see that
$$G(k)=\Aut^\ds(R|K)=C_2,$$ where
$C_2=\{1,-1\}$ is the group with two elements acting on $R$ by $\sqrt{x}\mapsto -\sqrt{x}$.
On the other hand, as the degree of $K(\sqrt{x},\ldots,\sqrt{x+i})$ over $K$ is $2^{i+1}$ for $i\geq 0$, we see that
$$G^\sharp(k^\sharp)=\Aut^\de(R|K)=C_2^\infty=C_2\times C_2\times\cdots.$$
As a $\s$-closed subgroup of $\Gl_{1,k}$ the $\s$-Galois group is given by
$$G(S)=\{g\in S^\times|\ g^2=1\}\leq \Gl_1(S)$$
for every $k$-$\s$-algebra $S$. An elements $g\in G(S)$ is acting on $R\otimes_k S$ by
$\sqrt{x+i}\mapsto \s^i(g)\sqrt{x+i}$ for $i\geq 0$.
\end{exa}

%

\begin{exa}
Let $K=\mathbb{C}(x)$ denote the field of rational functions in the variable $x$ over $\C$. We consider $K$ as $\ds$-field with derivation $\de=x\frac{d}{dx}$ and endomorphism $\s$, given by $\s(f(x))=f(x^d)$ for some integer $d\geq 2$. So $\hslash=d$ and $k:=K^\de=\C$. Let us consider the equation
$$\de(y)=\tfrac{1}{d}y$$
over $K$. It has the solution $x^\frac{1}{d}$. The field $L=\C(x^\frac{1}{d})$ is naturally a $\ds$-field extension of $K$ with
$\s(x^\frac{1}{d})=x$. Since $L^\de=\C$, it is clear that $L|K$ is a $\s$-Picard-Vessiot extension for $x^\frac{1}{d}$. The $\s$-Picard-Vessiot ring $R$ equals $L$. We consider the $\s$-Galois group $G=\sgal(L|K)$ as a $\s$-closed subgroup of $\Gl_{1,k}$.
For any $k$-$\s$-algebra $S$ and every $g\in G(S)\leq\Gl_1(S)$ we have $\s(g(x^\frac{1}{d}))=\s(g)x$. On the other hand,
$\s(g(x^\frac{1}{d}))=g(\s(x^\frac{1}{d}))=g(x)=x$. Therefore $\s(g)=1$. Since $(x^{\frac{1}{d}})^d\in K$ we easily see that $g^d=1$ and that the $\s$-Galois group is given by
$$G(S)=\{g\in S^\times|\ g^d=1,\ \s(g)=1\}\leq\Gl_1(S)$$
for any $k$-$\s$-algebra $S$. Note that $G$ is not $\s$-reduced, i.e., $\s$ is not injective on the $\s$-coordinate ring $k\{G\}$. Indeed, the image of $x-1$ in $k\{G\}=k\{x,x^{-1}\}_\s/[x^d-1, \s(x)-1]$ is a non-zero element in the kernel of $\s$.
\end{exa}

\begin{exa}
Let us consider Bessel's differential equation
$$x^2\de^2(y)+x\de(y)+(x^2-\alpha^2)y=0.$$
The matrix of the equivalent system is
$$
A =\begin{pmatrix} 0& 1 \\
\frac{\alpha^2}{x^2} -1 & \frac{-1}{x} \end{pmatrix}.
$$
As demonstrated in Proposition \ref{prop:taylorexp}, there exists a $\s$-Picard-Vessiot extension for $\de(y)=Ay$ over $\C(\alpha,x)$, where $\de=\frac{d}{d x}$ and $\s(f(\a,x))=f(\a+1,x)$. However, since the classical solutions $J_\alpha(x)$ and $Y_\alpha(x)$ are normalized by some factor meromorphic in $\alpha$, it is more convenient to work with meromorphic rather than rational functions in $\alpha$. Let $\mathcal{M}$ denote the field of meromorphic functions on $\{\alpha\in\C|\ \operatorname{Re}(\a)>0\}$. We consider $\mathcal{M}$ as $\s$-field by $\s(f(\alpha))=f(\alpha+1)$. Let $K=\mathcal{M}(x)$ denote the $\ds$-field of rational functions in $x$ over $\mathcal{M}$ with derivation $\de=\frac{d}{d x}$ and endomorphism $\s\colon K\to K$, extended from $\mathcal{M}$ by $\s(x)=x$. Then $k=K^\de=\mathcal{M}$.
As an ambient $\ds$-field, containing the Bessel function of the first kind $J_\a(x)$ and the Bessel function of the second kind $Y_a(x)$ we can, for example, choose the field $E=\mathcal{M}((x-1))$ of formal Laurent series in $x-1$ with coefficients in $\mathcal{M}$, where, as before, $\de=\frac{d}{d x}$ and $\s(\sum a_i(x-1)^i)=\sum \s(a_i)(x-1)^i$. For generalities on Bessel functions we refer the reader to \cite{Watson:ATreatiseOnTheTheoryOfBesselFunctions}.
The matrix
$$Y=\begin{pmatrix} J_\alpha(x) & Y_\alpha(x) \\
\de(J_\a(x)) & \de(Y_\a(x))\end{pmatrix}\in\Gl_n(E) $$
is a fundamental solution matrix for $\de(y)=Ay$. Since $E^\de=k$, we see that $L=K\langle Y\rangle_\s\subset E$ is a $\s$-Picard-Vessiot extension for $\de(y)=Ay$ over $K$. The recurrence formulas 
$$
\frac{2\alpha}{x} Z_\alpha(x) = Z_{\alpha-1}(x) + Z_{\alpha+1}(x)$$
and
 $$2\de(Z_\alpha(x)) = Z_{\alpha-1}(x) - Z_{\alpha+1}(x),$$
satisfied by $J_\a(x)$ and $Y_\a(x)$ can be rewritten in matrix form as $\s(Y)=BY$, where 
$$
B=
\begin{pmatrix}
\frac{\a}{x} & -1 \\
\frac{-\a(\a +1)}{x^2} +1 & \frac{\a +1}{x}
\end{pmatrix}\in\Gl_n(K).
$$
We consider the $\s$-Galois group $G=\sgal(L|K)$ as $\s$-closed subgroup of $\Gl_{2,k}$ via the fundamental solution matrix $Y$.
Let $S$ be a $k$-$\s$-algebra and $g\in G(S)\leq\Gl_2(S)$.
We have
$$g(\s(Y))=g(BY)=Bg(Y)=BYg.$$ On the other hand,
$$g(\s(Y))=\s(g(Y))=\s(Yg)=BY\s(g).$$ So $\s(g)=g$. The functions $J_\a(x), Y_\a(x), \de(J_\a(x)), \de(Y_\a(x))$ are not algebraically independent. Indeed, the Wronskian of $J_\a(x), Y_\a(x)$ equals $\frac{2}{\pi x}$. In particular, $\det(Y)\in K$.
It follows that $$\det(Y)=g(\det(Y))=\det(g(Y))=\det(Yg)=\det(Y)\det(g).$$ So $\det(g)=1$.

For a fixed $\alpha\in\C$, with $\alpha-\frac{1}{2}\notin\Z$, the (classical) Galois group of Bessel's equation over $\C(x)$
is $\Sl_{2,\C}$ (\cite[Appendix]{Kolchin:AlgebraicGroupsAndAlgebraicDependence}). Roughly speaking, this means that $\det(Y)\in \C(x)$ is the ``only'' algebraic relation among $J_\a(x), Y_\a(x), \de(J_\a(x)), \de(Y_\a(x))$ over $\C(x)$.

This implies that $\det(Y)\in K$ is the ``only'' algebraic relation among $J_\a(x), Y_\a(x), \de(J_\a(x)), \de(Y_\a(x))$ over $K$. Because $R=K\{Y,\frac{1}{\det(Y)}\}_\s=K[Y,\frac{1}{\det(Y)}]$ we find that
$G\leq\Gl_{2,k}$ is given by
$$G(S)=\{g\in\Sl_2(S)|\ \s(g)=g\},$$
for any $k$-$\s$-algebra $S$.
\end{exa}

\begin{exa}
Let $q$ be a complex number of norm greater than $1$.
We consider the Jacobi Theta function
$$
\theta_q(x)=\sum_{n\in\Z}q^{-n(n-1)/2}x^n
$$
and its logarithmic derivative $\ell_q(x)=\frac{\de(\theta_q(x))}{\theta_q(x)}$, where $\de=x\frac{d}{dx}$.
Since $|q|>1$, the formal series $\theta_q$ naturally defines a meromorphic function on $\C^*:=\C\smallsetminus\{0\}$ 
and satisfies the $q$-difference equation
$$
\theta_q(qx)=qx\theta_q(x),
$$
so that $\ell_q(qx)=\ell_q(x)+1$.
This implies that $\s(\de\l(\ell_q\r))=\de\l(\ell_q\r)$, where $\s$ is the $q$-difference operator $f(x)\mapsto f(qx)$.
We want to give an interpretation of these classical formulas in the present framework.

\par
Inside the $\ds$-field $\cM$ of meromorphic function on $\C^*$, we consider the $\ds$-subfield $K:=\cM^\s$ of $q$-elliptic functions
and the differential equation
\begin{equation}\label{eqn:equationlq}
\de(y)=\de(\ell_q) 
\end{equation}
with coefficients in $K$.
Since the $\de\s$-field $L:=K(\ell_q)$ is contained in $\cM$, we deduce that $L^{\de}=K^{\de}=\C=:k$. This 
means that $L|K$ is a $\s$-Picard-Vessiot extension for the
system
$$\de\begin{pmatrix}y_1 \\ y_2 \end{pmatrix}=\begin{pmatrix} 0 & \de(\ell_q) \\
0 & 0 \end{pmatrix}\begin{pmatrix}y_1 \\ y_2 \end{pmatrix}$$
associated with (\ref{eqn:equationlq}).
The $\s$-Picard-Vessiot ring is $R:=K[\ell_q]$ and the $\s$-Galois group $G:=\sgal(L|K)$ is naturally contained in the additive group $\bold{G}_{a,k}$.
For any $k$-$\s$-algebra $S$ and any $g\in G(S)\leq\bold{G}_a(S)$ we have
$$g(\s(\ell_q))=\s(g(\ell_q))=\s(\ell_q+g)=\ell_q+1+\s(g).$$
On the other hand, $g(\s(\ell_q))=\ell_q+1+g$ and consequently $\s(g)=g$.
Since $\ell_q$ does not belong to $K$, and therefore is transcendental over $K$, we see that
$$G(S)=\{g\in S|\ \s(g)=g\}\leq\bold{G}_a(S).$$
\end{exa}

\begin{exa} \label{exa:padic}
We go back to the $p$-adic example of the Dwork exponential already presented in the introduction.
See also Example \ref{exa:dsfields} for the notation. We assume that there exists $\pi\in k$ such that $\pi^{p-1}=-p$.
Our base $\ds$-field is $K:=\cE_{k}^\dag$.
We have already pointed out that
$\exp(\pi x)\notin K$ and that $L:=K(\exp(\pi x))$ is a $\ds$-field. In fact, since $L^\de=K^\de=k$, is is clear that $L|K$ is a $\s$-Picard-Vessiot extension for $$\de(y)=\pi x y.$$
The $\s$-Picard-Vessiot ring is $R:=K[\exp(\pi x),\exp(\pi x)^{-1}]$.
Let us consider the $\s$-Galois group $G:=\sgal(L|K)$ as a $\s$-closed subgroup of $\Gl_{1,k}$.
Notice that $\exp(\pi x)^p=\exp(p\pi x)\in K$. This implies that
$g^p=1$ for every $g\in G(S)\leq\Gl_1(S)$ and any $k$-$\s$-algebra $S$.
Moreover, since $\s(\exp(\pi x))\exp(\pi x)^{-1}\in K$, we also find that $\s(g)=g$.
As $\exp(\pi x)\notin K$, it is now easy to see that
$$G(S)=\{g\in S^\times|\ g^p=1,\ \s(g)=g\}\leq\Gl_1(S).$$
%
\end{exa}

The fact observed in Examples \ref{exa:exponential}, \ref{exa:benign} and \ref{exa:padic}, that the $\s$-algebraic relations satisfied by the solution of a first order linear differential equations $\de(y)=ay$ can be described by $\s$-monomials, is a general pattern which can be derived from the classification of the $\s$-closed subgroups of the multiplicative group (Lemma \ref{lemma:classgm}). See \cite{articletwo} for more details.

We continue by describing the relation between the $\s$-Galois group and the classical Galois group.


\begin{prop} \label{prop:Zariskiclosures}
Let $L|K$ be a $\s$-Picard-Vessiot extension with $\s$-field of $\de$-constants $k=K^\de$. Let $A\in K^{n\times n}$ and $Y\in\Gl_n(L)$ such that $L|K$ is a $\s$-Picard-Vessiot extension for $\de(y)=Ay$ with fundamental solution matrix $Y$. We consider the $\s$-Galois group $G$ of $L|K$ as a $\s$-closed subgroup of $\Gl_{n,k}$ via the embedding associated with the choice of $A$ and $Y$. For $d\geq0$, set $L_d=K\left(Y,\s(Y),\ldots,\s^d(Y)\right)\subset L$.

Then $L_d|K$ is a (classical) Picard-Vessiot extension for the linear system $\de(y)=A_dy$, where $A_d\in\Gl_{n(d+1)}(K)$ is
defined in equation (\ref{eqn:Ad}). 

The (classical) Galois group of $L_d|K$ is naturally isomorphic to $G[d]$, the $d$-th order Zariski closure of $G$ inside $\Gl_{n,k}$. (See Definition \ref{defi:Zariskiclosures}.) In particular, $L_0=K(Y)$ is a (classical) Picard-Vessiot extension for $\de(y)=Ay$ and the Zariski closure of $G$ inside $\Gl_{n,k}$ is the (classical) Galois group of $\de(y)=Ay$.
\end{prop}
\begin{proof}
By Lemma \ref{lemma: PVjets}, $L_d|K$ is a classical Picard-Vessiot extension for $\de(y)=A_dy$  and  for $d\geq 0$, the differential ring
\[R_d=K\left[Y,\s(Y),\ldots,\s^d(Y),\tfrac{1}{\det(Y\cdots\s^d(Y))}\right]\subset L_d\] is  the Picard-Vessiot ring of $L_d|K$. By Proposition \ref{prop:pvringpvext}, the $K$-$\ds$-algebra  $R=K\{Y,\tfrac{1}{\det(Y)}\}_\s\subset L$ is the $\s$-Picard-Vessiot ring of $L|K$.
We denote the Galois group of $L_d|K$ with $G_d$. So
\[G_d(T)=\Aut^\de(R_d\otimes_k T| K\otimes_k T)\]
for every $k$-algebra $T$. Since an automorphism
$\tau\in G^\sharp(T)=\Aut^\de(R\otimes_k T|K\otimes_k T)$ restricts to an automorphism $\widetilde{\tau}\in G_d(T)=\Aut^\de(R_d\otimes_kT|K\otimes_k T)$, we obtain a morphism
$G^\sharp\to G_d$ of group $k$-schemes. Because of the special shape of $Y_d$, we see that $G_d$ is a closed subscheme of $\Gl_{n,k}\times\cdots\times\Gl_{n,k}=(\Gl_{n,k})_d$. (See Section \ref{sec:ZariskiClosures} for an explanation of this notation.)
But, by definition, $G[d]$ is the smallest closed subscheme of $(\Gl_{n,k})_d$ such that $G^\sharp\to(\Gl_{n,k})_d$ factors through $G[d]\hookrightarrow (\Gl_{n,k})_d$. Thus $G[d]\subset G_d$.
The image of $G^\sharp\to G_d$ is a subfunctor of $G_d$ (not closed in general) contained in $G[d]$. Since every element of $L$ invariant under $G^\sharp$ must lie in $K$, we see that every element of $L_d$ invariant under $G[d]$ must lie in $K$. But then it follows from the Galois correspondence for $L_d|K$ that $G[d]=G_d$.
\end{proof}

We will finish this subsection by showing that the $\s$-transcendence degree of a $\s$-Picard-Vessiot extension equals the $\s$-dimension of its $\s$-Galois group. But let's first recall the definition of $\s$-transcendence degree.

\begin{defn}[Definition 4.1.7 in \cite{Levin}] 
Let $L|K$ be a $\s$-field extension. Elements $a_1,\ldots,a_n\in L$ are called \emph{transformally (or $\s$-algebraically) independent over $K$} if the elements $a_1,\ldots,a_n,\s(a_1),\ldots,\s(a_n),\ldots$ are algebraically independent over $K$. Otherwise, they are called transformally dependent over $K$. A \emph{$\s$-transcendence basis of $L$ over $K$} is a maximal transformally independent over $K$ subset of $L$. Any two $\s$-transcendence bases of $L|K$ have the same cardinality and so we can define the \emph{$\s$-transcendence degree of $L|K$}, or 
$\strdeg(L|K)$ for short,
as the cardinality of any $\s$-transcendence basis of $L$ over $K$.
\end{defn}
The definition of the $\s$-dimension $\sdim_k(G)$ of a $\s$-algebraic group $G$ over a $\s$-field $k$ is given in Section \ref{subsec:sdimension} of the appendix.

\begin{prop}\label{prop:dimensiondegtrans}
Let $L|K$ be a $\s$-Picard-Vessiot extension with $\s$-Galois group $G$ and constant field $k=K^\de$.  Then
\[\strdeg(L|K) = \sdim_k(G).\]
\end{prop}
\begin{proof}
Let $R\subset L$ denote the corresponding $\s$-Picard-Vessiot ring.
Then $R\otimes_K R=R\otimes_k k\{G\}$. Therefore $L\otimes_K R=L\otimes_kk\{G\}$. It follows from Lemma \ref{lemma: sdim stable under base extension} that
\[\sdim_k(G)=\sdim_L(L\otimes_kk\{G\})=\sdim_L(L\otimes_K R)=\sdim_K(R).\]
By Lemma \ref{lemma: sdim equals stranscendence degree}
\[\sdim_K(R)=\strdeg(L|K).\]

\end{proof}

\section{The $\s$-Galois correspondence} \label{sec: The SGalois correspondence}

In this section, we will establish the $\s$-versions of the first and second fundamental theorem of Galois theory.

\bigskip

Let $L|K$ be a $\s$-Picard-Vessiot extension with $\s$-Picard-Vessiot ring $R\subset L$, $k:=K^\de$, the $\s$-field of $\de$-constants and $G:=\sgal(L|K)$, the $\s$-Galois group of $L|K$.
Let $S$ be a $k$-$\s$-algebra, $\tau\in G(S)$ and $a\in L$. We follow \cite{dyckdesc} and \cite{MauGal} in giving meaning to the phrase ``$a$ is invariant under $\tau$''. By definition, $\tau$ is an automorphism of $R\otimes_k S$. The total quotientring $\quot(R\otimes_k S)$ contains $L$. It might not be possible to extend $\s$ to $\quot(R\otimes_k S)$, but in any case $\tau$ extends to an automorphism of rings $\tau\colon \quot(R\otimes_k S)\to \quot(R\otimes_k S)$ and it makes sense to say that $a$ is invariant under $\tau$. If we write $a=\frac{r_1}{r_2}$ with $r_1,r_2\in R$, $r_2\neq 0$ then $a$ is invariant under $\tau$ if and only if $\tau(r_1\otimes 1)\cdot r_2\otimes 1=r_1\otimes 1\cdot \tau(r_2\otimes 1)\in R\otimes_k S$.

If $H$ is a subfunctor of $G$, we say that $a\in L$ is invariant under $H$ if $a$ is invariant under every element of $H(S)\subset G(S)$ for every $k$-$\s$-algebra $S$.
The set of all elements in $L$, invariant under $H$ is denoted with $L^H$. Obviously $L^H$ is an intermediate $\ds$-field of $L|K$.

If $M$ is an intermediate $\ds$-field of $L|K$, then it is immediately clear from Definition \ref{defi:PVextPVring} that $L|M$ is a $\s$-Picard-Vessiot extension with $\s$-Picard-Vessiot ring $MR$, the ring compositum of $M$ and $R$ inside $L$. Let $S$ be a $k$-$\s$-algebra. When we fix a fundamental solution matrix $Y\in\Gl_n(L)$, then an $M\otimes_k S$-$\ds$-automorphism of $MR\otimes_k S$ is given by a matrix $[\tau]_Y\in\Gl_n(S)$. It follows that $\tau$ restricts to a $K\otimes_kS$-$\ds$-automorphism of $R\otimes_kS$. This yields an injection $\sgal(L|M)(S)\hookrightarrow\sgal(L|K)(S)$ whose image consists of precisely those $\tau\in G(S)$ that leave invariant every element of $M$. We will often identify $\sgal(L|M)$ with this subfunctor of $\sgal(L|K)$. Because $\sgal(L|M)$ and $\sgal(L|K)$ can be realized as $\s$-closed subgroups of $\Gl_{n,k}$ it follows that $\sgal(L|M)$ is then a $\s$-closed subgroup of $\sgal(L|K)$.

%


%
%
%

\begin{lemma} \label{lemma:compute}
Let $L|K$ be a $\s$-Picard-Vessiot extension with $\s$-Picard-Vessiot ring $R$, $\s$-Galois group $G=\sgal(L|K)$ and field of $\de$-constants $k$.
Let $a\in L$ and $r_1,r_2\in R$, $r_2\neq 0$ such that $a=\frac{r_1}{r_2}$. If $H\leq G$ is a $\s$-closed subgroup of $G$, then $a$ is invariant under $H$ if and only if
$r_1\otimes r_2-r_2\otimes r_1$ lies in the kernel of $R\otimes_K R=R\otimes_k k\{G\}\to R\otimes_kk\{H\}$. Moreover the invariants of $H$ can be computed as
\begin{equation} \label{eq:computeinvariants}
L^H=\{a\in L |\ 1\otimes a-a\otimes 1\in L\otimes_K L\cdot \I(H)\},
\end{equation}
where $\I(H)\subset k\{G\}$ denotes the defining ideal of $H$. (See Definition \ref{defi:sigmaclosed}.)

Conversely, if $M$ is an intermediate $\ds$-field of $L|K$, the defining ideal of $\sgal(L|M)$ in $k\{G\}$ can be computed as
\begin{equation} \label{eq:computevanishingideal}
\I(\sgal(L|M))=\ker(L\otimes_K L\to L\otimes_M L)\cap k\{G\}.
\end{equation}
\end{lemma}

\begin{proof}
The $K$-$\ds$-morphism
\[\psi\colon R\to R\otimes_K R=R\otimes_k k\{G\}\to R\otimes_k k\{H\},\]
where the first map is the inclusion into the second factor, extends to a $K\otimes_k k\{H\}$-$\ds$-morphism $\tau\colon R\otimes_k k\{H\}\to R\otimes_k k\{H\}$. It follows from Lemma \ref{lemma:endo=auto} that $\tau$ is an automorphism. I.e., $\tau\in G(S)$, where $S:=k\{H\}$. Now to say that $a$ is invariant under $\tau$ precisely means that $r_1\otimes r_2-r_2\otimes r_1$ lies in the kernel of $R\otimes_K R=R\otimes_k k\{G\}\to R\otimes_kk\{H\}$.

For the reverse direction, let $S$ be any $k$-$\s$-algebra. According to the identifications made in the proof of Proposition \ref{prop:defgal}, every $\tau\in H(S)\subset \Aut(R\otimes_k S|K\otimes_k S)$ is obtained from a $k$-$\s$-morphism $k\{H\}\to S$ by extending $R\xrightarrow{\psi} R\otimes_k k\{H\}\to R\otimes_k S$ to $R\otimes_k S$.
This shows that $a$ is invariant under $\tau$ if $r_1\otimes r_2-r_2\otimes r_1$ lies in the kernel of $R\otimes_K R=R\otimes_k k\{G\}\to R\otimes_kk\{H\}$.

In summary we see that $a$ is invariant under $H$ if and only if $r_1\otimes r_2-r_2\otimes r_1$ lies in the ideal of $R\otimes_K R$ generated by $\I(H)=\ker(k\{G\}\to k\{H\})$.
Working inside $L\otimes_K L$ we can divide by $r_2\otimes r_2$ to obtain $1\otimes a-a\otimes 1\in L\otimes_K L\cdot\I(H)$ if $a$ is invariant under $H$. This is ``$\subset$ '' of equation ($\ref{eq:computeinvariants}$).
To prove ``$\supset$'' it suffices to see that $(L\otimes_K L\cdot \ida)\cap R\otimes_K R=\ida$, where $\ida=R\otimes_K R\cdot\I(H)$ is the ideal of $R\otimes_K R$ generated by $\I(H)$.
Consider the inclusions of $\de$-rings
\[k\{G\}\subset R\otimes_K R \subset L\otimes_K L.\]
By Lemma \ref{lemma:simple} extension and contraction are mutually inverse bijections between the set of \mbox{($\de$-)ideals} of $k\{G\}$ and the set of $\de$-ideals of $R\otimes_K R$. Similarly, by \cite[Proposition 2.3, p. 135]{AmanoMasuokaTakeuchi:HopfPVtheory} extension and contraction are mutually inverse bijections between the set of ($\de$-)ideals of $k\{G\}$ and the set of
$\de$-ideals of $L\otimes_K L$. This implies that extension and contraction also are mutually inverse bijections between the set of $\de$-ideals of $R\otimes_K R$ and the set of $\de$-ideals of $L\otimes _K L$. In particular, $(L\otimes_K L\cdot \ida)\cap R\otimes_K R=\ida$.

\bigskip

It remains to prove the identity (\ref{eq:computevanishingideal}). So let $M$ be an intermediate $\ds$-field of $L|K$ and set $H:=\sgal(L|M)$.
If $S$ is a $k$-$\s$-algebra and $\tau\in G(S)$, then we denote with $\widetilde{\tau}\colon k\{G\}\to S$ the element of $\Alg_k^\s(k\{G\},S)$ corresponding to $\tau$ under
$G(S)\simeq\Alg_k^\s(k\{G\},S)$. An element $\tau\in G(S)$ leaves invariant an element $a=\frac{r_1}{r_2}\in L$ if and only if
$r_1\otimes r_2-r_2\otimes r_1$ lies in the kernel of $R\otimes_K R=R\otimes_k k\{G\}\xrightarrow{\operatorname{id}\otimes \widetilde{\tau}}R\otimes_k S$. Thus $\tau\in G(S)$ leaves invariant every element of $M$, i.e., $\tau\in H(S)$, if and only if $\idb$ lies in the kernel $R\otimes_K R=R\otimes_k k\{G\}\to R\otimes_k S$, where $\idb$ is the ideal of $R\otimes_K R $ generated by all elements of the form $r_1\otimes r_2-r_2\otimes r_1$ with $r_1, r_2\in R$, $r_2\neq 0$ and $\frac{r_1}{r_2}\in M$.
Note that $\idb$ is precisely the kernel of $R\otimes_K R\to MR\otimes_M MR$. In particular, $\idb$ is a $\ds$-ideal.
It follows from Lemma \ref{lemma:simple} that $\idb=R\otimes_k(\idb\cap k\{G\})$. So $\tau\in G(S)$ lies in $H(S)$ if and only if $\idb\cap k\{G\}\subset\ker\widetilde{\tau}$. This means that $\I(H)=\idb\cap k\{G\}$.
So \[\I(H)=\idb\cap k\{G\}=\ker(R\otimes_K R\to MR\otimes_M MR)\cap k\{G\}=\ker(L\otimes_K L\to L\otimes_M L)\cap k\{G\}.\]
\end{proof}

\begin{thm}[$\s$-Galois correspondence] \label{theo:Galoiscorrespondence}
Let $L|K$ be a $\s$-Picard-Vessiot extension with $\s$-Galois group $G=\sgal(L|K)$. Then there is an inclusion reversing bijection between the set of intermediate $\ds$-fields $M$ of $L|K$ and the set of $\s$-closed subgroups $H$ of $G$ given by
\[M\mapsto\sgal(L|M) \text{ and } H\mapsto L^H.\]
\end{thm}
\begin{proof}
We know from \cite[Theorem 2.6, p. 136]{AmanoMasuokaTakeuchi:HopfPVtheory} that the assignments
\[M\mapsto \ker(L\otimes_K L\to L\otimes_M L)\cap k\{G\}\]
and \[I\mapsto \{a\in L |\ 1\otimes a-a\otimes 1\in L\otimes_K L\cdot I\}\]
are inverse to each other, and yield a bijection between the set of all intermediate $\de$-fields $M$ of $L|K$ and all Hopf-ideals $I$ of $k\{G\}$. The claim thus follows from Lemma \ref{lemma:compute}.
\end{proof}

\begin{thm}[Second fundamental theorem of $\s$-Galois theory] \label{theo:secondfundamentaltheorem}
Let $L|K$ be a $\s$-Picard-Vessiot extension with $\s$-Galois group $G$. Let $K\subset M\subset L$ be an intermediate $\ds$-field and $H\leq G$ a $\s$-closed subgroup of $G$ such that $M$ and $H$ correspond to each other in the $\s$-Galois correspondence.

Then $M$ is a $\s$-Picard-Vessiot extension of $K$ if and only if $H$ is normal in $G$. If this is the case, the $\s$-Galois group of $M|K$ is the quotient $G/H$. (See Definition \ref{defi:quotients} for the definition and Theorem \ref{theo:existenceofquotients} for the existence of the quotient $G/H$.)
\end{thm}

\begin{proof}
We first assume that $M|K$ is $\s$-Picard-Vessiot. Let $R\subset L$ denote the $\s$-Picard-Vessiot ring of $L|K$ and $R'\subset M$ the $\s$-Picard-Vessiot ring of $M|K$. First of all, we need to convince ourselves that $R'\subset R$: The ring compositum $RR'$ inside $L$ is a $\s$-Picard-Vessiot ring contained in $L$ with quotient field $L$. Indeed, if $R$ is a $\s$-Picard-Vessiot ring for $\de(y)=Ay$ with fundamental solution matrix $Y\in\Gl_n(R)$ and $R'$ is a $\s$-Picard-Vessiot ring for $\de(y)=A'y$ with fundamental solution matrix $Y'\in\Gl_{n'}(R')$, then $RR'$ is a $\s$-Picard-Vessiot ring for
\[\de(y)=\left(\begin{array}{cc} A & 0 \\ 0 & A' \end{array}\right)y \ \text{ with fundamental solution matrix } \ \left(\begin{array}{cc} Y & 0 \\ 0 & Y' \end{array}\right)\in\Gl_{n+n'}(RR').\]
Since $R$ is the only $\s$-Picard-Vessiot ring inside $L|K$ with quotient field $L$ by Lemma \ref{lemma:sigmaPVringunique}, it follows that $RR'=R$, i.e., $R'\subset R$.

Set $G'=\sgal(M|K)$ and let $S$ be a $k$-$\s$-algebra. Because $R'\subset R$ and $R'$ is a $\s$-Picard-Vessiot ring we see that every automorphism $\tau\in G(S)=\Aut^\ds(R\otimes_kS| K\otimes_k S)$ restricts to an automorphism $\tau'\in G'(S)=\Aut^\ds(R'\otimes_k S|K\otimes_k S)$.
This defines a morphism $\phi\colon G\to G'$ of group $k$-$\s$-schemes. Because the quotient field of $R'$ is equal to $M$, it is clear that $H=\sgal(L|M)$ is the kernel of $\phi\colon G\to G'$. In particular, $H$ is normal in $G$.


Thus, by Corollary \ref{cor:quotients}, to see that $\phi\colon G\to G'$ is the quotient morphism of $G$ modulo $H$ it suffices to see that $\phi^*$ is injective. To get an explicit description of $\phi^*\colon k\{G'\}\to k\{G\}$, one has to take $S:=k\{G\}$ and to chase $\operatorname{id}\in\Alg_k^{\ds}(k\{G\},S)\simeq G(S)=G(k\{G\})$ through the identifications of the proof of Proposition \ref{prop:defgal}. One finds that $\phi^*$ is obtained from the inclusion $R'\otimes_K R'\hookrightarrow R\otimes_K R$ by taking $\de$-constants.
I.e., $\phi^*\colon k\{G'\}=(R'\otimes_K R')^\de\hookrightarrow (R\otimes_K R)^\de=k\{G\}$. So clearly $\phi^*$ is injective and we conclude that $G'=G/H$.

\bigskip

It remains to see that $M|K$ is $\s$-Picard-Vessiot if $H$ is normal in $G$. Let $A\in K^{n\times n}$ and $Y\in\Gl_n(L)$ such that $L|K$ is a $\s$-Picard-Vessiot extension for $\de(y)=Ay$ with fundamental solution matrix $Y$. As in Proposition \ref{prop:Zariskiclosures}, set $L_d=K\left(Y,\ldots,\s^d(Y)\right)$ for $d\geq 0$ and consider $G$ as a $\s$-closed subgroup of $\Gl_{n,k}$. The Galois group of $L_d|K$ equals $G[d]$ and $H[d]$ is a normal closed subgroup scheme of $G[d]$ by Lemma \ref{lemma:neededforSecondFundTheorem}. Now it follows from the (classical) second fundamental theorem of Galois theory (see \cite[Corollary 1.40, p. 31]{vdPutSingerDifferential} for the case of algebraically closed constants or \cite[Theorem 2.11, p. 138]{AmanoMasuokaTakeuchi:HopfPVtheory} for the general case) applied to $H[d]\unlhd G[d]$ that $L_d^{H[d]}|K$ is Picard-Vessiot.

An element $a\in L_d$ is invariant under $H$ if and only if it is invariant under $H[d]$. (This follows for example from Lemma \ref{lemma:compute}.)
Therefore $L_d^{H[d]}=L^H\cap L_d=M\cap L_d$ and $M\cap L_d$ is a Picard-Vessiot extension of $K$.


An intermediate $\s$-field of a finitely $\s$-generated $\s$-field extension is again finitely $\s$-generated (\cite[Theorem 4.4.1, p. 292]{Levin}). Thus we can find $a_1,\ldots,a_m\in M$ such that $M=K\langle a_1,\ldots, a_m\rangle_\s$. Now choose $d'\geq 0$ such that $a_1,\ldots,a_m\in L_{d'}$. Because $M\cap L_{d'}|K$ is Picard-Vessiot, there exist matrices $A'\in K^{n'\times n'}$ and $Y'\in\Gl_{n'}(M\cap L_{d'})$ such that $M\cap L_{d'}|K$ is a Picard-Vessiot extension for $\de(y)=A'y$ with fundamental solution matrix $Y'$.

We have $a_1,\ldots,a_m\in M\cap L_{d'}=K(Y')$ and so $M=K\langle a_1,\ldots, a_m\rangle_\s\subset K\langle Y'\rangle_\sigma$. Thus, $M=K\langle Y'\rangle_\sigma$ is a $\s$-Picard-Vessiot extension for $\de(y)=A'y$ since $M^\de\subset L^\de=k$.
\end{proof}

\begin{cor}
Let $L|K$ be a $\s$-Picard-Vessiot extension and $K\subset M\subset L$ an intermediate $\ds$-field such that $M|K$ is a $\s$-Picard-Vessiot extension. If $R$ is the $\s$-Picard-Vessiot ring of $L|K$ then $R\cap M$ is the $\s$-Picard-Vessiot ring of $M|K$.
\end{cor}
\begin{proof}
Let $R'\subset M$ denote the $\s$-Picard-Vessiot ring of $M|K$. We have already seen at the beginning of the proof of Theorem \ref{theo:secondfundamentaltheorem}
that $R'\subset R\cap M$. To see that $R'=R\cap M$ it suffices to note that the corresponding statement is true if we forget $\sigma$. (See the proof of \cite[Theorem 2.8, p. 137]{AmanoMasuokaTakeuchi:HopfPVtheory}).
\end{proof}

%

\section{$\s$-separability} \label{sec:sseparability}

Let $R$ be a $\s$-ring. There are some natural conditions which we can impose on $R$:

\begin{itemize}
 \item $R$ is $\s$-reduced, i.e., $\s\colon R\to R$ is injective.
\item $R$ is perfectly $\s$-reduced, i.e., if $f\in R$ and $\alpha_1,\ldots,\alpha_n\in\N$ such that $\s^{\alpha_1}(f)\cdots\s^{\alpha_n}(f)=0$, then $f=0$.
\item $R$ is a $\s$-domain, i.e., $R$ is an integral domain and $\s$-reduced. 
\end{itemize}

The importance of perfectly $\s$-reduced $\s$-rings stems from the fact that the finitely $\s$-generated, perfectly $\s$-reduced $k$-$\s$-algebras are precisely the $\s$-coordinate rings of the classical $\s$-varieties over some $\s$-field $k$. See \cite[Section 2.6]{Levin}. $\s$-domains correspond to irreducible $\s$-varieties.

The main point of this section is to understand the implications on the $\s$-Picard-Vessiot extension $L|K$ if we impose one of the above conditions on the $\s$-coordinate ring $k\{G\}$ of the $\s$-Galois group $G=\sgal(L|K)$.
We also use this insight to explain what remains of the $\s$-Galois correspondence if one naively insists that the $\s$-Galois group of a $\s$-Picard-Vessiot extension $L|K$ is the actual automorphism group $\Aut^\ds(L|K)$ of $L|K$.

\medskip

To see the complete picture we need to study the above properties under extension of the base $\s$-field.

\begin{defi} \label{defi:sigmaseparableSigmaregular}
Let $k$ be a $\s$-field and $R$ a $k$-$\s$-algebra. We say that $R$ is
\begin{itemize}
\item  \emph{$\s$-separable} over $k$ if $R\otimes_k k'$ is $\s$-reduced;
\item \emph{perfectly $\s$-separable} over $k$ if $R\otimes_k k'$ is perfectly $\s$-reduced;
\item \emph{$\s$-regular} over $k$ if $R\otimes_k k'$ is a $\s$-domain;
\end{itemize}
for every $\s$-field extension $k'$ of $k$.
\end{defi}
\begin{defi}
Let $k$ be a $\s$-field and $G$ a $\s$-algebraic group over $k$. We say that $G$ is \emph{absolutely $\s$-reduced/perfectly $\s$-reduced/$\s$-integral} if $k\{G\}$ is $\s$-separable/perfectly $\s$-separable/$\s$-regular over $k$.\end{defi}
The properties introduced above are studied in some more detail in Section \ref{subsec:somepropertiesrelatedtobaseextension} of the appendix. Below we make use of these results. See the table after Definition \ref{defi:sigmaschemeabsolutelySigmaReduced} for an overview of the nomenclature.

\begin{prop} \label{prop:correspondenceSigmaSeparabel}
Let $L|K$ be a $\s$-Picard-Vessiot extension with $\s$-field of $\de$-constants $k$ and $\s$-Galois group $G$. Then
\begin{enumerate}
\item $L$ is $\s$-separable over $K$ if and only if $G$ is absolutely $\s$-reduced. In particular, if $k$ is inversive, then $L$ is $\s$-separable over $K$ if and only if $G$ is $\s$-reduced.
\item $L$ is perfectly $\s$-separable over $K$ if and only if $G$ is absolutely perfectly $\s$-reduced. In particular, if $k$ is algebraically closed and inversive, then $L$ is perfectly $\s$-separable over $K$ if and only if $G$ is perfectly $\s$-reduced.
\item $L$ is $\s$-regular over $K$ if and only if $G$ is absolutely $\s$-integral. In particular, if $k$ is algebraically closed and inversive, then $L|K$ is $\s$-regular if and only if $G$ is $\s$-integral.
\end{enumerate}
\end{prop}
\begin{proof}
We give a simultaneous proof of all the statements. Let $R$ denote the $\s$-Picard-Vessiot ring of $L|K$.
Let $L^*$ denote the inversive closure of $L$ (\cite[Def. 2.1.6, p. 109]{Levin}) and let $L'$ denote an algebraic closure of $L^*$, equipped with an extension of $\s$. Then $L'$ is an algebraically closed and inversive $\s$-field.
The fundamental isomorphism $R\otimes_K R\simeq R\otimes_k k\{G\}$ extends to
$L'\otimes_K R\simeq L'\otimes_k k\{G\}$.

Assume that $L$ is $\s$-separable/perfectly $\s$-separable/$\s$-regular over $k$. Then also $R$ is $\s$-separable/perfectly $\s$-separable/$\s$-regular over $k$.
Therefore $L'\otimes_K R= L'\otimes_k k\{G\}$ is $\s$-reduced/perfectly $\s$-reduced/$\s$-integral. It follows from Lemma \ref{lemma:sigmaseparablesufficesextension} that
$G$ is absolutely $\s$-reduced/perfectly $\s$-reduced/$\s$-integral.

Conversely, if $G$ is absolutely $\s$-reduced/perfectly $\s$-reduced/$\s$-integral, then $R$ must be $\s$-separable/perfectly $\s$-separable/$\s$-regular over $K$. It follows from Lemma \ref{lemma:sigmaseparablequotientfield} that $L$ is $\s$-separable/perfectly $\s$-separable/$\s$-regular over $K$. The ``in particular'' statements are clear from Corollary \ref{cor:absolutelysufficesextension}.
\end{proof}

\begin{cor}\label{cor:inverssep} Let $K$ be $\ds$-field with $\s\colon K\to K$ an automorphism and
let $L|K$ be a $\s$-Picard-Vessiot extension with $\s$-Galois group $G$. Then
$L$ is  $\s$-separable over $K$ and $G$ is absolutely $\s$-reduced.\end{cor}
\begin{proof}
By Corollary \ref{cor:sigmaseparablesufficesextension} (i), since $K$ is inversive, $L|K$ is $\s$-separable.
\end{proof}

\subsection{A ``naive'' point of view: perfect $\s$-separability}\label{subsec:perfsep}

The purpose of this subsection is to explain what remains of the $\s$-Galois correspondence if one naively insists that the $\s$-Galois group of a $\s$-Picard-Vessiot extension $L|K$ is the actual automorphism group $\Aut^\ds(L|K)$ of $L|K$. Such an approach, closer to \cite{cassisinger} and \cite{HardouinSinger}, is in principle possible. It has the advantage of being notationally more convenient. For example, $\Aut^\ds(L|K)$ acts on $L$ whereas $\sgal(L|K)$ is acting (functorially) only on $R$, the $\s$-Picard-Vessiot ring of $L|K$. One disadvantage is that one must assume that the $\de$-constants are $\s$-closed. In the case of differential parameters, this is essentially the only disadvantage: If $k$ is a $\de$-closed $\de$-field, then the $\de$-closed subgroups of $\Gl_n(k)$ are the same as the $\de$-closed subgroups\footnote{In the sense of scheme theory.} of $\Gl_{n,k}$. (Because in characteristic zero every Hopf-algebra is reduced.)
However, in our case, the case of a difference parameter, there are more $\s$-closed subgroups of $\Gl_{n,k}$ than $\s$-closed subgroups of $\Gl_n(k)$ (even if $k$ is $\s$-closed). See the introduction for an example.
So only a certain part of the general $\s$-Galois correspondence (Theorem \ref{theo:Galoiscorrespondence}) will remain if we replace $\sgal(L|K)$ with $\Aut^\ds(L|K)$. Let us illustrate this phenomenon with an example. 

\begin{exa}
Let $L|K$ be the $\s$-Picard-Vessiot extension from Example \ref{exa:benign}. Then $\Aut^\ds(L|K)=C_2$ and so, in the ``naive'' sense, we only have two groups on the group side, namely $C_2$ and the trivial group. The trivial group corresponds to $L$ and the intermediate $\ds$-field of $L|K$ fixed by $C_2$ is
$$L^{C_2}=K(\sqrt{x+i}\sqrt{x+j}|\ i,j\geq 0)\subset L.$$ All the other intermediate $\ds$-fields of $L|K$ are ``lost''. For example, the intermediate $\ds$-field $K(\sqrt{x+1},\sqrt{x+2},\ldots)\subset L$, which corresponds to the $\s$-closed subgroup $H$ of $\sgal(L|K)$ given by
$H(S)=\{g\in S^\times|\ g^2=1,\ \s(g)=1\}$ for any $k$-$\s$-algebra $S$, does not appear in this correspondence.
\end{exa}

Let $L|K$ be a $\s$-Picard-Vessiot extension with $\s$-Picard-Vessiot ring $R$, $\s$-field of $\de$-constants $k$ and $\s$-Galois group $G=\sgal(L|K)$. As
\[G(k)=\Aut^\ds(R|K)=\Aut^\ds(L|K)\] the problem essentially boils down to ``When can a $\s$-algebraic $k$-$\s$-scheme $X$ be recovered from its $k$-rational points $X(k)$?'' But this is well known, it will be possible if $k$ is ``big enough'' and $X$ is perfectly $\s$-reduced.

By Proposition \ref{prop:correspondenceSigmaSeparabel} (ii), the (absolutely) perfectly $\s$-reduced subgroups of $G$ correspond to intermediate $\ds$-fields $M$ of $L|K$ such that $L|M$ is perfectly $\s$-separable. We shall give a more explicit characterization of these intermediate $\ds$-fields. We first need a simple lemma.

\begin{lemma} \label{lemma:relativeAlgebraicClosureIsGalois}
Let $L|K$ be a $\s$-Picard-Vessiot extension and let $K'$ denote the relative algebraic closure of $K$ inside $L$. Assume that $k=K^\de$ is algebraically closed. Then the field extension $K'|K$ is Galois.
\end{lemma}
\begin{proof}
As in Proposition \ref{prop:Zariskiclosures}, we can write $L|K$ as a directed union of Picard-Vessiot extensions $L_d|K$ ($d\geq 0$). Because $k=(L_d)^\de$ is algebraically closed, we know that the relative algebraic closure of $K$ in $L_d$ is a Galois extension of $K$ (\cite[Proposition 1.34, p. 25]{vdPutSingerDifferential}). Thus $K'$ is the directed union of Galois extension of $K$. So $K'$ is itself Galois over $K$.
\end{proof}


\begin{lemma} \label{lemma:characterisePerfectlySigmaReducedGroup}
Let $L|K$ be a $\s$-Picard-Vessiot extension with $\s$-Galois group $G$. Assume that $k=K^\de$ is algebraically closed and inversive.
Then $G$ is perfectly $\s$-reduced if and only if $L|K$ satisfies the following properties:
\begin{enumerate}
 \item $L$ is $\s$-separable over $K$.
\item The relative algebraic closure $K'$ of $K$ inside $L$ is a finite field extension of $K$.
\item Every (field) automorphism of $K'|K$ commutes with $\s\colon K'\to K'$.
\end{enumerate}
\end{lemma}
\begin{proof}
Assume that $G$ is perfectly $\s$-reduced. As $k$ is algebraically closed and inversive, it follows from Proposition \ref{prop:correspondenceSigmaSeparabel} (ii) that $L|K$ is perfectly $\s$-separable. A fortiori $L|K$ is $\s$-separable. Because $k$ is algebraically closed, we know from Lemma \ref{lemma:relativeAlgebraicClosureIsGalois} that $K'|K$ is Galois. We can thus apply Lemma \ref{lemma:perfectlysigmaseparableMain} to conclude that $L|K$ satisfies conditions (ii) and (iii).

The converse direction is similar: If $L|K$ satisfies conditions (i), (ii) and (iii) then it follows from Lemma \ref{lemma:perfectlysigmaseparableMain} that $L|K$ is perfectly $\s$-separable. So $G$ is perfectly $\s$-reduced.
\end{proof}

It seems interesting to note that the $\de$-analogs of conditions (ii) and (iii) are automatically satisfied: If $L|K$ is a finitely $\de$-generated extension of $\de$-fields then the relative algebraic closure $K'$ of $K$ inside $L$ is finite and every automorphism of $K'|K$ commutes with $\de$ (\cite[Corollary 2, Chapter II, Section 11, p. 113]{Kolchin:differentialalgebraandalgebraicgroups} and \cite[Lemma 1, Chapter II, Section 2, p. 90]{Kolchin:differentialalgebraandalgebraicgroups}).

\bigskip

Let $L|K$ be a $\s$-Picard-Vessiot extension for $\de(y)=Ay$ with fundamental solution matrix $Y\in\Gl_n(L)$. For the rest of this section, we will assume that $k=K^\de$ is $\s$-closed.
Every automorphism $\tau\in\Aut^\ds(L|K)$ is given by a matrix $[\tau]_Y\in\Gl_n(k)$ satisfying $\tau(Y)=Y[\tau]_Y$. The mapping
\[\Aut^\ds(L|K)\to\Gl_n(k),\ \tau\mapsto [\tau]_Y\] is an injective morphism of groups. We shall henceforth identify $\Aut^\ds(L|K)$ with the image of this embedding.

The $\s$-closed subsets of $\Gl_n(k)$ are defined as in Section \ref{sec:The semi-classical point of view}, i.e., as the solution sets of systems of $\s$-polynomials in the matrix entries.
Of course $\Aut^\ds(L|K)$ is $\s$-closed in $\Gl_n(k)$. Indeed, with the notation of the proof of Proposition \ref{prop:defgal}, $\Aut^\ds(L|K)$ is the solution set of the kernel of
$k\{X,\frac{1}{\det(X)}\}_\s\to k\{Z,\frac{1}{\det(Z)}\}_\s$, $X\mapsto Z$. If $\widetilde{H}$ is a $\s$-closed subgroup of $\Aut^\ds(L|K)$ we set
\[L^{\widetilde{H}}:=\{a\in L|\ h(a)=a \ \forall \ h\in\widetilde{H}\}.\]

\begin{lemma} \label{lemma:kpoints}
Let $L|K$ be a $\s$-Picard-Vessiot extension for $\de(y)=Ay$ with fundamental solution matrix $Y\in\Gl_n(L)$ and $\s$-Galois group $G=\sgal(L|K)$. Assume that $k=K^\de$ is $\s$-closed.
The assignment $H\mapsto H(k)$ defines a bijection between the set of the $\s$-closed subgroups of $G$ that are perfectly $\s$-reduced and the set of the $\s$-closed subgroups of
$\Aut^\ds(L|K)$. Moreover, if $H\leq G$ is perfectly $\s$-reduced, then $L^H=L^{H(k)}$.
\end{lemma}
\begin{proof}
The statement about the bijection follows from Lemma \ref{lemma:bijectionSigmaClosedandPerfectlySigmaReduced}.

Let $H$ be a perfectly $\s$-reduced $\s$-closed subgroup of $G$ and let $R$ denote the $\s$-Picard-Vessiot ring of $L|K$. By definition $L^H\subset L^{H(k)}$.  For $h\in H(k)\subset G(k)$, let $\operatorname{ev}_h\colon k\{G\}\to k$ denote the evaluation map. With the notation of the proof of Proposition \ref{prop:defgal} it is given by $Z\mapsto [h]_Y$. Then, for $r\in R$, $h(r)$ is the image of $1\otimes r$ under
$R\otimes_K R=R\otimes_k k\{G\}\xrightarrow{\operatorname{id}\cdot\operatorname{ev}_h}R$. We also set $\m_h=\ker(\operatorname{ev}_h)$.

Let $a\in L$. We may write $a=\frac{r_1}{r_2}$ with $r_1,r_2\in R$, $r_2\neq 0$. We have $h(a)=a$ if and only if $r_1\otimes r_2-r_2\otimes r_1$ lies in the kernel of
$R\otimes_K R=R\otimes_k k\{G\}\xrightarrow{\operatorname{id}\cdot\operatorname{ev}_h}R$.
Thus, $a$ is invariant under $H(k)$ if and only if $r_1\otimes r_2-r_2\otimes r_1$ lies in $R\otimes_k(\bigcap_{h\in H(k)}\m_h)$.

On the other side, by Lemma \ref{lemma:compute}, $a$ is invariant under $H$ if and only if $r_1\otimes r_2-r_2\otimes r_2$ lies in $R\otimes_k\I(H)$, where $\I(H)$ denotes the defining ideal of $H$ in $G$. Thus $L^H=L^{H(k)}$ because $\I(H)=\bigcap_{h\in H(k)}\m_h$ by Lemma \ref{lemma:sigmaClosedIntersection}.
\end{proof}

We now arrive at the reduced version of the $\s$-Galois correspondence which results if one wants to avoid the use of schemes.

\begin{prop} \label{prop:naiveSigmacorrespondence}
Let $L|K$ be a $\s$-Picard-Vessiot extension. Assume that $k=K^\de$ is a $\s$-closed $\s$-field. The assignments
$M\mapsto\Aut^{\ds}(L|M)$ and $H\mapsto L^H$ define mutually inverse bijections between the set of all intermediate $\ds$-fields $M$ of $L|K$ such that $L|M$ is perfectly $\s$-separable and the set of all $\s$-closed subgroups $H$ of $\Aut^{\ds}(L|K)$.

Moreover, for an intermediate $\ds$-field $M$ of $L|K$, the extension $L|M$ is perfectly $\s$-separable if and only if the following assertions are satisfied:
\begin{enumerate}
 \item $L$ is $\s$-separable over $M$.
\item The relative algebraic closure $M'$ of $M$ inside $L$ is a finite field extension of $M$.
\item Every automorphism of $M'|M$ commutes with $\s\colon M'\to M'$.
\end{enumerate}
\end{prop}
\begin{proof}
The statement about the bijections follows from the general $\s$-Galois correspondence (Theorem \ref{theo:Galoiscorrespondence}) together with Proposition \ref{prop:correspondenceSigmaSeparabel} (ii) and Lemma \ref{lemma:kpoints}. The second statement follows from Lemma \ref{lemma:characterisePerfectlySigmaReducedGroup}.
\end{proof}

Note that in the setting of Proposition \ref{prop:naiveSigmacorrespondence} one has $L^{\Aut^{\ds}(L|K)}=K$ if and only if $L|K$ is perfectly $\s$-separable.

\settocdepth{section}

\appendix

\section{Appendix: Difference algebraic groups}

While differential algebraic groups, i.e., group objects in the category of differential varieties, are a classical topic in differential algebra (see e.g. \cite{Cassidy:differentialalgebraicgroups}, \cite{Kolchin:differentialalgebraicgroups}, \cite{Buium:DifferntialAlgebraicGroupsOfFiniteDimension}), their difference analog has been neglected by the founding fathers of difference algebra. It appears that almost all results pertaining to groups defined by algebraic difference equations are relatively recent, due to model theorists and motivated by number-theoretic applications. See \cite{Hrushovskietal:ModelTheoryofDifferencefields}, \cite{ChatHrusPet}, \cite{HrushMan}, \cite{Chatdifgr}, \cite{Scanlon:DifferenceAlgebraicSubgroupsOfCommutativeAlgebraicGroups}, \cite{PillayKowalski:ANoteonGroupsDefinableInDifferenceFields}, \cite{KowPillsigma},
\cite{ChatzidakisHrushovski:OnSubgroupsOfSemiabelianVarietiesDefinedByDifferenceEquations}.

Even though, both, the notion of a group definable in ACFA and our notion of a difference algebraic group  (Definition \ref{defi:sigmaalgebraicgroup}), give precise meaning to the idea of a group defined by algebraic difference equations, none of these notions encompasses the other.
Our notion agrees with the notion of a linear $\mathfrak{M}$-group in \cite{Kamensky:TannakianFormalismOverFieldsWithOperators}, for a suitable choice of $\mathfrak{M}$. Also, a difference algebraic group in our sense, such that its coordinate ring is finitely generated as an algebra\footnote{Even in $\s$-dimension zero this will rarely happen.}, is essentially the same thing as an affine algebraic $\s$-group in the sense of \cite{KowPillsigma}.

The main purpose of this appendix is to provide a brief introduction to difference algebraic groups, suitable for the applications in the main text. A more systematic and complete account will eventually be given by the third author. Standard references for difference algebra are \cite{Cohn:difference} and the more recent \cite{Levin}. Many ideas can also be found in \cite{Hrushovski:elementarytheoryoffrobenius}.
We consider most of the constructions presented in this appendix as ``well-known''. However, it is sometimes difficult to pin down suitable references.

\subsection{Some terminology from difference algebra} \label{sec:Someterminoplogyfromdifferencealgebra}
Throughout the text we use some basic notions from difference algebra. For the convenience of the reader not well acquainted with difference algebra, we collect here some conventions and terminology.

\medskip

All rings are commutative with identity. A \emph{difference ring} (or \emph{$\s$-ring} for short) is a ring $R$ together with a ring endomorphism $\s\colon R\to R$.
Algebraic attributes (e.g. Noetherian) are understood to apply to the underlying ring. Attributes that apply to the difference structure are usually prefixed with $\s$ (e.g. finitely $\s$-generated). The expression $\s^0$ is understood to be the identity.

A morphism of $\s$-rings is a morphism of rings that commutes with $\s$. Let $R$ be a $\s$-ring. By an \emph{$R$-$\s$-algebra}, we mean a $\s$-ring $S$ together with a morphism $R\to S$ of $\s$-rings. A morphism of $R$-$\s$-algebras is a morphism of $R$-algebras that is also a morphism of $\s$-rings. If $S$ and $S'$ are $R$-$\s$-algebras, we write
$$\Alg^\s_R(S,S')$$ for the set of $R$-$\s$-algebra morphisms from $S$ to $S'$.

An \emph{$R$-$\s$-subalgebra} $S'$ of an $R$-$\s$-algebra $S$ is an $R$-subalgebra such that the inclusion morphism $S'\to S$ is a morphism of $\s$-rings. The tensor product $S_1\otimes_R S_2$ of two $R$-$\s$-algebras $S_1$ and $S_2$ naturally carries the structure of an $R$-$\s$-algebra by virtue of $\s(s_1\otimes s_2)=\s(s_1)\otimes \s(s_2)$.

Let $k$ be a $\s$-field, i.e., a $\s$-ring whose underlying ring is a field. Let $R$ be a $k$-$\s$-algebra and $B$ a subset of $R$. The smallest $k$-$\s$-subalgebra of $R$ that 
contains $B$ is denoted with $k\{B\}_\s$ and called the \emph{$k$-$\s$-subalgebra $\s$-generated by $B$}. As a $k$-algebra it is generated by $B,\s(B),\ldots$. If there exists a finite subset $B$ of $R$ such that $R=k\{B\}_\s$,
we say that $R$ is \emph{finitely $\s$-generated over $k$}.
The $k$-$\s$-algebra $k\{x\}_\s=k\{x_1,\ldots,x_n\}$ of $\s$-polynomials over $k$ in the $\s$-variables $x_1,\ldots,x_n$ is the polynomial ring over $k$ in the variables $x_1,\ldots,x_n,\s(x_1),\ldots,\s(x_n),\ldots$, with an action of $\s$ as suggested by the names of the variables.


Let $k$ be a $\s$-field. A \emph{$\s$-field extension} $k'$ of $k$ is a $\s$-field containing $k$ such that the inclusion map is a morphism of $\s$-rings. We also say that $k$ is a \emph{$\s$-subfield} of $k'$. If $B\subset k'$, the smallest $\s$-field extension of $k$ inside $k'$ that contains $B$ is denoted with $k\langle B\rangle_\s$. As a field extension of $k$ it is generated by $B,\s(B),\ldots$. We say that $k'$ is a \emph{finitely $\s$-generated $\s$-field extension} of $k$ if there exists a finite subset $B$ of $k'$ such that $k'=k\langle B\rangle_\s$.

\bigskip

Let $R$ be a $\s$-ring. A \emph{$\s$-ideal} $\ida$ of $R$ is an ideal $\ida\subset R$ such that $\s(\ida)\subset\ida$. Then $R/\ida$ is naturally a $\s$-ring. Let $B$ be a subset of $R$. We denote
by $[B]$ the $\s$-ideal generated by $B$ in $R$. As an ideal it is generated by $B,\s(B),\ldots.$
A $\s$-ideal $\ida$ of $R$ is called \emph{reflexive} if $\s^{-1}(\ida)=\ida$, i.e., $\s(r)\in\ida$ implies $r\in\ida$.
A $\s$-ideal $\ida$ of $R$ is called \emph{perfect} if
$\s^{\alpha_1}(r)\cdots\s^{\alpha_n}(r)\in\ida$ implies $r\in \ida$ for all $r\in R$, $n\geq 1$ and $\alpha_1,\ldots,\alpha_n\geq 0$.
A $\s$-ideal $\q$ of $R$ is called \emph{$\s$-prime} if it is a prime ideal and reflexive. Note that this property is stronger than being a prime $\s$-ideal. One can show that the perfect $\s$-ideals are precisely the intersections of $\s$-prime ideals.

A $\s$-ring $R$ is called inversive if $\s\colon R\to R$ is an automorphism. A $\s$-ring $R$ is called \emph{$\s$-reduced} if $\s\colon R\to R$ is injective. (Equivalently, the zero ideal is reflexive.) We say that $R$ is \emph{perfectly $\s$-reduced} if the zero ideal of $R$ is perfect. If the zero ideal is $\s$-prime we say that $R$ is a \emph{$\s$-domain}. This is equivalent to saying that $R$ is an integral domain with $\s\colon R\to R$ injective.

\subsection{$\sigma$-schemes}
Throughout the appendix $k$ denotes an arbitrary $\s$-field. Because the main text deals with derivations,
we have made it a general assumption that all fields are of characteristic zero. However, this appendix does not require the characteristic zero assumption. All products are understood to be products over $k$.

It is widely recognized that a functorial approach to algebraic groups has many benefits (\cite{Waterhouse:IntrotoAffineGroupSchemes}, \cite{demazuregabriel}, \cite{Milne:BasicTheoryOfAffineGroupSchemes}). Here we will adopt a similar point of view.

\begin{defi} \label{defi:ksigmaScheme}
Let $k$ be a $\s$-field. A \emph{$k$-$\s$-scheme} (or \emph{$\s$-scheme over $k$}) is a (covariant) functor from the category of of $k$-$\s$-algebras to the category of sets which is representable. Thus a functor $X$ from the category of $k$-$\s$-algebras to the category of sets is a $k$-$\s$-scheme if and only if there exists a $k$-$\sigma$-algebra $k\{X\}$ and an isomorphism of functors
\[X\simeq\Alg_k^\s(k\{X\},-).\]
By the Yoneda lemma, the $k$-$\s$-algebra $k\{X\}$ is uniquely determined up to unique $k$-$\s$-isomorphisms. We call it the \emph{$\s$-coordinate ring of $X$}. A morphism of $k$-$\s$-schemes is a morphism of functors. If $\phi\colon X\to Y$ is a morphism of $k$-$\s$-schemes, we denote the dual morphism of $k$-$\s$-algebras with $\phi^*\colon k\{Y\}\to k\{X\}$.

A $k$-$\s$-scheme $X$ is called \emph{$\s$-algebraic} (over $k$) if $k\{X\}$ is finitely $\s$-generated over $k$.
We say that a $k$-$\s$-scheme $X$ is \emph{$\s$-reduced/perfectly $\s$-reduced/$\s$-integral} if $k\{X\}$ is $\s$-reduced/perfectly $\s$-reduced/a $\s$-domain.
\end{defi}

It would be somewhat more accurate to add the word ``affine'' into the above definition. However, to avoid endless iterations of the word ``affine'' we make the following convention.

\bigskip

\noindent {\bf {\large Convention:} All schemes and $\s$-schemes considered are affine.}

\bigskip

The above definition does not agree\footnote{In essence this is due to the fact that, starting with a difference ring $R$, one can not recover $R$ from the global sections on $\spec^\s(R)$. Indeed, if the worst comes to the worst, $\spec^\s(R)$, the set of $\s$-prime ideals of $R$, is empty.} with the definition of a difference scheme given in \cite{Hrushovski:elementarytheoryoffrobenius}. The approach presented here is essentially equivalent to the approach in \cite{MoosaScanlon:GeneralizedHasseSchmidtVarietiesAndTheirJetSpaces}. The classical difference varieties (as in \cite{Cohn:difference} and \cite{Levin}) correspond to perfectly $\s$-reduced, $\s$-algebraic $k$-$\s$-schemes.

\begin{rem} \label{rem: yoneda for sigma schemes}
By the Yoneda lemma, the category of $k$-$\s$-schemes is anti-equivalent to the category of $k$-$\s$-algebras.
\end{rem}

\begin{defi} \label{defi:sigmaclosed}
Let $X$ be a $k$-$\s$-scheme. By a \emph{$\s$-closed $\s$-subscheme $Y\subset X$}, we mean a subfunctor $Y$ of $X$ which is represented by $k\{X\}/\I(Y)$ for some $\s$-ideal $\I(Y)$ of $k\{X\}$. To be precise, the requirement is that there exists a $\s$-ideal $\I(Y)$ of $k\{X\}$ and an isomorphism $Y\simeq\Alg_k^\s(k\{X\}/\I(Y),-)$ such that
\[
\xymatrix{
Y \ar[d]_-{\simeq} \ar@{^{(}->}[r]& X\ar[d]^\simeq\\
\Alg_k^\s(k\{X\}/\I(Y),-) \ar@{^{(}->}[r] & \Alg_k^\s(k\{X\},-)
}
\]
commutes. The ideal $\I(Y)$ of $k\{X\}$ is uniquely determined by $Y$ and vice versa. We call it the \emph{defining ideal of $Y$} (in $k\{X\}$).

A morphism of $k$-$\s$-schemes $\phi\colon Y\to X$ is called a \emph{$\s$-closed embedding} if it induces an isomorphism of $Y$ with a $\s$-closed $\s$-subscheme of $X$. This is equivalent to saying that $\phi^*\colon k\{X\}\to k\{Y\}$ is surjective.
\end{defi}

The reader displeased by the apparent foolery of the above definitions should indulge in the following example. In principle we are only interested in the situation described in this example.

\begin{ex} \label{ex sigmaaffinespace}
Affine $n$-space over $k$ (or difference affine $n$-space over $k$, if we want to be very precise) is the $k$-$\s$-scheme $\A_k^n$ such that $\A_k^n(S)=S^n$ for every $k$-$\s$-algebra $S$. It is represented by $k\{x\}_\s=k\{x_1,\ldots,x_n\}_\s$ -- the $\s$-polynomial ring over $k$ in the $\s$-variables $x_1,\ldots,x_n$.

Let $F\subset k\{x\}_\s$ be a system of algebraic difference equations. For any $k$-$\s$-algebra $S$, we consider the $S$-rational solutions
\[\V_S(F):=\{a\in S^n| \ p(a)=0 \text{ for all } p\in F\}\]
of $F$ in $S^n$. The functor $X$ defined by $X(S)=\V_S(F)$ is a $\s$-closed $\s$-subscheme of $\A_k^n$. It is represented by
$k\{X\}=k\{x\}_\s/[F]$. Here $[F]$ denotes the difference ideal of $k\{x\}_\s$ generated by $F$. Note that the defining ideal $\I(X)$ of $X$ in $k\{x\}_\s$ equals
\[[F]=\{p\in k\{x\}_\s| \ p(a)=0 \text{ for all } a\in X(S) \text{ and all $k$-$\s$-algebras } S\}.\]
Moreover, every $\s$-closed $\s$-subscheme of $\A^n_k$ is of the above described form.

If $X$ is a $\s$-algebraic $\s$-scheme over $k$, then choosing a $\s$-closed embedding of $X$ into $\A^n_k$ is equivalent to specifying $n$ generators of $k\{X\}$ as $k$-$\s$-algebra.
\end{ex}

\begin{lemma}\label{lemma:inverseimagescheme}
Let $\phi\colon X\to Y$ be a morphism of $k$-$\s$-schemes and $Z\subset Y$ a $\s$-closed $\s$-subscheme. By setting $\phi^{-1}(Z)(S)=\phi(S)^{-1}(Z(S))$ for every $k$-$\s$-algebra $S$, we can naturally define a $\s$-closed $\s$-subscheme $\phi^{-1}(Z)$ of $X$ (the \emph{inverse image of $Z$}). Indeed, $\phi^{-1}(Z)$ is the $\s$-closed $\s$-subscheme of $X$ defined by the ideal of $k\{X\}$ generated by $\phi^*(\I(Z))$.
\end{lemma}
\begin{proof}
Let $\ida$ denote the ideal of $k\{X\}$ generated by $\phi^*(\I(Z))$. Note that $\ida$ is a $\s$-ideal. For $\psi\in X(S)=\Alg_k^\s(k\{X\},S)$ we have
\[\phi(S)(\psi)\in Z(S) \Leftrightarrow \I(Z)\subset\ker(\phi(S)(\psi))=\ker(\psi\circ\phi^*)\Leftrightarrow \phi^*(\I(Z))\subset\ker\psi.\]
Thus $\psi\in\phi^{-1}(Z)(S)$ if and only if $\ida\subset\ker\psi$. This means that $\phi^{-1}(Z)$ is the $\s$-closed $\s$-subscheme of $X$ defined by $\ida$.
\end{proof}

\subsection{The semi-classical point of view} \label{sec:The semi-classical point of view}

The classical set-up for difference algebraic geometry, as it can be found in the standard textbooks \cite{Cohn:difference} and \cite{Levin}, is in spirit close to the ``Foundations of algebraic geometry'' as laid down by Andr\'{e} Weil. The story roughly runs as follows: Suppose we want to study $\s$-algebraic equations over a fixed $\s$-field $k$. Usually $k$ will not contain ``enough'' solutions, so one has to look for solutions in $\s$-field extensions of $k$. One fixes a family of $\s$-overfields of $k$ which is ``large enough'', called the universal system of $\s$-overfields of $k$ (\cite[Definition 2.6.1, p. 149]{Levin}).
A difference variety over $k$ is then the set of solutions in the universal system of $\s$-overfields of $k$ of some set of $\s$-polynomials with coefficients in $k$.
There is a one-to-one correspondence between the $\s$-varieties defined by $\s$-polynomials in the $\s$-variables $x_1,\ldots,x_n$ and the perfect $\s$-ideals of the 
$\s$-polynomial ring $k\{x_1,\ldots,x_n\}_\s$ (\cite[Theorem 2.6.4, p. 151]{Levin}).

It is a characteristic feature of difference algebra that one really needs to consider a \emph{family} of $\s$-overfields of $k$, i.e., in general one can not find one big $\s$-overfield of $k$ containing ``enough'' solutions. However, if we assume that $k$ itself is ``large enough'', we can discard the universal family and we arrive at a setting analogous to \cite[Chapter I]{hartshorne}. This is what we mean with the semi-classical point of view, it is usually adopted by model theorists. See \cite{Macintyre:GenricAutomorphismsOfFields},\cite{Hrushovskietal:ModelTheoryofDifferencefields},\cite{ChatHrusPet}. We shall now outline very briefly the semi-classical set-up. The results below are used in Section \ref{sec:sseparability}.

We start by recalling the precise meaning of ``large enough'':
\begin{defi}
A $\s$-field $k$ is called \emph{$\s$-closed} if for every finitely $\s$-generated $k$-$\s$-algebra $R$ which is a $\s$-domain, there exists a morphism $R\to k$ of $k$-$\s$-algebras.
\end{defi}
In other words, a $\s$-field $k$ is $\s$-closed if and only if every system of algebraic difference equations over $k$, which has a solution in a $\s$-field extension of $k$, already has a solution in $k$. The $\s$-closed $\s$-fields are also called models of ACFA. One has to exercise some caution: If $R$ is a finitely $\s$-generated $k$-$\s$-algebra over a $\s$-closed $\s$-field $k$, there need not exist a morphism $R\to k$. In fact, we have the following:

\begin{lemma} \label{lemma:explainsclosed}
Let $k$ be a $\s$-closed $\s$-field and $R$ a finitely $\s$-generated $k$-$\s$-algebra. The following statements are equivalent:
\begin{enumerate}
 \item There exists a morphism $R\to k$ of $k$-$\s$-algebras.
\item There exists a $\s$-prime ideal in $R$.
\end{enumerate}
\end{lemma}
\begin{proof}
The implication (i)$\Rightarrow$(ii) is clear since the kernel of a $k$-$\s$-morphism $R\to k$ is a $\s$-prime ideal.
Conversely, if $\q$ is a $\s$-prime ideal of $R$, then $R/\q$ is a $\s$-domain, and since $k$ is $\s$-closed there exists a $k$-$\s$-morphism
$R/\q\to k$ which we can compose with $R\to R/\q$ to obtain a $k$-$\s$-morphism $R\to k$.
\end{proof}

A maximal $\s$-ideal, i.e., a maximal element in the set of all proper $\s$-ideals ordered by inclusion, need not be prime. By a \emph{maximal $\s$-prime ideal}, we mean a $\s$-prime ideal which is maximal in the set of all $\s$-prime ideals ordered by inclusion. If $R$ is a finitely $\s$-generated $k$-$\s$-algebra over a $\s$-closed $\s$-field $k$ and $\q\subset R$ a maximal $\s$-prime ideal, then $R/\q=k$. So, in this case, a maximal $\s$-prime ideal is maximal as an ideal.

The following lemma is surely well-known. For lack of a suitable reference we include a proof.
\begin{lemma} \label{lemma:sigmaClosedIntersection}
Let $k$ be a $\s$-closed $\s$-field, $R$ a finitely $\s$-generated $k$-$\s$-algebra and $\ida$ a perfect $\s$-ideal of $R$. Then $\ida$ is the intersection of all maximal $\s$-prime ideals of $R$ containing $\ida$.
\end{lemma}
\begin{proof}
Let $f\in R$ such that $f$ is contained in every maximal $\s$-prime ideal containing $\ida$. We have to show that $f\in\ida$. Suppose $f\notin\ida$. Let $g$ denote the image of $f$ in $R/\ida$. Since $\ida$ is perfect and $f\notin\ida$, the multiplicatively closed subset $S$ of $R/\ida$ generated by $g,\s(g),\ldots$ does not contain zero. So the localization $R':=S^{-1}(R/\ida)$ is not the zero ring. Note that $R'$ is naturally a $\s$-ring. Moreover, $R'=(R/\ida)\{\frac{1}{g}\}_\s$ is finitely $\s$-generated over $k$. It is easy to see that the zero ideal of $R'$ is perfect. Because every perfect $\s$-ideal is the intersection of $\s$-prime ideals (\cite[p. 88]{Cohn:difference}), this implies that there exists a $\s$-prime ideal in $R'$. Since $k$ is $\s$-closed we deduce the existence of a $k$-$\s$-morphism $R'\to k$ from Lemma \ref{lemma:explainsclosed}. Composing with the canonical map $R\to R'$ this yields a $k$-$\s$-morphism $\psi\colon R\to k$. By construction, the kernel of $\psi$ is a maximal $\s$-prime ideal of $R$ not containing $f$; a contradiction.
\end{proof}
Let $k$ be a $\s$-closed $\s$-field. A subset of $k^n$ is called \emph{$\s$-closed}\footnote{It is not hard to see that this actually defines a topology on $k^n$.} if it is of the form
\[\V_k(F):=\{a\in k^n|\ p(a)=0\ \forall\ p\in F\}\]
for some subset $F$ of $k\{x_1,\ldots,x_n\}_\s$. Equivalently, a subset of $k^n$ is $\s$-closed if it is of the form $X(k)$ for some $\s$-closed $\s$-subscheme $X$ of $\A^n_k$. (Cf. Example \ref{ex sigmaaffinespace}.) The $\s$-closed subsets of $k^n$ (for some $n$) are sometimes also called \emph{$\s$-varieties}.

If $\ida$ is a perfect $\s$-ideal of $k\{x_1,\ldots,x_n\}_\s$, then we can reinterpret Lemma \ref{lemma:sigmaClosedIntersection} as ``Every $\s$-polynomial that vanishes on $\V_k(\ida)$ must lie in $\ida$.'' It follows that
$\ida\mapsto \V_k(\ida)$ defines a bijection between the $\s$-closed subsets of $k^n$ and the perfect $\s$-ideals of $k\{x_1,\ldots,x_n\}_\s$. This in turn implies the following lemma.

\begin{lemma} \label{lemma:bijectionSigmaClosedandPerfectlySigmaReduced}
Let $k$ be a $\s$-closed $\s$-field. The assignment $X\mapsto X(k)$ defines a bijection between the perfectly $\s$-reduced $\s$-closed $\s$-subschemes of $\A^n_k$ and the $\s$-closed subsets of $k^n$. \qed
\end{lemma}

Let $X\subset k^n$ and $Y\subset k^m$ be $\s$-closed. If one defines a morphism $f\colon X\to Y$ to be a mapping given by $\s$-polynomials then one finds easily that the category of $\s$-varieties is equivalent to the category of perfectly $\s$-reduced, $\s$-algebraic $k$-$\s$-schemes.

\subsection{The $k$-$\sigma$-scheme associated with a $k$-scheme} \label{subsec:ksigmaSchemeAssociatedwithkScheme}

There is a natural way to associate a $k$-$\s$-scheme to a scheme over $k$ which formalizes the fact that solutions of a system of algebraic equations can be interpreted as solutions of difference equations, i.e., algebraic equations are difference equations.
Cf. \cite[Section 3.2, p. 23]{Hrushovski:elementarytheoryoffrobenius} and \cite[Section 3.2, p. 25]{Kamensky:TannakianFormalismOverFieldsWithOperators}. We first treat the algebraic point of view.

\bigskip

Let $R$ be a $k$-algebra and $d\geq 0$. We set ${}^{\s^d}R=R\otimes_k k$, where the tensor product is formed by using $\s^d\colon k\to k$ on the right hand side. We consider ${}^{\s^d}R$ as $k$-algebra via the right factor. So if $R=k[x]/\ida$ then ${}^{\s^d}R$ equals $k[x]/\ida'$ where $\ida'\subset k[x]=k[x_1,\ldots,x_n]$ denotes the ideal generated by the polynomials obtained from polynomials from $\ida$ by applying $\s^d$ to the coefficients. Thus if $a\in k^n$ is a $k$-rational point of $R$ then $\s^d(a)\in k^n$ is a $k$-rational point of ${}^{\s^d}R$.

There is a natural map $\psi_d$ from ${}^{\s^d}R=R\otimes_k k$ to ${}^{\s^{d+1}}R=R\otimes_k k$ given by $\psi_d(r\otimes\lambda)=r\otimes\sigma(\lambda)$. We set \[R_d=R\otimes_k{}^{\s}R\otimes_k \cdots\otimes_k{}^{\s^d}R.\]
We have natural inclusions
$R_d\hookrightarrow R_{d+1}$ of $k$-algebras and ring morphisms $\s_d\colon R_d\to R_{d+1}$ defined by
$\sigma_d(r_0\otimes\cdots\otimes r_d)=1\otimes\psi_0(r_0)\otimes\cdots\otimes\psi_d(r_d)$. The $\s_d$'s are not morphisms of $k$-algebras but make the diagram
\[
\xymatrix{
R_d\ar[r]^-{\s_d} & R_{d+1} \\
k \ar[u]\ar[r]^\s & \ar[u] k
}
\]
commutative. Now we can define $[\s]_kR$ as the limit (i.e., the union) of the $R_d$'s ($d\geq 0$). Taking the limit of the $\s_d$'s yields a morphism $\s\colon [\s]_kR\to [\s]_kR$, turning $[\s]_kR$ into a $k$-$\s$-algebra.
The inclusion $R=R_0\hookrightarrow [\s]_kR$ is characterized by the following universal property.

\begin{lemma}
Let $R$ be a $k$-algebra. There exists a $k$-$\s$-algebra $[\s]_kR$ together with a morphism $\psi\colon R\to [\s]_kR$ of $k$-algebras satisfying the following universal property: For every $k$-$\s$-algebra $S$ and every morphism $\psi'\colon R\to S$ of $k$-algebras there exists a unique morphism $\varphi\colon [\s]_kR\to S$ of $k$-$\s$-algebras making
\[
\xymatrix{
R \ar[rr]^-\psi \ar[rd]_{\psi'} & & [\s]_kR \ar@{..>}[ld]^\varphi \\
& S &
}
\]
commutative.
\end{lemma}
\begin{proof}
This follows immediately from the universal property of the tensor product and the limit.
\end{proof}

Note that if $R=k[x]$, the polynomial ring in the variables $x=(x_1,\ldots,x_n)$, then $[\s]_kR=k\{x\}_\s$, the $\s$-polynomial ring in the $\s$-variables $x=(x_1,\ldots,x_n)$, and 
the inclusion map $R\to [\s]_kR$ is simply saying that a polynomial is a $\s$-polynomial. Moreover, $R_d\subset k\{x\}_\s$ is the $k$-subalgebra of
 all $\s$-polynomials of order at most $d$.

More generally, if $R=k[x]/\ida$ then $[\s]_kR=k\{x\}_\s/[\ida]$. In particular, if $R$ is finitely generated as a $k$-algebra, then $[\s]_kR$ is finitely generated as a $k$-$\s$-algebra.

Alternatively $R_d$ can be described as the $k$-subalgebra of $[\s]_kR$ generated by all elements of the form $\s^i(r)$ for $i\leq d$ and $r\in R$.

If $\varphi\colon R\to R'$ is a morphism of $k$-algebras, then so is $R\to R'\to [\s]_kR'$, and from the universal property we obtain a morphism $[\s]_k(\varphi)\colon [\s]_k R\to [\s]_k R'$. We thus obtain a functor $[\s]_k$ from the category of $k$-algebras to the category of $k$-$\s$-algebras.

If $S$ is a $k$-$\s$-algebra, we denote by $S^\sharp$ the underlying $k$-algebra of $S$. I.e., $(-)^\sharp$ is the forgetful functor from $k$-$\s$-algebras to $k$-algebras that forgets $\s$.
For every $k$-algebra $R$ and every $k$-$\s$-algebra $S$, we have
\[\Alg_k^\s([\s]_kR,S)\simeq \Alg_k(R,S^\sharp).\]
In other words, $[\s]_k$ is left adjoint to $(-)^\sharp$.

\bigskip

We now return to schemes. If $V=\spec(k[V])$ is a scheme over $k$, we can define a functor $[\s]_kV$ from the category of $k$-$\s$-algebras to the category of sets by setting
\[([\s]_kV)(S)=V(S^\sharp)\]
for every $k$-$\s$-algebra $S$. Then $[\s]_kV$ is a $k$-$\s$-scheme. Indeed, as
\[([\s]_kV)(S)=V(S^\sharp)=\Alg_k(k[V],S^\sharp)=\Alg_k^\s([\s]_kk[V],S)\]
for every $k$-$\s$-algebra $S$, we find that $[\s]_kV$ is represented by $[\s]_kk[V]$, i.e.,
$k\{[\s]_kV\}=[\s]_kk[V]$. If confusion is unlikely, we shall sometimes denote the $k$-$\s$-scheme $[\s]_kV$ associated with $V$ with the same letter $V$. For example, we shall write  $k\{V\}$ instead of $k\{[\s]_kV\}$, $\A^n_k$ instead of $[\s]_k\A^n_k$, as in Example \ref{ex sigmaaffinespace} or $\Gl_{n,k}$ instead of $[\s]_k\Gl_{n,k}$. Note that if $V$ is algebraic over $k$, then $[\s]_kV$ is $\s$-algebraic over $k$.

\bigskip

From a $k$-$\s$-scheme $X$, one can obtain a scheme $X^\sharp$ over $k$ by forgetting the $\s$-structure, i.e., $X^\sharp=\operatorname{Spec}(k\{X\}^\sharp)$ or $X^\sharp=\Alg_k(k\{X\}^\sharp,-)$. This defines a forgetful functor $(-)^\sharp$ from the category of $k$-$\s$-schemes to the category of $k$-schemes. If $V$ is a scheme over $k$ and $X$ a $k$-$\s$-scheme then
\[\Hom(X,[\s]_k V)\simeq\Hom(X^\sharp,V).\]
So, on schemes, $[\s]_k$ is right adjoint to $(-)^\sharp$.

%
%

\subsection{Zariski closures} \label{sec:ZariskiClosures}

We next introduce the Zariski closures of a $\s$-closed $\s$-subscheme of a scheme. Cf. \cite[Section 4.3]{Hrushovski:elementarytheoryoffrobenius}.

Let $V$ be a $k$-scheme. For $d\geq 0$, we set $^{\s^d}V=V\times_{\spec(k)}\spec(k)$, where the morphism on the right hand side is induced from $\s^d\colon k\to k$.
Note that if $V$ descents to $k^\s$, i.e., $V=V'\times_{\spec(k^\s)}\spec(k)$ for some scheme $V'$ over $k^\s$, then $^{\s^d}V=V$. (This, for example, is the case for $V=\A^n_k$ or $V=\Gl_{n,k}$.)

We also set
\[V_d=V\times ^\s V\times\cdots\times ^{\s^d}V.\]
Of course this notation is compatible with the notation from the previous section: If $V=\spec(k[V])$, then $^{\s^d}V=\spec(^{\s^d}k[V])$ and $V_d=\spec(k[V]_d)$.
Since $k\{V\}=[\s]_kk[V]$ is the union of the $k$-subalgebras $k[V]_d$ of $k\{V\}$ we can see $[\s]_kV$ as the projective limit of the $V_d$'s.

By a $\s$-closed $\s$-subscheme $X$ of $V$, we mean a $\s$-closed $\s$-subscheme of $[\s]_kV$. By definition, $X$ is given by a $\s$-ideal $\I(X)$ of $k\{V\}$.

We define $X[d]$ to be the closed subscheme of $V_d$ defined by the ideal $\I(X)\cap k[V]_d\subset k[V]_d$. Clearly $\I(X[d])=\I(X)\cap k[V]_d$ is the largest ideal of $k[V]_d$ such that
$k[V]_d\to k\{V\}\to k\{X\}=k\{V\}/\I(X)$ factors through $k[V]_d\to k[V]_d/\I(X[d])$. The geometric significance of this is subsumed in the following definition.

\begin{defi} \label{defi:Zariskiclosures}
Let $V$ be a scheme over $k$ and $X$ a $\s$-closed $\s$-subscheme of $V$. For $d\geq 0$, the smallest closed subscheme $X[d]$ of $V_d$ such that
$X^\sharp\to V_d$ factors through $X[d]\hookrightarrow V_d$ is called the \emph{$d$-th order Zariski closure\footnote{This is called the $d$-th order weak Zariski closure in \cite[Section 4.3, p. 32]{Hrushovski:elementarytheoryoffrobenius}.} of $X$ inside $V$.} The $0$-th order Zariski closure is also called the \emph{Zariski closure}. If the Zariski closure of $X$ inside $V$ is equal to $V_0=V$ we say that $X$ is \emph{Zariski dense} in $V$.
\end{defi}

The above definition can be subsumed by saying that $X[d]$ is the scheme-theoretic image of $X^\sharp\to V_d$. (Cf. \cite[Exercise II.3.11 (d), p. 92]{hartshorne}.)

We have natural dominant projections $X[d+1]\to X[d]$ and $X^\sharp=\varprojlim X[d]$.

\begin{ex}
Let $V=\A^n_k$ and $X$ the $\s$-closed $\s$-subscheme of $V$ defined by a $\s$-ideal $\I(X)$ of $k\{\A^n_k\}=k\{x\}=k\{x_1,\ldots,x_n\}$. (Cf. Example \ref{ex sigmaaffinespace}.) Then  $k\{X\}=k\{x\}/\I(X)=k\{\overline{x}\}$ and
\[X[d]=\spec(k[\overline{x},\s(\overline{x}),\ldots,\s^d(\overline{x})])\]
for every $d\geq 0$.
\end{ex}

We note that this notion of Zariski closure is more subtle than the naive set theoretic Zariski closure in, say, $k^n$ where $k$ is a $\s$-closed $\s$-field.

\subsection{Some properties related to base extension} \label{subsec:somepropertiesrelatedtobaseextension}
In this section, we study the $\s$-analogs of separable and regular algebras, which, in classical algebraic geometry, correspond to absolutely reduced and absolutely integral schemes. Cf. \cite[Lemma 3.26, p. 25]{Hrushovski:elementarytheoryoffrobenius}. This is used in Section \ref{sec:sseparability}.

Let $k$ be a $\s$-field. Recall (Definition \ref{defi:sigmaseparableSigmaregular}) that a $k$-$\s$-algebra $R$ is called $\s$-separable/perfectly $\s$-separable/$\s$-regular if $R\otimes_k k'$ is $\s$-reduced/perfectly $\s$-reduced/a $\s$-domain for every $\s$-field extension $k'$ of $k$.

\begin{lemma} \label{lemma:sigmaseparablesufficesextension}
Let $k$ be a $\s$-field and $R$ a $k$-$\s$-algebra.
\begin{enumerate}
 \item Let $k'$ be an inversive $\s$-field extension of $k$. Then $R$ is $\s$-separable over $k$ if and only if $R\otimes_k k'$ is $\s$-reduced.
\item Let $k'$ be an inversive algebraically closed $\s$-field extension of $k$. Then $R$ is perfectly $\s$-separable/$\s$-regular over $k$ if and only if $R\otimes_k k'$ is perfectly $\s$-reduced/a $\s$-domain.
\end{enumerate}
\end{lemma}
\begin{proof}
Point (i) follows from \cite[Proposition 1.5.2, p. 17]{Wibmer:thesis}, cf. \cite[Lemma 3.26 (2), p. 25]{Hrushovski:elementarytheoryoffrobenius}.
The case of $\s$-regularity in (ii) follows from (i) and the fact that $R$ is regular over $k$ if $R\otimes_k k'$ is an integral domain. (See e.g. \cite[Corollary 1, Chapter 5, $\S$ 17, No. 5, A.V.143]{Bourbaki:Algebra2}.)

It remains to see that $R$ is perfectly $\s$-separable if $R\otimes_k k'$ is perfectly $\s$-reduced. So let $k''$ be a $\s$-field extension of $k$. We have to show that $R\otimes_k k''$ is perfectly $\s$-reduced. Because $k'$ is algebraically closed there exists a $\s$-field extension $k'''$ of $k$ containing $k'$ and $k''$
(cf. \cite[Theorem 5.1.6, p. 313]{Levin}).
If $R\otimes_k k'''$ is perfectly $\s$-reduced then also $R\otimes_k k''\subset R\otimes_k k'''$ is perfectly $\s$-reduced. Therefore we can assume that $k'\subset k''$.
As $R\otimes_k k''=(R\otimes_k k')\otimes_{k'} k''$ we can reduce to showing that every perfectly $\s$-reduced $k$-$\s$-algebra over an algebraically closed inversive $\s$-field $k$ is perfectly $\s$-separable over $k$. In other words, we may assume that $k=k'$.

Let $\q$ be a $\s$-prime ideal of $R$. Then $R/\q$ is a $\s$-domain and it follows from the case of $\s$-regularity proved above that $(R/\q)\otimes_k k''$ is a $\s$-domain. Consequently $\q\otimes k''$ is a $\s$-prime ideal of $R\otimes_k k''$. Because $R$ is perfectly $\s$-reduced, the intersection of all $\s$-prime ideals of $R$ is the zero
ideal (\cite[Proposition 2.3.4, p. 122]{Levin} or \cite[End of Section 6, Chapter 3, p. 88]{Cohn:difference}). It follows that
\[\bigcap_\q (\q\otimes k'')=(\bigcap_\q\q)\otimes k''=(0)\subset R\otimes_k k'',\]
where the intersection is taken over all $\s$-prime ideals $\q$ of $R$. Thus the zero ideal of $R\otimes_k k''$ is the intersection of $\s$-prime ideals. This shows that $R\otimes_k k''$ is perfectly $\s$-reduced.
\end{proof}

\begin{cor} \label{cor:sigmaseparablesufficesextension}
Let $R$ be a $k$-$\s$-algebra.
\begin{enumerate}
 \item If $k$ is inversive, then $R$ is $\s$-separable over $k$ if and only if $R$ is $\s$-reduced.
\item If $k$ is inversive and algebraically closed, then $R$ is perfectly $\s$-separable/$\s$-regular over $k$ if and only if $R$ is perfectly $\s$-reduced/a $\s$-domain.
\end{enumerate}
\end{cor}
\begin{proof}
This is clear from Lemma \ref{lemma:sigmaseparablesufficesextension}.
\end{proof}

\begin{lemma} \label{lemma:localize}
Let $R$ be a $\s$-ring and $S$ a multiplicatively closed $\s$-stable subset of $R$ consisting of non-zero divisors. If $R$ is $\s$-reduced/perfectly $\s$-reduced/a $\s$-domain, then so is $S^{-1}R$.
\end{lemma}
\begin{proof}
This is a straight forward verification.
\end{proof}

\begin{lemma} \label{lemma:sigmaseparablequotientfield}
Let $k$ be a $\s$-field and $R$ a $k$-$\s$-domain. If $R$ is $\s$-separable/perfectly $\s$-separable/$\s$-regular over $k$, then also the quotientfield of $R$ is $\s$-separable/perfectly $\s$-separable/$\s$-regular over $k$.
\end{lemma}
\begin{proof}
Let $k'$ be a $\s$-field extension of $k$ and $S=R\smallsetminus\{0\}$ the multiplicatively closed subset of non-zero divisors of $R$. Because $R$ is a $\s$-domain, $S$ is stable under $\s$ and so the quotientfield $L$ of $R$ is naturally a $\s$-ring. Since $R\otimes_ k k'$ is $\s$-reduced/perfectly $\s$-reduced/a $\s$-domain it follows from Lemma \ref{lemma:localize} that also $S^{-1}(R\otimes_ k k')$ is $\s$-reduced/perfectly $\s$-reduced/a $\s$-domain. But
$L\otimes_k k'=S^{-1}(R\otimes_ k k')$.
\end{proof}


The property of a $k$-$\s$-scheme to be $\s$-reduced/perfectly $\s$-reduced/$\s$-integral is not stable under base extension. So we need to supplement these definitions. To speak meaningfully about base extensions of $k$-$\s$-schemes we record that:

\begin{rem} \label{rem:productsexistinksschemes}
The category of $k$-$\s$-schemes has products. Indeed, if $X$ and $Y$ are $k$-$\s$-schemes then $X\times Y$ is represented by $k\{X\}\otimes_k k\{Y\}$.
\end{rem}
\begin{proof}
This follows from Remark \ref{rem: yoneda for sigma schemes} and the fact that the tensor product is the coproduct in the category of $k$-$\s$-algebras.
\end{proof}

Let $X$ be a $k$-$\s$-scheme and $k'$ a $\s$-field extension of $k$. Let $Y$ denote the $k$-$\s$-scheme represented by the $k$-$\s$-algebra $k'$. We say that
$X_{k'}:=X\times Y$ is obtained from $X$ via the base extension $k'|k$.

\begin{defi} \label{defi:sigmaschemeabsolutelySigmaReduced}
Let $X$ be a $k$-$\s$-scheme. We say that $X$ is \emph{absolutely $\s$-reduced/perfectly $\s$-reduced/$\s$-integral} if $X_{k'}$ is $\s$-reduced/perfectly $\s$-reduced/$\s$-integral for every $\s$-field extension $k'$ of $k$.
\end{defi}
Thus a $k$-$\s$-scheme $X$ is absolutely $\s$-reduced/perfectly $\s$-reduced/$\s$-integral if and only if $k\{X\}$ is $\s$-separable/perfectly $\s$-separable/$\s$-regular over $k$.
Corollary \ref{cor:sigmaseparablesufficesextension} is reinterpreted as:

\begin{cor} \label{cor:absolutelysufficesextension}
Let $X$ be a $k$-$\s$-scheme.
\begin{enumerate}
 \item If $k$ is inversive, then $X$ is absolutely $\s$-reduced if and only if $X$ is $\s$-reduced.
\item If $k$ is inversive and algebraically closed, then $X$ is absolutely perfectly $\s$-reduced/$\s$-integral if and only if $X$ is perfectly $\s$-reduced/$\s$-integral.
\end{enumerate} \qed
\end{cor}

\subsubsection{More on perfect $\s$-separability}

\begin{lemma} \label{lemma:perfectlysigmaseperableReducetoAlgebraicClosure}
Let $L|K$ be an extension of $\s$-fields and let $K'$ denote the relative algebraic closure of $K$ inside $L$. Assume that $L|K$ is $\s$-separable and separable (as field extension). Then $L|K$ is perfectly $\s$-separable if and only if $K'|K$ is perfectly $\s$-separable.
\end{lemma}
\begin{proof}
Clearly $K'|K$ is perfectly $\s$-separable if $L|K$ is perfectly $\s$-separable.

Assume that $K'|K$ is perfectly $\s$-separable. Let $M$ be a $\s$-field extension of $K$. We have to show that $L\otimes_K M=L\otimes_{K'}(K'\otimes_K M)$ is perfectly $\s$-reduced.
Because $L|K$ is separable also $L|K'$ is separable and since $K'$ is relatively algebraically closed in $L$ we see that the field extension $L|K'$ is regular. This implies that $L\otimes\q$ is a prime ideal of $L\otimes_{K'}(K'\otimes_K M)$ for every prime ideal $\q$ of $K'\otimes_K M$. As $K'\otimes_K M$ is perfectly $\s$-reduced, the zero ideal of $K'\otimes_K M$ is the intersection of $\s$-prime ideals.
It follows that the zero ideal of $L\otimes_K M=L\otimes_{K'}(K'\otimes_K M)$ is the intersection of the prime $\s$-ideals $L\otimes\q$, where $\q$ runs through the $\s$-prime ideals of $K'\otimes_K M$. In particular, the zero ideal of $L\otimes_K M$ is the intersection of prime $\s$-ideals.

Since $L|K$ is $\s$-separable, $\s$ is injective on $L\otimes_K M$. It is now easy to see that $L\otimes_K M$ is perfectly $\s$-reduced: Indeed, let $a\in L\otimes_K M$ and $\alpha_1,\ldots,\alpha_n\geq 0$ such that
$\s^{\alpha_1}(a)\cdots\s^{\alpha_n}(a)=0$. If $\q'$ is a prime $\s$-ideal of $L\otimes_K M$ then $\s^\alpha(a)\in\q'$, where $\alpha$ denotes the maximum of the $\alpha_i$. Therefore $\s^\alpha(a)=0$ and it follows that $a=0$.
\end{proof}

\begin{lemma} \label{lemma:perfectlysigmaseparabelfinite}
Let $L|K$ be a finitely $\s$-generated perfectly $\s$-separable extension of $\s$-fields such that the underlying field extension is algebraic. Then $L|K$ is finite.
\end{lemma}
\begin{proof}
Let $\overline{K}$ denote an algebraic closure of $K$ containing $L$. We can extend $\s\colon L\to L$ to $\s\colon \overline{K}\to\overline{K}$. By assumption there is an $L$-tuple $a$ such that $L=K\langle a\rangle_\s$. Because $L$ is algebraic over $K$, we have $L=K\{a\}_\s$.
By assumption $L\otimes_K\overline{K}$ is perfectly $\s$-reduced. Because $L\otimes_K\overline{K}$ is a finitely $\s$-generated $\overline{K}$-$\s$-algebra, it follows from the $\s$-basis theorem (See \cite[Theorem 2.5.5, p. 143 and Theorem 2.5.11, p. 147]{Levin}.) that the zero ideal of $L\otimes_K\overline{K}$ is the intersection of finitely many $\s$-prime ideals. In particular $L\otimes_K\overline{K}$ has only finitely many minimal prime ideals. This is only possible if $L|K$ is finite.
\end{proof}

\begin{lemma}  \label{lemma:perfectlysigmaseparabelGalois}
Let $L|K$ be an extension of $\s$-fields such that the underlying extension of fields is finite and Galois. Then $L|K$ is perfectly $\s$-separable if and only if every (field) automorphism of $L|K$ commutes with $\s$.
\end{lemma}
\begin{proof}
 Because $L|K$ is Galois, there is a bijection between the prime ideals of $L\otimes_K L$ and the automorphisms of $L|K$: If $\q$ is a prime ideal of $L\otimes_K L$ then
$k(\q)=(L\otimes_K L)/\q$ and the inclusions $\tau_1$ and $\tau_2$ into the first and second factor, respectively, are isomorphisms. So $\tau:=\tau_\q:=\tau_1^{-1}\tau_2$ is an automorphism of $L|K$. Conversely, if $\tau$ is an automorphism of $L|K$, then the kernel $\q=\q_\tau$ of $L\otimes_K L\to L,\ a\otimes b\mapsto a\tau(b)$ is a prime ideal of $L\otimes_K L$. The relation between $\tau$ and $\q$ is determined by $1\otimes a-\tau(a)\otimes 1\in\q$ for every $a\in L$.


Let $\q_1,\ldots,\q_m$ denote the prime ideals of $L\otimes_K L$. They are maximal and minimal and $\q_1\cap\cdots\cap\q_m=(0)$. 
We have a mapping $\q\mapsto\s^{-1}(\q)$ from the prime ideals of $L\otimes_K L$ into the prime ideals of $L\otimes_K L$.

A $\s$-ideal $\ida$ of a $\s$-ring which is the finite intersection of prime ideals, is perfect if and only if the prime ideals minimal above $\ida$ are $\s$-prime ideals. (See \cite[Proposition 2.3.4, p. 122]{Levin} or \cite[End of Section 6, Chapter 3, p. 88]{Cohn:difference}.)
It follows that $L\otimes_K L$ is perfectly $\s$-reduced if and only if $\q\mapsto\s^{-1}(\q)$ is the identity.

Let $\q$ be a prime ideal of $L\otimes_K L$ and $a\in L$. Then $1\otimes a-\tau_{\s^{-1}(\q)}(a)\otimes 1\in\s^{-1}(\q)$ and so
$1\otimes\s(a)-\s(\tau_{\s^{-1}(\q)}(a))\otimes 1\in\q$. Therefore $\tau_\q(\s(a))=\s(\tau_{\s^{-1}(\q)}(a))$ for every $a\in L$. In other words, $\tau_\q\s=\s\tau_{\s^{-1}(\q)}$.
If $\q\mapsto\s^{-1}(\q)$ is the identity, then $\s$ commutes with every automorphism of $L|K$. Conversely, if $\tau_\q\s=\s\tau_\q$ for every prime ideal $\q$, then
$\s\tau_\q=\tau_\q\s=\s\tau_{\s^{-1}(\q)}$ implies $\tau_\q=\tau_{\s^{-1}(\q)}$ and so $\q=\s^{-1}(\q)$.

In summary, we see that $L\otimes_K L$ is perfectly $\s$-reduced if and only if $\s$ commutes with every automorphism of $L|K$.
If $L|K$ is perfectly $\s$-separable then $L\otimes_K L$ is perfectly $\s$-reduced and so $\s$ must commute with every automorphism of $L|K$.

It remains to see that $L|K$ is perfectly $\s$-separable if $\s$ commutes with every automorphism of $L|K$. Let $M$ be an inversive algebraically closed $\s$-field extension of $K$ containing $L$. By assumption, we have $L\otimes_K L=L\oplus\cdots\oplus L$ with $\s$ given by $\s(a_1\oplus\cdots\oplus a_m)=\s(a_1)\oplus\cdots\oplus\s(a_m)$.
Thus
\[M\otimes_K L=M\otimes_L(L\otimes_K L)=M\oplus\cdots\oplus M\]
is perfectly $\s$-reduced. It follows from Lemma \ref{lemma:sigmaseparablesufficesextension} that $L$ is perfectly $\s$-separable over $K$.
\end{proof}

\begin{lemma} \label{lemma:perfectlysigmaseparableMain}
Let $L|K$ be a finitely $\s$-generated $\s$-separable extension of $\s$-fields and let $K'$ denote the relative algebraic closure of $K$ inside $L$. Assume that $L|K$ is separable and that $K'$ is Galois over $K$. Then $L$ is perfectly $\s$-separable over $K$ if and only if $K'|K$ is finite and every automorphism of $K'|K$ commutes with $\s$.
\end{lemma}
\begin{proof}
Assume that $L|K$ is perfectly $\s$-separable. Then $K'|K$ is also perfectly $\s$-separable. An intermediate $\s$-field of a finitely $\s$-generated $\s$-field extension is finitely $\s$-generated (\cite[Theorem 4.4.1, p.292]{Levin}). Therefore $K'|K$ is finitely $\s$-generated. It follows from Lemma \ref{lemma:perfectlysigmaseparabelfinite} that $K'|K$ is finite and from Lemma \ref{lemma:perfectlysigmaseparabelGalois} that every automorphism of $K'|K$ commutes with $\s$.

The reverse direction follows from Lemma \ref{lemma:perfectlysigmaseparabelGalois} and Lemma \ref{lemma:perfectlysigmaseperableReducetoAlgebraicClosure}.
\end{proof}

\subsection{Difference dimension} \label{subsec:sdimension}

Let $k$ be a $\s$-field. We would like to define a notion of dimension for a $k$-$\s$-scheme $X$ which is $\s$-algebraic over $k$, i.e., $k\{X\}$ is finitely $\s$-generated over $k$. If $X$ is $\s$-integral, this is classical: The $\s$-dimension of $X$ (or $k\{X\}$) is the $\s$-transcendence degree of the ``function field'' $\quot(k\{X\})$ over $k$. (See \cite[Section 7.1, p. 394]{Levin}.)

As in \cite[Section 4]{Hrushovski:elementarytheoryoffrobenius} one can generalize this definition by considering

\[\sup\{ \strdeg(k(\q)|k)|\ \q \text{ is a $\s$-prime ideal of } k\{X\}\}.\]

This definition has some drawbacks: Firstly, it does not quite make sense if $k\{X\}$ has no $\s$-prime ideals. Secondly, it is not stable under extension of the base $\s$-field. Our aim here is to introduce a notion of $\s$-dimension which agrees with the classical definition if $k\{X\}$ is $\s$-integral and which is stable under extension of the base $\s$-field. It is well-known that the $\s$-dimension can be computed as 
the leading coefficient of an appropriate dimension polynomial (\cite[Def. 4.2.21, p. 273]{Levin}). Here we follow this idea.

%
%

Let $a=(a_1,\ldots,a_m)$ be a $\s$-generating set for $k\{X\}$ over $k$. The basic idea is to define the $\s$-dimension of $k\{X\}$ over $k$ as the ``growth rate'' of the sequence
\[d_i:=\dim(k[a,\s(a),\ldots,\s^{i}(a)]),\ i=0,1,\ldots \]
where $\dim$ denotes the usual Krull-dimension. There are two difficulties: First we need to make precise what we mean by ``growth rate'', and then we need to show that the definition is independent of the choice of generators $a$.

\begin{prop} \label{prop: sigmadim independent}
Let $k$ be a $\s$-field, $R$ a $k$-$\s$-algebra and $a=(a_1,\ldots,a_m)$ a $\s$-generating set for $R$ over $k$. Then
\[\limsup_{i\to\infty}\left(\dim(k[a,\ldots,\s^{i}(a)])/(i+1)\right)\]
is independent of the choice of $a$.
\end{prop}
\begin{proof}
Let $a'=(a'_1,\ldots,a'_{m'})$ be another $\s$-generating of $R$ over $k$. Then all the components of $a'$ lie in $k[a,\ldots,\s^{j}(a)]$ for some $j\in\N$. It follows that
\[k[a',\ldots,\s^i(a')]\subset k[a,\ldots,\s^{j+i}(a)].\]
We abbreviate $d_i:=\dim(k[a,\ldots,\s^{i}(a)])$ and $d'_i:=\dim(k[a',\ldots,\s^{i}(a')])$. Because of the above inclusion $d'_i\leq d_{j+i}$  for $i\in\N$.
Since $k[a,\ldots,\s^{j+i}(a)]$ can be generated by $mj$ elements over $k[a,\ldots,\s^{i}(a)]$, we have $d_{j+i}\leq d_i+mj$. In summary,
 \[\tfrac{d'_i}{i+1}\leq\tfrac{d_{j+i}}{i+1}\leq\tfrac{d_i}{i+1}+\tfrac{mj}{i+1}.\]
 Because $\lim_{i\to\infty}\tfrac{mj}{i+1}=0$ we obtain
\[\limsup_{i\to\infty}\tfrac{d'_i}{i+1}\leq \limsup_{i\to\infty}\tfrac{d_i}{i+1}.\]
 By symmetry the above values are actually equal.
\end{proof}

Because we want the difference dimension to be an integer we make the following definition.

\begin{defi} \label{defi:sigmadimension}
Let $k$ be a $\s$-field and $R$ a finitely $\s$-generated $k$-$\s$-algebra. We define the $\s$-dimension of $R$ over $k$ as
\[\s\text{-}\dim_k(R)=\left\lfloor\limsup_{i\to\infty}\left(\dim(k[a,\ldots,\s^{i}(a)])/(i+1)\right)\right\rfloor,\]
where $\lfloor x\rfloor$ denotes the largest integer not greater than $x$ and $a=(a_1,\ldots,a_m)$ is a $\s$-generating set of $R$ over $k$. (By Proposition \ref{prop: sigmadim independent} this definition does not depend on the choice of $a$.) If $X$ is a $k$-$\s$-scheme such that $k\{X\}$ is finitely $\s$-generated over $k$, then we set $\sdim_k(X)=\sdim_k(k\{X\})$.
\end{defi}

\begin{lemma} \label{lemma: sdim equals stranscendence degree}
Let $k$ be a $\s$-field and $R$ a $\s$-domain which is finitely $\s$-generated over $k$. Then
\[\sdim_k(R)=\strdeg(\quot(R)|k).\]
\end{lemma}
\begin{proof}
Assume that $a=(a_1,\ldots,a_m)$ $\s$-generates $R$ over $k$. Set $d=\strdeg(\quot(R)|k)$. There exists an integer $e$ such that
\[\dim(k[a,\ldots,\s^i(a)])=d(i+1)+e \text{ for } i\gg 0.\]
See \cite[Def. 4.2.21, p. 273]{Levin}. Therefore
\[\limsup_{i\to\infty}\left(\dim(k[a,\ldots,\s^{i}(a)])/(i+1)\right)=d.\]
\end{proof}

\begin{lemma} \label{lemma: sdim stable under base extension}
Let $k$ be a $\s$-field and $R$ a finitely $\s$-generated $k$-$\s$-algebra. If $k'$ is a $\s$-field extension of $k$ then
\[\sdim_{k'}(R\otimes_kk')=\sdim_k(R).\]
\end{lemma}
\begin{proof}
Assume that $a=(a_1,\ldots,a_m)$ $\s$-generates $R$ over $k$. Then $a$ also $\s$-generates $R\otimes_kk'$ over $k'$. The claim now follows from
the fact that \[\dim(k[a,\ldots,\s^i(a)])=\dim(k[a,\ldots,\s^i(a)]\otimes_k k')=\dim(k'[a,\ldots,\s^i(a)]).\]
\end{proof}

\begin{lemma}
Let $k$ be a $\s$-field and $R$ a finitely generated $k$-algebra. Then
$$\sdim_k([\s]_kR)=\dim(R).$$
\end{lemma}
\begin{proof}
Assume that $a=(a_1,\ldots,a_m)$ generates $R$ over $k$. Then $a$ $\s$-generates $[\s]_kR$ over $k$. Moreover, with the notation of Section \ref{subsec:ksigmaSchemeAssociatedwithkScheme}, for every $i\geq 0$ we have $k[a,\ldots,\s^i(a)]=R_i$.
Therefore $$\dim(k[a,\ldots,\s^i(a)])=\dim R_i=(i+1)\dim(R)$$ and
$\sdim_k([\s]_kR)=\lim_{i\to\infty}\dim(k[a,\ldots,\s^i(a)])/(i+1)=\dim(R)$.
\end{proof}
The geometric interpretation of the above lemma is:
\begin{cor} \label{cor:dimensionconsistent}
Let $k$ be a $\s$-field and $V$ a scheme of finite type over $k$. Then
$$\sdim_k([\s]_kV)=\dim(V).$$ \qed
\end{cor}

\begin{rem}
In the situation of Definition \ref{defi:sigmadimension}, one can show that if $R$ is a $k$-$\s$-Hopf algebra (i.e., the $\s$-coordinate ring of a $\s$-algebraic group, see Section \ref{subsec:groupksschemes} below), then \[\lim_{i\to\infty}\left(\dim(k[a,\ldots,\s^{i}(a)])/(i+1)\right)\]
exists and is an integer. So the floor function and the limes superior are not needed. Since we shall not require this fact we omit the details. 
\end{rem}

\subsection{Group $k$-$\sigma$-schemes} \label{subsec:groupksschemes}

We already noted in Remark \ref{rem:productsexistinksschemes} that the category of $k$-$\s$-schemes has products: If $X$ and $Y$ are $k$-$\s$-schemes then $X\times Y$ is represented by $k\{X\}\otimes_k k\{Y\}$. There also is a terminal object: the functor sending every $k$-$\s$-algebra to a one element set. It is represented by $k$. Therefore the following definition makes sense.

\begin{defi} \label{defi:sigmaalgebraicgroup}
A \emph{group $k$-$\s$-scheme} is a group object in the category of $k$-$\s$-schemes. In other words, a group $k$-$\s$-scheme is a $k$-$\s$-scheme $G$ such that $G(S)$ is equipped with a group structure which is functorial in $S$. A morphism of group $k$-$\s$-schemes $\phi\colon G\to H$ is a morphism of $k$-$\s$-schemes such that $\phi(S)\colon G(S)\to H(S)$ is a morphism of groups for every $k$-$\s$-algebra $S$.

A \emph{$\s$-algebraic group over $k$} is a group $k$-$\s$-scheme that is $\s$-algebraic (over $k$).
\end{defi}

Let $G$ be a group scheme over $k$. Then $[\s]_kG$ is a group $k$-$\s$-scheme. (This is clear from $([\s]_k G)(S)=G(S^\sharp)$.) By a \emph{$\s$-closed subgroup of $G$}, we mean a $\s$-closed subgroup of $[\s]_k G$. As in Section \ref{subsec:ksigmaSchemeAssociatedwithkScheme}, we write $k\{G\}$ instead of $k\{[\s]_kG\}$ for the $\s$-coordinate ring of $G$. So, for example,
the $\s$-coordinate ring of the general linear group $\Gl_{n,k}$ (over $k$) is $k\{\Gl_{n,k}\}=k\{X,\frac{1}{\det(X)}\}_\s$.
Here $X=(x_{ij})_{1\leq i,j\leq n}$ is an $n\times n$ matrix of $\s$-indeterminates over $k$ and $k\{X,\frac{1}{\det(X)}\}_\s$ is obtained from the $\s$-polynomial ring $k\{x_{ij}|\ 1\leq i,j\leq n\}_\s$ by localizing at the multiplicatively closed subset generated by $\det(X),\s(\det(X)),\ldots$.

\begin{exa}
Let $G$ be defined by
$$G(S)=\{g\in\Gl_n(S)|\ g\s(g)^{\mathrm{T}}=\s(g)^{\mathrm{T}}g=I\}$$ for every $k$-$\s$-algebra $S$, where $I$ denotes the identity matrix of size $n$. Then $G$ is a $\s$-closed subgroup of $\Gl_{n,k}$.
\end{exa}

\begin{exa}
A homogeneous, linear $\s$-polynomial $p=a_n\s^n(x)+\cdots+a_1\s(x)+a_0x\in k\{x\}_\s$ defines a $\s$-closed subgroup $G$ of the additive group $\bold{G}_{a,k}$ by
$$G(S)=\{g\in S|\ p(g)=0\}\leq\bold{G}_a(S),$$
for any $k$-$\s$-algebra $S$.
\end{exa}
\begin{exa}
Given $m_0,m_1,\ldots,m_n\in\Z$, we can define a $\s$-closed subgroup $G$ of the multiplicative group $\bold{G}_{m,k}$ by
$$G(S)=\{g\in S^\times|\ g^{m_0}\s(g)^{m_1}\cdots\s^n(g)^{m_n}=1\}\leq\bold{G}_m(S),$$
for any $k$-$\s$-algebra $S$.
\end{exa}

\begin{defi}
A $k$-$\s$-Hopf algebra is a $k$-$\s$-algebra equipped with the structure of a Hopf algebra over $k$ such that the Hopf algebra structure maps are morphisms of difference rings. A morphism of $k$-$\s$-Hopf algebras is a morphism of Hopf algebras over $k$ which is also a morphism of $k$-$\s$-algebras.
\end{defi}

\begin{rem}
The category of group $k$-$\s$-schemes is anti-equivalent to the category of $k$-$\s$-Hopf algebras.
\end{rem}
\begin{proof}
This is all tautology, cf. \cite[Section 1.4]{Waterhouse:IntrotoAffineGroupSchemes}.
\end{proof}

\begin{defi}
Let $G$ be a group $k$-$\s$-scheme. By a \emph{$\s$-closed subgroup $H$ of $G$} (In symbols: $H\leq G$.) we mean a $\s$-closed $\s$-subscheme $H$ of $G$ such that $H(S)$ is a subgroup of $G(S)$ for every $k$-$\s$-algebra $S$. We call $H$ \emph{normal} if $H(S)$ is a normal subgroup of $G(S)$ for every $k$-$\s$-algebra $S$. In symbols: $H\trianglelefteq G$.
\end{defi}

\begin{rem}
$\s$-Closed subgroups correspond to $\s$-Hopf ideals, i.e., Hopf ideals which are difference ideals. Normal $\s$-closed subgroups correspond to normal $\s$-Hopf-ideals, i.e., $\s$-Hopf-ideals which are normal Hopf ideals.
\end{rem}

If $\phi\colon G\to H$ is a morphism of group $k$-$\s$-schemes, then we can define a functor $\ker(\phi)$ from $k$-$\s$-algebras to groups by setting $\ker(\phi)(S)=\ker(\phi(S))$ for every $k$-$\s$-algebra $S$. It follows from Lemma \ref{lemma:inverseimagescheme} that $\ker(\phi)=\phi^{-1}(1_H)$ is a $\s$-closed $\s$-subscheme of $G$. Here $1_H\subset H$ denotes the $\s$-closed $\s$-subscheme of the identity element. Obviously $\ker(\phi)$ is a $\s$-closed normal subgroup of $G$. The question, whether every $\s$-closed normal subgroup of $G$ is the kernel of some morphism $G\to H$ will be answered in Section \ref{subsec:quotients} below.

\vspace{5mm}

Let $G$ be a group scheme over $k$ and let $d\geq 0$. Then $^{\s^d}G$ is a group scheme over $k$ and also $G_d=G\times ^\s G\times\cdots\times ^{\s^d}G$ is a group scheme over $k$. Moreover the natural projection
$([\s]_k G)^\sharp\to G_d$ is a morphism of group schemes over $k$.

\begin{lemma} \label{lemma:neededforSecondFundTheorem}
Let $G$ be a group scheme over $k$, $H$ a $\s$-closed subgroup of $G$ and $d\geq 0$.
\begin{enumerate}
\item Then $H[d]$ is a closed subgroup scheme of $G_d$ and $H^\sharp\to H[d]$ is a morphism of group schemes over $k$.
\item If $N$ is a normal $\s$-closed subgroup of $H$, then $N[d]$ is a normal closed subgroup scheme of $H[d]$.
\end{enumerate}
\end{lemma}
\begin{proof} For (i) we note that the intersection of the Hopf ideal
$\I(H)\subset k\{G\}$ with $k[G]_d$ yields a Hopf ideal $\I(H[d])=k[G]_d\cap \I(H)\subset k[G]_d$ because $k[G]_d$ is a $k$-sub-Hopf algebra of $k\{G\}$.

For (ii) let $\I(H)\subset\I(N)\subset k\{G\}$ denote the ideals of $H$ and $N$ respectively. By assumption, $\overline{\I(N)}$, the image of $\I(N)$ in $k\{H\}=k\{G\}/\I(H)$ is a normal Hopf ideal. The ideal $\I(N[d])\subset k[H[d]]=k[G]_d/\I(H[d])$ of $N[d]$ in $H[d]$ is obtained from $\overline{\I(N)}$ by intersecting with $k[H[d]]$ via
$k[G]_d/\I(H[d])\hookrightarrow k\{G\}/\I(H)$. Thus $\I(N[d])$ is a normal Hopf ideal.
\end{proof}


To illustrate the use of Zariski closures, let us describe the $\s$-closed subgroups of tori.
Let $k$ be a $\s$-field. As usual, we denote by $\bold{G}_m$ the multiplicative group scheme over $k$. We think of $\bold{G}_{m}^n$ as a $\s$-algebraic group over $k$, i.e., $\bold{G}_{m}^n(S)=(S^\times)^n$ for any $k$-$\s$-algebra $S$. The $\s$-coordinate ring of $\Gm^n$ is
$k\{\Gm^n\}= k \{ x_1,\ldots,x_n, \frac{1}{x_1},\ldots,\frac{1}{x_n}  \}_\s$. By a multiplicative function $\psi\in k \{ x_1,\ldots,x_n, \frac{1}{x_1},\ldots,\frac{1}{x_n}  \}_\s$ we mean an element which is of the form
$$\psi(x)=x^{\alpha_0}\s(x^{\alpha_1})\cdots \s^l(x^{\alpha_l})$$
for some $\alpha_i \in\Z^n$  and $l \in \N$. Here $x^\beta:=x_1^{\beta_1}\cdots x_n^{\beta_n}$ for $\beta\in\Z^n$.

\begin{lemma}\label{lemma:classgm}
Let $k$ be a $\s$-field and let $G$ be a $\s$-closed subgroup of $\Gm^n$.
Then there exists a set $\Psi$ of multiplicative functions such that
$$G(S)=\{g\in \Gm^n(S)|\ \psi(g)=1 \text{ for } \psi\in\Psi\}$$
for any $k$-$\s$-algebra $S$.\end{lemma}
\begin{proof}
For $d \in \N$, let $G[d]$ be the $d$-th order Zariski closure of $G$ inside $\bold{G}_m^n$. Then $G[d]$ is an algebraic subgroup of $(\bold{G}_m^n)_d=\bold{G}_m^{n(d+1)}$ (Lemma \ref{lemma:neededforSecondFundTheorem}). The claim now follows from the fact that an algebraic subgroup of $\bold{G}_m^{n(d+1)}$ is the intersection of kernels of characters of $\bold{G}_m^{n(d+1)}$.
\end{proof}

\subsection{Quotients} \label{subsec:quotients}

In the category of groups, the quotient $G/N$ of a group $G$ by a normal subgroup $N$ is characterized by the following universal property: Every morphism of groups $\phi\colon G\to H$ such that $N\subset\ker(\phi)$ factors uniquely through $G\to G/N$. Replacing groups with group $k$-$\s$-schemes, we arrive at the following definition.

\begin{defi} \label{defi:quotients}
Let $G$ be group $k$-$\s$-scheme and $N\trianglelefteq G$ a normal $\s$-closed subgroup. By a \emph{quotient of $G$ modulo $N$}, we mean a morphism
$\pi\colon G\to G/N$ of group $k$-$\s$-schemes with $N\subset\ker(\pi)$ satisfying the following universal property: For every morphism $\phi\colon G\to H$ of group $k$-$\s$-schemes with $N\subset\ker(\phi)$ there exists a unique morphism $\psi\colon G/N\to H$ of group $k$-$\s$-schemes making
\[
\xymatrix{
G \ar[rr]^\pi \ar[rd]_-\phi & & G/N \ar@{..>}[ld]^\psi\\
& H &
}
\]
commutative.
\end{defi}

As usual, if the quotient exists it is unique up to unique isomorphisms. In all generality, the existence of quotients of group schemes is a somewhat delicate issue. Since we are only interested in the affine case and normal closed subgroups everything can be done on the ring side and no heavy geometric machinery is necessary. We will follow the purely Hopf-algebraic approach of M. Takeuchi presented in \cite{Takeuchi:correspondence}.

Below we will use some standard notations from the theory of Hopf algebras: If $R$ is a Hopf algebra over a field $k$ then $\Delta\colon R\to R\otimes R$ denotes the comultiplication and $\varepsilon\colon R\to k$ denotes the counit. The kernel of $\varepsilon$ is denoted with $R^+$, i.e., $R^+$ is the vanishing ideal of the unit element.

\begin{prop}[M. Takeuchi] \label{prop:Takeuchi}
Let $k$ be a field, $R$ a Hopf algebra over $k$ with comultiplication $\Delta$ and $\ida\subset R$ a normal Hopf ideal.
Set $R(\ida)=\{r\in R| \ \Delta(r)-r\otimes 1\in R\otimes_k\ida\}$.
Then $R(\ida)$ is a sub-Hopf algebra of $R$ with $RR(\ida)^+=\ida$, i.e., the ideal of $R$ generated by $R(\ida)^+$ is equal to $\ida$. Moreover, $R(\ida)$ is the only sub-Hopf algebra of $R$ with this property and the inclusion map $\iota\colon R(\ida)\hookrightarrow R$ satisfies the following universal property:
Every morphism $\psi\colon R'\to R$ of $k$-Hopf algebras such that $\psi(R'^+)\subset\ida$ factors uniquely through $\iota$.
\end{prop}
\begin{proof}
The statement that $R(\ida)$ is a sub-Hopf algebra of $R$ is Lemma 4.4 in \cite{Takeuchi:correspondence}. That $RR(\ida)^+=\ida$ is proved in \cite[Theorem 4.3]{Takeuchi:correspondence}.
By \cite[Corollary 3.10]{Takeuchi:correspondence} the map $T\mapsto RT^+$ from sub-Hopf algebras of $R$ to normal Hopf-ideals of $R$ is injective. Thus $R(\ida)$ is the only sub-Hopf algebra of $R$ with $RR(\ida)^+=\ida$. If $\psi\colon R'\to R$ is a morphism of $k$-Hopf algebras such that $\psi(R'^+)\subset\ida$ then $\psi(R')$ is a sub-Hopf algebra of $R$ and
$\psi(R')^+=\psi(R'^+)\subset\ida$. By \cite[Lemma 4.7]{Takeuchi:correspondence} the Hopf-algebra $R(\ida)$ is the greatest sub-Hopf algebra of $R$ such that $R(\ida)^+\subset \ida$. Therefore $\psi(R')\subset R(\ida)$, i.e., $\psi$ factors through $\iota$.
\end{proof}

\begin{thm} \label{theo:existenceofquotients}
Let $G$ be a group $k$-$\s$-scheme and $N\trianglelefteq G$ a normal $\s$-closed subgroup. Then the quotient $\pi\colon G\to G/N$ exists and satisfies $N=\ker(\pi)$.
\end{thm}
\begin{proof}
We know from Proposition \ref{prop:Takeuchi} that $k\{G\}(\I(N))=\{r\in k\{G\}| \ \Delta(r)-r\otimes 1\in k\{G\}\otimes_k\I(N)\}$ is a sub-Hopf algebra of $k\{G\}$. It is also a $k$-$\s$-subalgebra. Let $H$ denote the group $k$-$\s$-scheme with $\s$-coordinate ring $k\{H\}=k\{G\}(\I(N))$. We will show that the morphism $\phi\colon G\to H$ of group $k$-$\s$-schemes corresponding to the inclusion $k\{H\}\subset k\{G\}$ of $k$-$\s$-Hopf algebras is the quotient of $G$ modulo $N$.

So let $\varphi\colon G\to H'$ be a morphism of group $k$-$\s$-schemes with $N\subset\ker(\varphi)$. Since $\ker(\varphi)=\varphi^{-1}(1_{H'})$ is the $\s$-closed $\s$-subscheme of $G$ defined by the ideal of $k\{G\}$ generated by $\varphi^{*}(k\{H'\}^+)$ (Lemma \ref{lemma:inverseimagescheme}) we conceive that the algebraic meaning of $N\subset\ker(\varphi)$ is $\varphi^{*}(k\{H'\}^+)\subset \I(N)$. It follows from Proposition \ref{prop:Takeuchi} that the morphism $\varphi^*\colon k\{H'\}\to k\{G\}$ factors uniquely through $k\{H\}\hookrightarrow k\{G\}$, i.e., $\varphi^*(k\{H'\})\subset k\{H\}$. The induced morphism $k\{H'\}\to k\{H\}$ of $k$-$\s$-Hopf algebras gives rise to the desired morphism $H\to H'$ making

\[
\xymatrix{
G \ar[rr]^\phi \ar[rd]_-\varphi & & H \ar@{..>}[ld]\\
& H' &
}
\]
commutative. Thus $\phi\colon G\to H$ is the quotient of $G$ modulo $N$. The algebraic meaning of $N=\ker(\phi)$ is $k\{G\}k\{H\}^+=\I(N)$. This identity
holds by Proposition \ref{prop:Takeuchi}.
\end{proof}

\begin{cor} \label{cor:quotients}
Let $G$ be a group $k$-$\s$-scheme, $N\trianglelefteq G$ a normal $\s$-closed subgroup and $\phi\colon G\to H$ a morphism of group $k$-$\s$-schemes such that $N=\ker(\phi)$ and
$\phi^*\colon k\{H\}\to k\{G\}$ is injective. Then $\phi$ is the quotient of $G$ modulo $N$, i.e., $H=G/N$.
\end{cor}
\begin{proof}
Identifying $k\{H\}$ with the image of $\phi^*$, we may assume that $k\{H\}\subset k\{G\}$. The assumption $N=\ker(\phi)$ translates to $\I(N)=k\{G\}k\{H\}^+$. But from Proposition \ref{prop:Takeuchi} we know that $k\{G\}(\I(N))$ is the only sub-Hopf algebra with this property. Thus $k\{H\}=k\{G\}(\I(N))=k\{G/N\}$ and $H=G/N$ as desired.
\end{proof}

\bibliography{bibdata}

\def\cprime{$'$}
\begin{thebibliography}{Wib12b}

\bibitem[AM05]{AmanoMasuoka:artiniansimple}
Katsutoshi Amano and Akira Masuoka.
\newblock Picard-{V}essiot extensions of {A}rtinian simple module algebras.
\newblock {\em J. Algebra}, 285(2):743--767, 2005.

\bibitem[AMT09]{AmanoMasuokaTakeuchi:HopfPVtheory}
Katsutoshi Amano, Akira Masuoka, and Mitsuhiro Takeuchi.
\newblock Hopf algebraic approach to {P}icard-{V}essiot theory.
\newblock In {\em Handbook of algebra. {V}ol. 6}, volume~6 of {\em Handb.
  Algebr.}, pages 127--171. Elsevier/North-Holland, Amsterdam, 2009.

\bibitem[AOT]{AntieauOvchinnikovTrushin:PeriodicParameters}
Benjamin Antieau, Alexey Ovchinnikov, and Dmitry Trushin.
\newblock Galois theory of difference equations with periodic parameters.
\newblock arXiv:1009.1159.

\bibitem[Bou90]{Bourbaki:Algebra2}
N.~Bourbaki.
\newblock {\em Algebra. {II}. {C}hapters 4--7}.
\newblock Elements of Mathematics (Berlin). Springer-Verlag, Berlin, 1990.
\newblock Translated from the French by P. M. Cohn and J. Howie.

\bibitem[Bui92]{Buium:DifferntialAlgebraicGroupsOfFiniteDimension}
Alexandru Buium.
\newblock {\em Differential algebraic groups of finite dimension}, volume 1506
  of {\em Lecture Notes in Mathematics}.
\newblock Springer-Verlag, Berlin, 1992.

\bibitem[Cas72]{Cassidy:differentialalgebraicgroups}
P.~J. Cassidy.
\newblock Differential algebraic groups.
\newblock {\em Amer. J. Math.}, 94:891--954, 1972.

\bibitem[CH]{ChatzidakisHrushovski:OnSubgroupsOfSemiabelianVarietiesDefinedByDifferenceEquations}
Zo{\'e} Chatzidakis and Ehud Hrushovski.
\newblock On subgroups of semi-abelian varieties defined by difference
  equations.
\newblock arXiv:1112.0920.

\bibitem[CH99]{Hrushovskietal:ModelTheoryofDifferencefields}
Zo{\'e} Chatzidakis and Ehud Hrushovski.
\newblock Model theory of difference fields.
\newblock {\em Trans. Amer. Math. Soc.}, 351(8):2997--3071, 1999.

\bibitem[Cha97]{Chatdifgr}
Zo{\'e} Chatzidakis.
\newblock Groups definable in {ACFA}.
\newblock In {\em Algebraic model theory ({T}oronto, {ON}, 1996)}, volume 496
  of {\em NATO Adv. Sci. Inst. Ser. C Math. Phys. Sci.}, pages 25--52. Kluwer
  Acad. Publ., Dordrecht, 1997.

\bibitem[CHP02]{ChatHrusPet}
Zo{\'e} Chatzidakis, Ehud Hrushovski, and Ya'acov Peterzil.
\newblock Model theory of difference fields. {II}. {P}eriodic ideals and the
  trichotomy in all characteristics.
\newblock {\em Proc. London Math. Soc. (3)}, 85(2):257--311, 2002.

\bibitem[Coh65]{Cohn:difference}
Richard~M. Cohn.
\newblock {\em Difference algebra}.
\newblock Interscience Publishers John Wiley \& Sons, New York-London-Sydeny,
  1965.

\bibitem[CS06]{cassisinger}
P.J. Cassidy and M.F. Singer.
\newblock Galois theory of parameterized differential equations and linear
  differential algebraic groups.
\newblock In {\em Differential Equations and Quantum Groups}, volume~9 of {\em
  IRMA Lectures in Mathematics and Theoretical Physics}, pages 113--157. 2006.

\bibitem[CS12]{ChenSinger:ResiduesAndTelescopersForBivariateRationalFunctions}
Shaoshi Chen and Michael~F. Singer.
\newblock Residues and telescopers for bivariate rational functions.
\newblock {\em Adv. in Appl. Math.}, 49(2):111--133, 2012.

\bibitem[Del90]{Deligne:categoriestannakien}
P.~Deligne.
\newblock Cat\'egories tannakiennes.
\newblock In {\em The {G}rothendieck {F}estschrift, {V}ol.\ {II}}, volume~87 of
  {\em Progr. Math.}, pages 111--195. Birkh\"auser Boston, Boston, MA, 1990.

\bibitem[DG70]{demazuregabriel}
Michel Demazure and Pierre Gabriel.
\newblock {\em Groupes alg\'ebriques. {T}ome {I}: {G}\'eom\'etrie alg\'ebrique,
  g\'en\'eralit\'es, groupes commutatifs}.
\newblock Masson \& Cie, \'Editeur, Paris, 1970.
\newblock Avec un appendice {{\i}t Corps de classes local} par Michiel
  Hazewinkel.

\bibitem[DGS94]{DGS}
Bernard Dwork, Giovanni Gerotto, and Francis~J. Sullivan.
\newblock {\em An introduction to {$G$}-functions}, volume 133 of {\em Annals
  of Mathematics Studies}.
\newblock Princeton University Press, 1994.

\bibitem[Dre]{Dreyfus:density}
Thomas Dreyfus.
\newblock A density theorem for parameterized differential {G}alois theory.
\newblock arXiv:1203.2904.

\bibitem[DV12]{DiVizio:ApprocheGaloisienneDeLaTranscendanceDifferentielle}
Lucia Di~Vizio.
\newblock Approche galoisienne de la transcendance diff{\'e}rentielle.
\newblock In {\em Transendance et irrationalit{\'e}}, SMF Journ\'ee Annuelle
  [SMF Annual Conference], pages 1--20. Soci\'et\'e Math\'ematique de France,
  2012.

\bibitem[DVHW]{articletwo}
Lucia Di~Vizio, Charlotte Hardouin, and Michael Wibmer.
\newblock Difference algebraic relations among solutions of linear differential
  equation.
\newblock arXiv:1310.1289.

\bibitem[Dyc]{dyckdesc}
T.~Dyckerhoff.
\newblock The inverse problem of differential {G}alois theory over the field
  $\mathbb{R}(z)$.
\newblock arXiv:0802.2897.

\bibitem[Har77]{hartshorne}
Robin Hartshorne.
\newblock {\em Algebraic Geometry}.
\newblock Number~52 in GTM. Springer, 1977.

\bibitem[Hru01]{HrushMan}
Ehud Hrushovski.
\newblock The {M}anin-{M}umford conjecture and the model theory of difference
  fields.
\newblock {\em Ann. Pure Appl. Logic}, 112(1):43--115, 2001.

\bibitem[Hru04]{Hrushovski:elementarytheoryoffrobenius}
Ehud Hrushovski.
\newblock The elementary theory of the {F}robenius automorphisms, 2004.
\newblock arXiv:math/0406514v1, updated version available from
  http://www.ma.huji.ac.il/$\sim$ ehud/.

\bibitem[HS08]{HardouinSinger}
Charlotte Hardouin and Michael~F. Singer.
\newblock Differential {G}alois theory of linear difference equations.
\newblock {\em Mathematische Annalen}, 342(2):333--377, 2008.

\bibitem[Kam12]{Kamensky:TannakianFormalismOverFieldsWithOperators}
Moshe Kamensky.
\newblock Tannakian formalism over fields with operators.
\newblock {\em Int. Math. Res. Not. IMRN}, 2012.
\newblock doi: 10.1093/imrn/rns190.

\bibitem[Ked10]{Kedlaya:padicdifferentialequations}
Kiran~S. Kedlaya.
\newblock {\em {$p$}-adic differential equations}, volume 125 of {\em Cambridge
  Studies in Advanced Mathematics}.
\newblock Cambridge University Press, Cambridge, 2010.

\bibitem[Kol68]{Kolchin:AlgebraicGroupsAndAlgebraicDependence}
E.~R. Kolchin.
\newblock Algebraic groups and algebraic dependence.
\newblock {\em Amer. J. Math.}, 90:1151--1164, 1968.

\bibitem[Kol73]{Kolchin:differentialalgebraandalgebraicgroups}
E.~R. Kolchin.
\newblock {\em Differential algebra and algebraic groups}.
\newblock Academic Press, New York, 1973.
\newblock Pure and Applied Mathematics, Vol. 54.

\bibitem[Kol74]{Kolchin:constrainedExtensions}
E.~R. Kolchin.
\newblock Constrained extensions of differential fields.
\newblock {\em Advances in Math.}, 12:141--170, 1974.

\bibitem[Kol85]{Kolchin:differentialalgebraicgroups}
E.~R. Kolchin.
\newblock {\em Differential algebraic groups}, volume 114 of {\em Pure and
  Applied Mathematics}.
\newblock Academic Press Inc., Orlando, FL, 1985.

\bibitem[Kov03]{Kovacic:differentialgaloistheoryofstronglynormal}
Jerald~J. Kovacic.
\newblock The differential {G}alois theory of strongly normal extensions.
\newblock {\em Trans. Amer. Math. Soc.}, 355(11):4475--4522, 2003.

\bibitem[KP02]{PillayKowalski:ANoteonGroupsDefinableInDifferenceFields}
Piotr Kowalski and A.~Pillay.
\newblock A note on groups definable in difference fields.
\newblock {\em Proc. Amer. Math. Soc.}, 130(1):205--212 (electronic), 2002.

\bibitem[KP07]{KowPillsigma}
Piotr Kowalski and Anand Pillay.
\newblock On algebraic {$\sigma$}-groups.
\newblock {\em Trans. Amer. Math. Soc.}, 359(3):1325--1337 (electronic), 2007.

\bibitem[Lan08]{Landesman:GeneralizedDifferentialGaloisTheory}
Peter Landesman.
\newblock Generalized differential {G}alois theory.
\newblock {\em Trans. Amer. Math. Soc.}, 360(8):4441--4495, 2008.

\bibitem[Lev08]{Levin}
Alexander Levin.
\newblock {\em Difference algebra}, volume~8 of {\em Algebra and Applications}.
\newblock Springer, New York, 2008.

\bibitem[Mac97]{Macintyre:GenricAutomorphismsOfFields}
Angus Macintyre.
\newblock Generic automorphisms of fields.
\newblock {\em Ann. Pure Appl. Logic}, 88(2-3):165--180, 1997.
\newblock Joint AILA-KGS Model Theory Meeting (Florence, 1995).

\bibitem[Mau10]{MauGal}
Andreas Maurischat.
\newblock Galois theory for iterative connections and nonreduced {G}alois
  groups.
\newblock {\em Trans. Amer. Math. Soc.}, 362(10):5411--5453, 2010.

\bibitem[Mil12]{Milne:BasicTheoryOfAffineGroupSchemes}
James~S. Milne.
\newblock Basic theory of affine group schemes, 2012.
\newblock Available at www.jmilne.org/math/.

\bibitem[MS11]{MoosaScanlon:GeneralizedHasseSchmidtVarietiesAndTheirJetSpaces}
Rahim Moosa and Thomas Scanlon.
\newblock Generalized {H}asse-{S}chmidt varieties and their jet spaces.
\newblock {\em Proc. Lond. Math. Soc. (3)}, 103(2):197--234, 2011.

\bibitem[MS12]{MitschiSinger:MonodromyGroupsOfParameterizedLinearDifferentialEquationsWithRegularSingularities}
Claude Mitschi and Michael~F. Singer.
\newblock Monodromy groups of parameterized linear differential equations with
  regular singularities.
\newblock {\em Bull. Lond. Math. Soc.}, 44(5):913--930, 2012.

\bibitem[MS13]{MitchSinger:ProjectiveIsomonodromyAndGaloisGroups}
Claude Mitschi and Michael~F. Singer.
\newblock Projective isomonodromy and {G}alois groups.
\newblock {\em Proc. Amer. Math. Soc.}, 141(2):605--617, 2013.

\bibitem[OW]{OvchinnikovWibmer:SGaloisTheoryOfLinearDifferenceEquations}
Alexey Ovchinnikov and Michael Wibmer.
\newblock $\sigma$-{G}alois theory of linear difference equation.
\newblock In preparation.

\bibitem[Sei56]{seiden}
A.~Seidenberg.
\newblock Contribution to the {P}icard-{V}essiot theory of homogeneous linear
  differential equations.
\newblock {\em Amer. J. Math.}, 78:808--818, 1956.

\bibitem[SV99]{Scanlon:DifferenceAlgebraicSubgroupsOfCommutativeAlgebraicGroups}
Thomas Scanlon and Jos{\'e}~Felipe Voloch.
\newblock Difference algebraic subgroups of commutative algebraic groups over
  finite fields.
\newblock {\em Manuscripta Math.}, 99(3):329--339, 1999.

\bibitem[Tak72]{Takeuchi:correspondence}
Mitsuhiro Takeuchi.
\newblock A correspondence between {H}opf ideals and sub-{H}opf algebras.
\newblock {\em Manuscripta Math.}, 7:251--270, 1972.

\bibitem[Tak89]{Takeuchi:hopfalgebraicapproach}
Mitsuhiro Takeuchi.
\newblock A {H}opf algebraic approach to the {P}icard-{V}essiot theory.
\newblock {\em J. Algebra}, 122(2):481--509, 1989.

\bibitem[vdPS97]{vdPutSingerDifference}
Marius van~der Put and Michael~F. Singer.
\newblock {\em Galois theory of difference equations}.
\newblock Springer-Verlag, Berlin, 1997.

\bibitem[vdPS03]{vdPutSingerDifferential}
Marius van~der Put and Michael~F. Singer.
\newblock {\em Galois theory of linear differential equations}.
\newblock Springer-Verlag, Berlin, 2003.

\bibitem[Wat79]{Waterhouse:IntrotoAffineGroupSchemes}
William~C. Waterhouse.
\newblock {\em Introduction to affine group schemes}, volume~66 of {\em
  Graduate Texts in Mathematics}.
\newblock Springer-Verlag, New York, 1979.

\bibitem[Wat95]{Watson:ATreatiseOnTheTheoryOfBesselFunctions}
G.~N. Watson.
\newblock {\em A treatise on the theory of {B}essel functions}.
\newblock Cambridge Mathematical Library. Cambridge University Press,
  Cambridge, 1995.
\newblock Reprint of the second (1944) edition.

\bibitem[Wib10]{Wibmer:thesis}
Michael Wibmer.
\newblock {\em Geometric difference {G}alois theory}.
\newblock PhD thesis, Heidelberg, 2010.

\bibitem[Wib12a]{Wibmer:Chevalley}
Michael Wibmer.
\newblock A {C}hevalley theorem for difference equations.
\newblock {\em Math. Ann.}, 354(4):1369--1396, 2012.

\bibitem[Wib12b]{wibdesc}
Michael Wibmer.
\newblock Existence of {$\partial$}-parameterized {P}icard-{V}essiot extensions
  over fields with algebraically closed constants.
\newblock {\em J. Algebra}, 361:163--171, 2012.

\bibitem[WW88]{WW}
E.T. Whittaker and G.N. Watson.
\newblock {\em A Course of Modern Analysis}.
\newblock Cambridge University Press, Cambridge, 1988.

\end{thebibliography}

\end{document}